\documentclass{article}
\usepackage[left=3cm,right=3cm,top=2.5cm,bottom=2.5cm]{geometry} % page settings
\usepackage{algorithm}
\usepackage{algpseudocode}
\usepackage{amsmath} % provides many mathematical environments & tools
\usepackage{amsthm}
\usepackage{amsfonts}
\usepackage{amssymb}
\usepackage{bbm}
\usepackage{relsize}
\usepackage{units}
\usepackage{xcolor}
\usepackage[shortlabels, inline]{enumitem}

\usepackage{todonotes}

\usepackage{algorithm}%
\usepackage{algorithmicx}%
\usepackage{algpseudocode}%\\

\usepackage{graphicx} %graphicspath

\usepackage{subcaption} %subfigure
\usepackage{adjustbox} % To adjust the alignment

\usepackage{pgfplots}
\pgfplotsset{compat=newest} % Ensure compatibility

\usepackage{hyperref}
\usepackage[
    capitalise,
    noabbrev,
  ]{cleveref}

\usepackage{xparse}%\DeclareDocumentCommand
\newtheorem{theorem}{Theorem}[section]
\newtheorem{example}[theorem]{Example}
\newtheorem{proposition}[theorem]{Proposition}
\newtheorem{remark}[theorem]{Remark}

\newtheorem{definition}[theorem]{Definition}
\newtheorem{lemma}[theorem]{Lemma}

\newcommand{\N}{\mathbb{N}}

\newcommand{\R}{\mathbb{R}}
\newcommand{\mcG}{\mathcal{G}}

\newcommand{\mcN}{\mathcal{N}}
\newcommand{\mcH}{\mathcal{H}^M}
\newcommand{\mcHc}[2]{\mathcal{H}_{#1}^{#2}}
\newcommand{\mcE}{\mathcal{E}}
\newcommand{\wmcH}{\widetilde{\mathcal{H}}}

\newcommand{\TV}{\operatorname{TV}}
\newcommand{\dx}{\text{d}x}

\DeclareMathOperator*{\argmin}{arg\,min}
\DeclareMathOperator*{\Div}{div}
\newcommand{\tf}{TensorFlow}
\DeclareMathOperator*{\Glim}{\Gamma\mbox{-}\lim}
\newcommand{\FD}{\text{FD}}
\newcommand{\radius}{R}

\DeclareDocumentCommand{\InputData}{mO{}}{
  \def\DataPrefix{#2}
  \input{#1}}
  
 \usepackage{fancyhdr}
\usepackage{xparse}
\usepackage{soul}
\usepackage{xcolor}

\makeatletter
\NewDocumentCommand{\nnewsout}{O{red}O{black}+m}
    {%
        \begingroup
        \setulcolor{#1}%
        \setul{-.5ex}{.4pt}%
        \def\SOUL@uleverysyllable{%
            \rlap{%
                \color{#2}\the\SOUL@syllable
                \SOUL@setkern\SOUL@charkern}%
            \SOUL@ulunderline{%
                \phantom{\the\SOUL@syllable}}%
        }%
        \ul{#3}%
        \endgroup
    }
\makeatother

\title{DeepTV: A neural network approach for total variation minimization}

\author{Andreas Langer\thanks{Corresponding author}\hspace{0.15cm}$^{, }$\thanks{Centre for Mathematical Sciences, Lund University, Box 118, 221 00 Lund, Sweden, \url{andreas.langer@math.lth.se}}\hspace{0.15cm} \and Sara Behnamian\thanks{Globe Institute, University of Copenhagen, Oester Voldgade 5-7, 1350 Copenhagen K, Denmark, \url{sara.behnamian@sund.ku.dk}}\hspace{0.15cm}$^{, }$\thanks{Department of Biology, Lund University, Box 118, 221 00 Lund, Sweden}}
\date{}

\begin{document}
\maketitle

\begin{abstract}
Neural network approaches have been demonstrated to work quite well to solve partial differential equations in practice. In this context approaches like physics-informed neural networks and the Deep Ritz method have become popular. In this paper, we propose a similar approach to solve an infinite-dimensional total variation minimization problem using neural networks. We illustrate that the resulting neural network problem does not have a solution in general. To circumvent this theoretic issue, we consider an auxiliary neural network problem, which indeed has a solution, and show that it converges in the sense of $\Gamma$-convergence to the original problem. For computing a numerical solution we further propose a discrete version of the auxiliary neural network problem and again show its $\Gamma$-convergence to the original infinite-dimensional problem. In particular, the $\Gamma$-convergence proof suggests a particular discretization of the total variation. Moreover, we connect the discrete neural network problem to a finite difference discretization of the infinite-dimensional total variation minimization problem. Numerical experiments are presented supporting our theoretical findings.
\end{abstract}
\vspace{0.5cm}
\noindent\textbf{Keywords:} Physics-Informed Neural Networks, Deep Ritz Method, Deep Neural Networks, Unsupervised Learning, Total Variation Minimization, $\Gamma$-Convergence, Image Reconstruction, Deep learning, DeepTV

\vspace{0.5cm}
\noindent \textbf{Mathematics Subject Classification:} 65K10, 46N10, 68U10

\section{Introduction}
Total variation (TV) minimization is a well-established technique in image processing, as it is known to preserve discontinuities in the solution. In this context, one typically minimizes a functional consisting of a data term, which enforces the consistency between the observation and the solution, and the total variation, as a regularization term. 
 In this paper we consider the so-called $L^1$-$L^2$-TV model first introduced in \cite{HintermullerLanger:13}, which consists of a combined $L^1$/$L^2$ data term. Let $\Omega\subset \R^d$, $d=1,2$, be an open, bounded and simple connected domain with Lipschitz boundary, then the $L^1$-$L^2$-TV model writes as
\begin{equation}\label{Eq:ProblemTV}
\inf_{u\in BV(\Omega)\cap L^2(\Omega)} \left\{ E(u):= \alpha_1 \|T u - g \|_{L^1(\Omega)} + \alpha_2 \|T u - g \|_{L^2(\Omega)}^2 + \lambda\TV(u)\right\},
\end{equation} 
where $T:L^2(\Omega) \to L^2(\Omega)$ is a bounded and linear operator, $\alpha_1,\alpha_2,\lambda \geq 0$ are parameters, $g\in L^2(\Omega)$ denotes the given data and $\TV(u)$ is the total variation of $u$ in $\Omega$.
We recall, that for $u\in L^1(\Omega)$
\begin{equation}\label{Eq:TV}
\TV(u) := \sup\left\{\int_\Omega u \Div q\ \dx \colon q\in C_c^1(\Omega)^d, |q(x)|_{2} \leq 1 \text{ for allmost all } x\in\Omega\right\},
\end{equation}
where $C_c^1(\Omega)$ denotes the space of continuously differentiable functions with compact support in $\Omega$ and $|\cdot|_2$ is the standard Euclidean vector norm ($\ell^2$-norm). We note that if $u\in W^{1,1}(\Omega)$, then $\TV(u) = \int_\Omega |\nabla u|_2 \dx$ \cite[Section 10.1]{AtBuMi:14}. Further we define $BV(\Omega):=\{u\in L^1(\Omega)\colon \TV(u) < \infty\}$ the space of functions with bounded (total) variation. The space $BV(\Omega)$ equipped with the norm $\|u\|_{BV(\Omega)} :=\|u\|_{L^1(\Omega)} + \TV(u)$ is a Banach space \cite{Giusti:84}.

It is easy to see that the $L^1$-$L^2$-TV model is a generalization or fusion of two other well-known models in image processing. More precisely, if we set $\alpha_1=0$ in \eqref{Eq:ProblemTV} then we obtain the $L^2$-TV model, which is well suited to the task of removing additive Gaussian noise in images, see e.g.\ \cite{ChCaCrNoPo:10}, while for $\alpha_2=0$ we obtain the $L^1$-TV model, which has been introduced in \cite{Alliney:97, Nikolova:02, Nikolova:04} for eliminating impulse noise, as e.g.\ salt-and-pepper noise. 
While these two special cases work well with either Gaussian noise or impulse noise only, it has been demonstrated in \cite{HintermullerLanger:13,Langer:17a,Langer:19} that the $L^1$-$L^2$-TV model is superior to these two models in removing simultaneously Gaussian and impulse noise. 
It is worth mentioning that several modifications of the $L^1$-$L^2$-TV model have been introduced in the literature, see e.g. \cite{GoShTo:14,HiLaAl:23,LiShYuWa:15}.

In order to numerically compute a solution of \eqref{Eq:ProblemTV} a discretisation method is needed. Typically total variation minimization is discretized by a finite difference method.  The reason for this might be that the image information is usually stored pixel-wise. 
Moreover, the functional in \eqref{Eq:ProblemTV} is non-smooth, which prevents the use of standard derivatives and complicates the process of finding its weak formulation, a requirement for applying the finite element method. 
 Nevertheless, attempts for using finite element methods for total variation minimization \cite{Bartels:12} and in particular for the $L^1$-$L^2$-TV model \cite{AlHiLa:24,AlkamperLanger:17,HiLaAl:23} have been made.  We refer the reader to \cite{ChambollePock:21} for an overview of different finite difference and finite element methods for total variation minimization. 

A recently becoming popular discretization technique is based on neural networks. Thereby the solution space is replaced by a space of neural networks rendering the problem finite dimensional. 
This approach is known under the name physics-informed neural networks (PINNs) \cite{LaLiFo:98,RaPeKa:19}, or in connection with optimization problems as the Deep Ritz Method \cite{EYu:18}. 
Both methods are originally introduced for solving partial differential equations, but their approach differs. 
In PINNs the resulting optimization problem is constituted from the residual of the partial differential equation together with its respective boundary and initial values, while the Deep Ritz Method considers the energy functional of a partial differential equation. 
For an overview about PINNs and neural network approaches in general we refer the reader to the review articles \cite{CuDiGiRoRaPi:22,ShZhKa:24} and the references therein. 
To the best of our knowledge, we are not aware of any work which considers problems of the type \eqref{Eq:ProblemTV}, i.e.\ total variation minimization over the space of bounded variation, in the context of neural network discretisation, as PINNs or the Deep Ritz Method. 
Nevertheless, it is worth mentioning, that first attempts to apply PINNs to solve the total variation flow, see e.g.\ \cite{GrDiKoSc:22}, have been presented. 
However, while minimizing the residual of the total variation flow is interesting in its own right, it leads to a different problem than solving \eqref{Eq:ProblemTV}.

In this paper we are interested in solving the $L^1$-$L^2$-TV model via a neural network approach. As with the Deep Ritz method, our approach is based on the following three basic steps:   
\begin{enumerate*}[(i)]
\item select a space of neural networks as the solution space;
\item choose a quadrature rule for approximating the functional;
\item optimize the resulting minimization problem via a suitable algorithm. 
\end{enumerate*}

We choose to replace the solution space $BV(\Omega)\cap L^2(\Omega)$ in \eqref{Eq:ProblemTV} by a set of neural networks $\mcH$, whose activation functions are  ReLUs (rectifier linear units) and $M\in\N$ denotes the number of weights and biases of a neural network in $\mcH$. In the sequel we will call such neural networks ReLU neural networks (ReLU-NNs). While any other type of neural networks might be used in this approach too, we note that ReLU-NNs are theoretically quite well studied, which is useful for our theoretical analysis. The resulting optimization problem reads then
\begin{equation}\label{Eq:ProblemTV:NN}
\inf_{u\in \mcH} \alpha_1 \|T u - g \|_{L^1(\Omega)} + \alpha_2 \|T u - g \|_{L^2(\Omega)}^2 + \lambda\TV(u).
\end{equation}
As the Deep Ritz Method and PINNs, this approach can be interpreted as unsupervised learning. 
However, in a more classical view \eqref{Eq:ProblemTV:NN} is simply an optimization problem with solution space $\mcH$. 
 Unfortunately, it turns out that even if \eqref{Eq:ProblemTV} has a solution in $BV(\Omega)\cap L^2(\Omega)$ this solution is in general not in $\mcH$, see \cref{Example:Counterexample} below. This has also been observed in \cite{AvelinJulin:20}. The reason for this negative result is that the space $BV(\Omega)\cap L^2(\Omega)$ contains functions with jumps and ReLU-NNs cannot represent such functions. 
 In fact ReLU-NNs can ``only'' represent continuous and piecewise affine functions \cite{ArBaMiMu:16}. 
 Despite this negative observation, we suggest an auxiliary neural network model, which is a constrained version of \eqref{Eq:ProblemTV:NN}. 
 That is we add a constraint to \eqref{Eq:ProblemTV:NN} such that a solution in $\mcH$ is indeed guaranteed. 
 We prove that this auxiliary model $\Gamma$-converges to the $L^1$-$L^2$-TV model, which fully justifies our approach. We are aware that $\Gamma$-convergence for a Deep Ritz Method is also studied in \cite{DoMuZe:22}, but in a completely different setting, i.e.\ for a different type of optimization problem, namely for variational problems arising from partial differential equations, and in Sobolev spaces and not in the space of bounded variation. Hence our analysis completely differs from the one in \cite{DoMuZe:22}. However, our $\Gamma$-convergence proof is inspired by the $\Gamma$-convergence analysis in \cite{Langer:11}, where the $\Gamma$-convergence of a discrete total variation to the continuous total variation \eqref{Eq:TV} is presented.

To numerically compute a solution of the auxiliary model, numerical quadrature to approximate the integrals are needed. Note that since $\mcH \subset W^{1,1}(\Omega)$, we have that $\TV(u) = \int_\Omega |\nabla u|_2 \dx$ for any $u\in \mcH$. Therefore, the total variation of $u\in\mcH$ is computed using the gradient of $u$, which must be carefully approximated in a numerical setting, i.e.\ when applying quadrature. It seems reasonable to approximate the gradient by a finite difference approximation. The reason for this is that replacing the integral by a numerical quadrature evaluates the gradient only at a finite number of points. However, this can be problematic for piecewise constant functions, because the gradient may be zero at the specific evaluation points chosen by the quadrature, even though the overall function might have significant variation. This leads to an inaccurate computation of the total variation.
A finite difference approximation of the gradient circumvents this shortcoming.  
We would like to emphasize that despite the described discretizations we still optimize the resulting discrete auxiliary model over the space $\mcH$. 
Further, we are able to show that this discrete auxiliary model again $\Gamma$-converges to \eqref{Eq:ProblemTV}. 
In particular, the $\Gamma$-convergence proof suggests a particular discretization of the gradient and hence total variation which ensures that the discrete solutions approximate the continuous solution as the discretization becomes finer. The same particular discretization of the total variation has been already used in \cite{Langer:11} to show $\Gamma$-convergence and later in \cite{LaiMessi:12}, but not in the context of neural networks. 
Our proof is actually based on the $\Gamma$-convergence proof in \cite{Langer:11} and adjusted to our neural network setting.
However, any convergence result of minimizers of the discrete auxiliary model to a solution of the continuous problem is missing in \cite{Langer:11}, while we present such a result in our setting. 
In \cite{LaiMessi:12} a convergence result of minimizers is presented, however with respect to the $L^2(\Omega)$ topology and under certain regularity assumption on the data $g$, which we do not need. Hence our result is more general and even stated with respect to the $L^1(\Omega)$ topology.

The discrete auxiliary model has interesting connections to existing models. In particular, it is connected to a respective finite difference discretization of \eqref{Eq:ProblemTV}. More precisely, if the respective finite difference discretization of \eqref{Eq:ProblemTV} has a solution then it is equivalent to the finite difference approximation of the solution of the discrete auxiliary model. 
Further, if $\alpha_1=0$ then the discrete auxiliary model looks like the total variation regularized deep image prior (DIP-TV) \cite{LiSuXuKa:19} and if additionally also $\lambda=0$ then it resembles the deep image prior (DIP) \cite{UlVeLe:18}. Note that in \cite{LiSuXuKa:19} an anisotropic discrete total variation is considered, while we consider an isotropic total variation, due to our definition of the continuous total variation. Additionally, we would like to remark that in image reconstruction one is usually interested in an isotropic approximation of the total variation \cite{Condat:17}, which is why we focus only on the isotropic approximation of the total variation. However, our $\Gamma$-convergence proofs can be easily extended to cover also the anisotropic discrete total variation, although then converging to a different continuous total variation, namely to \eqref{Eq:TV} with the $\ell^2$-norm $|\cdot|_2$ replaced by the $\ell^\infty$-norm $|\cdot|_\infty$.
Hence, our analysis connects DIP and DIP-TV, when used with ReLU-NNs, to the continuous $L^1$-$L^2$-TV model.

For minimizing the resulting discrete objective functional we utilize a stochastic gradient descent method, more precisely Adam \cite{KingmaBa:15}. This allows us to compute a numerical solution without considering the dual formulation of \eqref{Eq:ProblemTV:NN} and its discrete approximations. Our implementation of the solution process is based on \tf\ \cite{tensorflow2015}. 
However, this comes with a drawback, namely we observed that the square-root in the $\ell^2$-norm cannot be handled well. 
To overcome this shortcoming we approximate $|\cdot|_2$ and thus the discrete total variation. We suggest three different approximations, one of which is the commonly used Huber-regularization \cite{Huber:73}. For all three presented options we prove that the regularized discrete auxiliary model $\Gamma$-converges to \eqref{Eq:ProblemTV}. In particular, utilizing $\Gamma$-convergence we justify all our proposed models approximating the $L^1$-$L^2$-TV model. 

In order to compute the distance of a function, for example a solution of the auxiliary model, to a solution of \eqref{Eq:ProblemTV} an error estimate might be necessary. Note that error estimates for problems of the type \eqref{Eq:ProblemTV} have already been derived in \cite{AlHiLa:24,Bartels:15a,BartelsMilicevic:20,HintermullerRincon-Camacho:14}. However, all of these error estimates rely on a dual formulation of \eqref{Eq:ProblemTV} and a respective approximate solution of the dual problem. As stated above, our solution process does not need a dual formulation of the problem and hence in our setting we do not have an approximate solution of the dual problem of \eqref{Eq:ProblemTV}. Consequently, the error estimates in \cite{AlHiLa:24,Bartels:15a,BartelsMilicevic:20,HintermullerRincon-Camacho:14} are not directly applicable here. Therefore, we present a new error estimate for the $L^1$-$L^2$-TV model which is solely based on \eqref{Eq:ProblemTV}. 

The rest of the paper is organized as follows: In \cref{Sec:DefinitionNotation}, we review and introduce essential definitions related to neural networks and functional analysis, which are foundational for the subsequent sections. We analyze the neural network problem \eqref{Eq:ProblemTV:NN} in \cref{Sec:Problem} and in particular suggest an auxiliary model which indeed possesses a minimizer in $\mcH$. We show its $\Gamma$-convergence to the $L^1$-$L^2$-TV model and provide an error estimate. Moreover, we present a universal approximation result for $BV$-functions. In \cref{Sec:NumericalRealization} we describe the numerical realization of the proposed model. In particular, we present a numerical realization of the total variation. We connect the fully discretized neural network model to the original model \eqref{Eq:ProblemTV} by utilizing $\Gamma$-convergence. Further we discuss a numerical issue and propose a solution to circumvent it. In \cref{Sec:NumericalExperiments} we present several numerical experiments supporting our theoretical findings. We conclude with a short discussion in \cref{Sec:Conclusion}.

\section{Definitions and notations}\label{Sec:DefinitionNotation}

We recall that the function $\sigma:\R\to\R$ defined by $\sigma(x)=\max\{0,x\}$ is called rectifier linear unit (ReLU) and commonly used in neural network based machine learning.

\begin{definition}[ReLU-NN, depth]
A ReLU-NN is a dense feedforward neural network whose activation functions live only on the neurons of the hidden layers and are ReLUs. Thereby the depth of the ReLU-NN is the number of hidden layers plus one. That is a ReLU-NN with $L\in\N$ hidden layers has depth $L+1$.
\end{definition}

As we are interested to represent functions $u:\Omega \to \R$ with a ReLU-NN, we consider only ReLU-NNs whose input layer consists of $d\in\{1,2\}$ neurons and whose output layer consists of 1 neuron. Thereby the number of neurons in the hidden layers might be arbitrary but fixed such that the number of weights and biases is $M\in\N$. We denote the set of ReLU-NNs with this fixed architecture by $\mcH$ to emphasize the number $M$ of weights and biases, which is in particularly descriptive in a convergence process where the number $M$ might change. 
The collection of all weights and biases of a neural network in $\mcH$ is denoted by $\theta\in \R^M$. 
For a neural network $u_\theta \in \mcH$ we use the index $\theta$ to emphasize that the neural network  depends on the weights and biases $\theta$.%\todo{In sequel, we sometimes write $\mcH^M$ instead of $\mcH$ in order to emphasize the number of weights and biases of the considered neural networks. In particular, this allows to stress that $M$ may not be fixed but vary.}

\begin{definition}[continuous piecewise affine]\label{Def:CPWA}
A function $f:\Omega \to \R$ is called \emph{continuous piecewise affine} if there exists a finite partition of polyhedra that cover $\Omega$ and $f$ is affine on each polyhedron and continuous in $\Omega$.
\end{definition}
Note that the continuity requirement in \cref{Def:CPWA} automatically holds and could be dropped in the definition, since affine functions are continuous, and all polyhedra are closed and they cover $\Omega$. We anyway kept it to make the continuity explicitly clear. From \cref{Def:CPWA} it immediately follows that any arbitrary continuous piecewise affine function is in the Sobolev space $W^{1,\infty}(\Omega)$ and hence also in $W^{1,p}(\Omega)$ for $p\ge 1$.

For a Banach space $V$ we denote its norm by $\|\cdot\|_V$ and its continuous dual space by $V^*$. For the duality pairing between $V^*$ and $V$ we use $\langle\cdot,\cdot\rangle_{V^*,V}$. If $V=L^2(\Omega)$, then we denote by $\langle\cdot,\cdot\rangle$ the standard $L^2(\Omega)$ inner product. For an operator $T$, we represent its adjoint by $T^*$. 
By $|\cdot|_r$, $r\in [1,\infty]$, we denote the vector $\ell^r$-norm and $|\cdot|$ is the absolute value of a value in $\R$. Hence, for $v\in L^p(\Omega)^d$, we define $\|v\|_{L^p(\Omega)} := \left( \int_\Omega |v|_2^p\dx\right)^{1/p}$, $p\in\N$. 
Given that our analysis is situated within the framework of $BV(\Omega)$-functions, it is useful for us to introduce an intermediate convergence in $BV(\Omega)$. It refers to a type of convergence that falls between strong and weak convergence.
\begin{definition}[{intermediate convergence, cf. \cite[Definition 10.1.3]{AtBuMi:14}}]
Given a sequence $(u_n)_n \subset BV(\Omega)$. We say that $u_n$ converges to $u\in BV(\Omega)$ in the sense of \emph{intermediate convergence}, if $u_n \to u$ strongly in $L^1(\Omega)$, i.e.\ $\lim_{n\to\infty}\|u_n - u\|_{L^1(\Omega)} = 0$, and $\TV(u_n) \to \TV(u)$, i.e.\ $\lim_{n\to\infty} \TV(u_n) = \TV(u)$.
\end{definition}
We recall some further frequently used definitions from functional analysis.
\begin{definition}[coercive, lower semicontinuous, subdifferential]
Let $V$ be a Banach space.
\begin{enumerate}[(a)]
\item A functional $F:V \to \overline{\R}:=\R \cup \{+\infty\}$ is \emph{coercive} on $V$, if for $(v_n)_n\subset V$ with $\|v_n\|_V \to \infty$ implies $F(v_n) \to \infty$.
\item A functional $F: V \to \overline{\R}$ is called \emph{(weakly) lower semicontinuous} if for all $u \in V$ we have that $\liminf_{n \to \infty} F(v_n) \ge F(u)$ for any sequence $(v_n)_n \subset V$ (weakly) converging to $u$ as $n \to \infty$.
\item For a convex functional $F : V \to \overline{\R}$, we define the \emph{subdifferential} of $F$ at $v\in V$, as the set valued function $\partial F(v) = \emptyset$ if $F(v)=\infty$, and otherwise as
$$
\partial F(v) = \{v^* \in V^* \ : \ \langle v^*, u-v\rangle_{V^*,V} + F(v) \leq F(u) \ \ \forall u\in V \}.
$$
\end{enumerate}
\end{definition}

\section{Problem description}\label{Sec:Problem}

We are interested in solving the $L^1$-$L^2$-TV model numerically by a neural network approach. More precisely, we propose to minimize the functional in \eqref{Eq:ProblemTV} over the set of ReLU-NNs which leads to the following optimization problem

\begin{equation}\label{Eq:ProblemHw}
\min_{u_\theta \in \mcH} \alpha_1 \|T u_\theta - g \|_{L^1(\Omega)} + \alpha_2 \|T u_\theta - g \|_{L^2(\Omega)}^2 + \lambda \TV(u_\theta).
\end{equation}
Solving \eqref{Eq:ProblemHw} amounts to find $\theta \in \R^M$, where $M\in \N$ is the number of weights and biases, such that $u_\theta$ minimizes \eqref{Eq:ProblemHw}. That is we could rewrite \eqref{Eq:ProblemHw} as
\begin{equation}\label{Eq:ProblemHw:Theta}
\min_{\theta \in \R^M} \alpha_1 \|T u_\theta - g \|_{L^1(\Omega)} + \alpha_2 \|T u_\theta - g \|_{L^2(\Omega)}^2 + \lambda \TV(u_\theta).
\end{equation}

It is well-known that any $L^1$-function can be arbitrary exactly approximated by neural networks with non-polynomial activation functions \cite{LeLiPiSch:93}. 
In the case of ReLU-NNs, we even have the following depth limited universal approximation result.

\begin{theorem}[{\cite[Theorem 2.3]{ArBaMiMu:16}}]\label{Cor:UniversalApprox}
Let $u \in L^1(\Omega)$. Then for any arbitrary $\varepsilon >0$ there is an $u_\varepsilon$ which can be represented by a ReLU-NN with depth
 at most $\lceil \log_2(d+1) \rceil+1$ such that $\|u - u_\varepsilon\|_{L^1(\Omega)} < \varepsilon$.
\end{theorem}

Note that $BV(\Omega) \subset L^1(\Omega)$ and hence \cref{Cor:UniversalApprox} includes any $BV$-function. A similar result holds also for $BV$-functions with respect to the intermediate convergence.

\begin{theorem}\label{Cor:UniversalApprox:BV}
Let $u \in BV(\Omega)$. Then for any arbitrary $\varepsilon >0$ there is an $u_\varepsilon$ which can be represented by a ReLU-NN with depth at most $\lceil \log_2(d+1) \rceil+1$ such that $\|u - u_\varepsilon\|_{L^1}  + |\TV(u) - \TV(u_\epsilon)|< \varepsilon$.
\end{theorem}

\begin{proof}
For $u \in BV(\Omega)$ there exists a sequence $(v_n)_n \subset C^\infty(\Omega)$ such that $\lim_{n\to \infty} \|u - v_n\|_{L^1(\Omega)} =0$ and $\lim_{n\to\infty} \int_\Omega |\nabla v_n|_2 \dx = \TV(u)$ \cite[Theorem 3.9]{AmFuPa:00}. Hence for any $\epsilon>0$ there is an $v_\epsilon\in C^\infty(\Omega)$ such that 
\begin{equation}\label{Proof:Cor:UniversalApprox:BV:Eq1}
\|u - v_\epsilon\|_{L^1(\Omega)} + \left| \TV(u) - \TV(v_\epsilon) \right| < \epsilon/2.
\end{equation}
Further, $v_\epsilon$ can be approximated by a continuous piecewise affine function $u_\epsilon$ on a regular mesh of simplices (intervals for $d=1$, triangles for $d=2$) such that $u_\epsilon$ coincides with $v_\epsilon$ in every vertex of the mesh and using linear interpolation within each simplex.  For sufficiently small mesh size we have 
$$
\|v_\epsilon -  u_\epsilon\|_{L^{1}(\Omega)} + \| \nabla v_\epsilon -  \nabla u_\epsilon\|_{L^{1}(\Omega)} \leq \epsilon/2,
$$
see e.g.\ \cite[Corollary 1.109]{ErnGuermond:04}. Hence, by using the triangle inequality and noting that $\TV(v) = \|\nabla v\|_{L^1(\Omega)}$ for all $v\in W^{1,1}(\Omega)\cup C^\infty(\Omega)$ this yields
\begin{equation}\label{Proof:Cor:UniversalApprox:BV:Eq2}
\|v_\epsilon -  u_\epsilon\|_{L^{1}(\Omega)}  + \left| \TV(v_\epsilon) - \TV(u_\epsilon) \right| < \epsilon/2.
\end{equation}
Combining \eqref{Proof:Cor:UniversalApprox:BV:Eq1} and \eqref{Proof:Cor:UniversalApprox:BV:Eq2} and using again the triangle inequality yields 
\begin{equation*}
\|u -  u_\epsilon\|_{L^{1}(\Omega)}  + \left| \TV(u) - \TV(u_\epsilon) \right| < \epsilon.
\end{equation*}
That is $BV(\Omega)$ is dense in the space of continuous piecewise affine functions with respect to the intermediate convergence. Since any continuous piecewise affine function from $\Omega$ to $\R$ can be represented by a ReLU-NN with depth of at most $\lceil \log_2(d+1) \rceil+1$ \cite[Theorem 2.1]{ArBaMiMu:16}, the assertion follows.
\end{proof}

\subsection{Existence of a minimizer}
Unfortunately we cannot expect that \eqref{Eq:ProblemHw}, respectively \eqref{Eq:ProblemHw:Theta}, admits a minimizer as a simple one-dimensional example borrowed from \cite{AvelinJulin:20} indicates:
\begin{example}[{cf.\ \cite[Example 1.1]{AvelinJulin:20}}]\label{Example:Counterexample}
	Let $\mcH$ be the set of one hidden layer neural networks, where each one hidden layer neural network is defined as 
	$$
	u_{w,a,b}: \R \to \R \qquad u_{w,a,b}(x) = \sum_{i=1}^m a_i \sigma(w_i \cdot x + b_i).
	$$
Here $m\in\N$ is the number of neurons in the hidden layer, $w\in\R^m$ is the vector of weights from the input neuron to the hidden layer, $a\in\R^m$ is the vector of weights from the hidden layer to the output neuron, $b\in\R^m$ is a vector of biases, and $\sigma:\R \to \R$ is the ReLU activation function, i.e.\ $\sigma(x) = \max\{0,x\}$. 
Hence, we have that $M=3m$.
If we try to approximate the BV-function 
	\begin{equation}\label{Eq:Example:g}
	g(x)= \begin{cases}
	1 & \text{if } x\in [0,l_u)\\
	0 & \text{if } x\in (-l_\ell,0)
	\end{cases}
	\end{equation}
on $\Omega=(-l_\ell, l_u)\subset \R$ by solving $\min_{u_{w,a,b}\in \mcH} \|u_{w,a,b} - g\|_{L^2(\Omega)}^2$, i.e.\ we set $T=I$, $\alpha_1=0=\lambda$ and $\alpha_2=1$ in \eqref{Eq:ProblemHw}, then it turns out that this problem does not have a solution. In fact, any function in $\mcH$ is continuous and since $g$ is not continuous we obtain that $\|u_{w,a,b} - g\|_{L^2(\Omega)}^2 > 0$ for all $u_{w,a,b}\in \mcH$. To see that this is not a contradiction to the universal approximation theorem, see \cref{Cor:UniversalApprox,Cor:UniversalApprox:BV}, i.e.\ that any $BV$-function can be approximated by a ReLU-NN arbitrarily good, let us set $m=2$, $a_1=1$ and $a_2=-1$. Further we choose $w_1=\frac{1}{h}=w_2$, $h>0$, and $b_1=1$, $b_2=0$. This yields
	$$
	u_{w,a,b}(x) = \sigma(\frac{1}{h}x + 1) - \sigma(\frac{1}{h}x) = \frac{\sigma(x + h) - \sigma(x)}{h} \xrightarrow{h \to 0} g(x),\qquad x\in(-l_\ell,l_u),
	$$
	and hence by the Dominated Convergence Theorem we get that 
	$$
	\inf_{u_{w,a,b}\in \mcH} \|u_{w,a,b} - g\|_{L^2(\Omega)}^2 = 0.
	$$
	\end{example}
Even if $\lambda>0$ the issue described in \cref{Example:Counterexample} preserves. Note that a solution of \eqref{Eq:ProblemTV} with $T=I$, $\alpha_1,\alpha_2>0$, $\lambda=1$ and $g$ as in \eqref{Eq:Example:g} is given as
\begin{align*}
u(x) = \begin{cases}
c_1 &\text{if } x\in [0,l_u),\\
c_2 &\text{if } x\in (-l_\ell,0),
\end{cases}
\end{align*}
where $c_1,c_2\in\R^+_0$ such that $c_1 \leq c_2$ and $c_1,c_2$ depend on $\alpha_1$ and $\alpha_2$. The analytic computation of the exact solution for this example can be founded in \cref{AppendixA}. Hence, if $\alpha_1$ and $\alpha_2$ are chosen such that $c_1 < c_2$, then for any $M\in\N$ a function in $\mcH$ cannot solve this problem, but only approximates a solution arbitrary close for sufficiently large $M$. From \cref{Example:Counterexample} we learn that the failure of the existence of a solution of \eqref{Eq:ProblemHw} might be due to the unboundedness of the weights $\theta$, rendering the solution space not compact. Thus we slightly modify \eqref{Eq:ProblemHw} to
\begin{equation}\label{Eq:ProblemHw:constrained}
\min_{u_\theta \in \mcHc{c}{M}} 
\alpha_1 \|T u_\theta - g \|_{L^1(\Omega)} + \alpha_2 \|T u_\theta - g \|_{L^2(\Omega)}^2 + \lambda \TV(u_\theta)
\end{equation}
where $\mcHc{c}{M}:=\{u_\theta \in \mcH \colon |\theta|_\infty \leq c\}$ 
with $c\ge 0$ typically to be chosen large such that \eqref{Eq:ProblemHw:constrained} is a close approximation of \eqref{Eq:ProblemHw}. Note that in $\R^M$ all norms are equivalent and hence any arbitrary vector-norm could be chosen instead of $|\cdot|_\infty$. 

\begin{theorem}\label{Prop:ExistenceMinimizer}
The optimization problem in \eqref{Eq:ProblemHw:constrained} possesses a solution $u_\theta \in \mcH$.
\end{theorem}
\begin{proof}
The optimization problem \eqref{Eq:ProblemHw:constrained} can be equivalently written as
\begin{equation}\label{Eq:ProblemHw:Theta2}
\min_{\theta \in \Theta} \alpha_1 \|T u_\theta - g \|_{L^1(\Omega)} + \alpha_2 \|T u_\theta - g \|_{L^2(\Omega)}^2 + \lambda \TV(u_\theta).
\end{equation}
where $\Theta:=\{\theta \in \R^M \colon |\theta|_\infty\leq c\}$. Since $\Theta$ is closed and bounded, it is compact (Heine-Borel Theorem). 
Let us denote the functional in \eqref{Eq:ProblemHw:Theta2} by $J(\theta) := J_c(\theta) + J_d(\theta)$, where $J_c(\theta):= \alpha_1 \|T u_\theta - g \|_{L^1(\Omega)} + \alpha_2 \|T u_\theta - g \|_{L^2(\Omega)}^2 $ and $J_d(\theta) := \lambda\TV(u_\theta)$.  We show now that $J$ is lower semi-continuous. Since the ReLU activation function $\sigma$ is continuous, functions in $\mcH$ are continuous with respect to $\theta$. Hence, for a sequence $(\theta_n)_n\subset \R^M$ with limit $\theta\in\R^M$ for $n\to \infty$ we have $u_{\theta_n} \to u_\theta$ as $\theta_n \to \theta$. Since strong convergence implies weak convergence we also have $u_{\theta_n} \rightharpoonup u_\theta$ as $\theta_n \to \theta$. By the weak lower semi-continuity of the total variation we obtain $\TV(u_\theta) \leq \liminf_{\theta_n \to \theta} \TV(u_{\theta_n})$ and hence $J_d$ is lower semi-continuous. Further, since $J_c$ is continuous, we have that $J$ is lower semi-continuous.

Since $\Theta$ is compact and $J$ is lower semi-continuous, the existence of a minimizer follows by the Extreme Value Theorem.
\end{proof} 
We note that \cref{Prop:ExistenceMinimizer} can be extended to any activation function $\sigma\in C(\R)$, as the proof only requires that the activation functions are continuous. That is if in \eqref{Eq:ProblemHw:constrained} the solution space $\mcH$ is replaced by a set of neural networks with arbitrary continuous activation functions, then \eqref{Eq:ProblemHw:constrained} still possesses a solution in this respective space.

\begin{remark}
The constraint $|\theta|_\infty \leq c$ in \eqref{Eq:ProblemHw:constrained} guarantees that the weights and biases $\theta$ stay bounded from above and below. The same effect can also be achieved by adding the term $\alpha_\theta |\theta|_\infty$ to \eqref{Eq:ProblemHw} yielding
\begin{equation}\label{Eq:ProblemHw:penalized}
\min_{u_\theta \in \mcH} \alpha_1 \|T u_\theta - g \|_{L^1(\Omega)} + \alpha_2 \|T u_\theta - g \|_{L^2(\Omega)}^2 + \lambda \TV(u_\theta) + \alpha_\theta |\theta|_\infty,
\end{equation}
where $\alpha_\theta > 0$ should be typically chosen small such that \eqref{Eq:ProblemHw:penalized} is a close approximation of \eqref{Eq:ProblemHw}.
Then similar as in \cref{Prop:ExistenceMinimizer} we can show the existence of a minimizer for \eqref{Eq:ProblemHw:penalized}.
\end{remark}

We recall that \eqref{Eq:ProblemTV} has a solution in $BV(\Omega)\cap L^2(\Omega)$ if $T$ does not annihilate constant functions, $\alpha_1 + \alpha_2 > 0$ and $\lambda > 0$ \cite[Lemma 3.3]{Langer:17a}. The minimizer is even unique if additionally $\alpha_2 > 0$ and $T$ is injective; see \cite[Theorem 3.3]{Langer:17a}. The assumption that $T$ does not annihilate constant functions is necessary to show that \eqref{Eq:ProblemTV} is coercive in $BV(\Omega)$. 
In \cref{Prop:ExistenceMinimizer} the coercivity condition follows directly from the constraint $|\theta|_\infty\leq c$.
In fact, this implies that the existence of a solution of \eqref{Eq:ProblemHw:constrained} 
is guaranteed for any $\alpha_1,\alpha_2,\lambda\ge 0$ and hence also if $\lambda=0$. 

Note that \cref{Prop:ExistenceMinimizer} guarantees ``only'' the existence of a minimizer, but does not say anything about the uniqueness. In fact, it seems elusive to show uniqueness of a solution of \eqref{Eq:ProblemHw:constrained}, since $u_\theta$ is usually non-convex with respect to $\theta$ rendering \eqref{Eq:ProblemHw:constrained} non-convex. On the contrary \eqref{Eq:ProblemTV} is convex with respect to $u$ and even strictly convex if $\alpha_2 > 0$ and $T$ is injective. This means that changing the solution space in \eqref{Eq:ProblemTV} to $\mcH$ or $\mcHc{c}{M}$ also changes the nature of the problem. Still \eqref{Eq:ProblemTV} and \eqref{Eq:ProblemHw:constrained} are closely related as we will show in the next subsection.

\subsection{$\Gamma$-convergence}
We will now analyze how minimizers of \eqref{Eq:ProblemHw:constrained} behave for $c\to \infty$ by utilizing $\Gamma$-convergence. For this purpose let us recall the definition of $\Gamma$-convergence.
\begin{definition}[$\Gamma$-convergence {\cite[Definition 1.5]{Braides:02}}]\label{Def:GammaConvergence}
Let $V$ be a Banach space and $\mcE,\mcE_n: V \to \overline{\R}$. The sequence of functionals $(\mcE_n)_n$ is said to $\Gamma$-converge to $\mcE$ if for all $u\in V$ we have
\begin{enumerate}[(i)]
\item\label{Def:GammaConvergence:i} ($\liminf$-condition) $\mcE(u) \leq \liminf_{n\to \infty} \mcE_n(u_n)$ for every $(u_n)_n \subset V$ converging to $u$;
\item\label{Def:GammaConvergence:ii} ($\limsup$-condition) $\mcE(u) \geq \limsup_{n\to \infty} \mcE_n(u_n)$ for at least one sequence $(u_n)_n \subset V$ converging to $u$.
\end{enumerate}
\end{definition}
For further reading on the concept of $\Gamma$-convergence we refer to \cite{Braides:02, DalMaso:12}.

\begin{theorem}[$\Gamma$-convergence]\label{Thm:GammaConvergence}
We define $\mcE_n, \mcE: L^1(\Omega) \to \overline{\R}$ as
\begin{equation}\label{Eq:En}
\mcE_n(u):= \begin{cases}
\alpha_1 \|T u - g \|_{L^1(\Omega)} + \alpha_2 \|T u - g \|_{L^2(\Omega)}^2 + \lambda \TV(u) & \text{if }\ u \in \mcHc{c_n}{M_n},\\
+\infty & \text{otherwise},
\end{cases}
\end{equation}
and 
\begin{equation}\label{Eq:E}
\mcE(u):= \begin{cases}
\alpha_1 \|T u - g \|_{L^1(\Omega)} + \alpha_2 \|T u - g \|_{L^2(\Omega)}^2 + \lambda \TV(u) & \text{if }\ u \in BV(\Omega)\cap L^2(\Omega),\\
+\infty & \text{otherwise}.
\end{cases}
\end{equation}
Then there exist sequences $(c_n)_n \subset \R$ and $(M_n)_n\subset \N$ such that the sequence $(\mcE_n)_n$ $\Gamma$-converges to $\mcE$ (shortly $\Glim_{n\to\infty} \mcE_n = \mcE$) with respect to the $L^1(\Omega)$ topology.

Moreover, if $\mcE$ is coercive on $BV(\Omega)$, then we have that 
\begin{equation}\label{Eq:Convergence:Minimizer}
\lim_{n\to\infty} \inf_{u\in L^1(\Omega)} \mcE_n(u) = \min_{u\in L^1(\Omega)} \mcE(u)
\end{equation}
and every limit of a subsequence of $(u_{n})_{n}$ is a minimizer of $\mcE$, whereby $u_n$ is a minimizer of $\mcE_n$, $n\in\N$.
\end{theorem}
\begin{proof}
Note that for any $c\geq 0$ and $M\in\N$ we have $\mcHc{c}{M} \subseteq W^{1,2}(\Omega)$. 
We check the two conditions in \cref{Def:GammaConvergence}:
\begin{enumerate}[(i)]
\item\label{Proof:GammaConvergence:i} By the lower semicontinuity of the total variation in the $L^1(\Omega)$ topology \cite[Remark 3.5]{AmFuPa:00} we have for any $(u_n)_n \subset L^1(\Omega)$ converging to $u\in L^1(\Omega)$ that $\mcE(u) \leq \liminf_{n\to\infty} \mcE(u_n)$ and hence 
$$
\mcE(u) \leq \liminf_{n\to\infty} \mcE_n(u_n),
$$
which proves condition \ref{Def:GammaConvergence:i} of \cref{Def:GammaConvergence}.

\item\label{Proof:GammaConvergence:ii} To show condition \ref{Def:GammaConvergence:ii} of \cref{Def:GammaConvergence} we follow the strategy in \cite[Section 2.7]{Langer:11}. 

If $u\in L^1(\Omega)\setminus (BV(\Omega)\cap L^2(\Omega))$, then clearly $\mcE(u) = +\infty$. Hence, for any sequence $(u_n)_n\subset L^1(\Omega)$ which converges to $u$ we have $\limsup_{n\to\infty} \mcE_n(u_n) \leq \mcE(u)$.

For $u\in BV(\Omega)\cap L^2(\Omega)$ by \cite[Theorem 3.9]{AmFuPa:00} there exists a sequence $(v_n)_n \subset C^\infty(\Omega)$ converging to $u$ in $L^1(\Omega)$ such that
\begin{align}\label{Eq:1}
\lim_{n\to \infty} \int_\Omega |\nabla v_n|_2 \dx = \TV(u).
\end{align}
\begin{figure}[htbp]
\centering
\begin{tikzpicture}[scale=0.3]
% Draw the outer frame
\draw[thick] (8.5,2) rectangle (20.5,14);
% Draw the horizontal lines
\draw[thick] (8.5,10) -- (20.5,10);
\draw[thick] (8.5,6) -- (20.5,6);
% Draw the vertical lines
\draw[thick] (12.5,2) -- (12.5,14);
\draw[thick] (16.5,2) -- (16.5,14);
% Draw the diagonal lines
\draw[thick] (16.5,2) -- (20.5,6);
\draw[thick] (12.5,2) -- (20.5,10);
\draw[thick] (8.5,2) -- (20.5,14);
\draw[thick] (8.5,6) -- (16.5,14);
\draw[thick] (8.5,10) -- (12.5,14);
\end{tikzpicture}
\hspace{1cm}
\begin{tikzpicture}[scale=0.3]
% Draw the outer frame
\draw[thick] (8.5,2) rectangle (20.5,14);
% Draw the horizontal lines
\draw[thick] (8.5,12) -- (20.5,12);
\draw[thick] (8.5,10) -- (20.5,10);
\draw[thick] (8.5,8) -- (20.5,8);
\draw[thick] (8.5,6) -- (20.5,6);
\draw[thick] (8.5,4) -- (20.5,4);
% Draw the vertical lines
\draw[thick] (10.5,2) -- (10.5,14);
\draw[thick] (12.5,2) -- (12.5,14);
\draw[thick] (14.5,2) -- (14.5,14);
\draw[thick] (16.5,2) -- (16.5,14);
\draw[thick] (18.5,2) -- (18.5,14);
% Draw the diagonal lines
\draw[thick] (18.5,2) -- (20.5,4);
\draw[thick] (16.5,2) -- (20.5,6);
\draw[thick] (14.5,2) -- (20.5,8);
\draw[thick] (12.5,2) -- (20.5,10);
\draw[thick] (10.5,2) -- (20.5,12);
\draw[thick] (8.5,2) -- (20.5,14);
\draw[thick] (8.5,4) -- (18.5,14);
\draw[thick] (8.5,6) -- (16.5,14);
\draw[thick] (8.5,8) -- (14.5,14);
\draw[thick] (8.5,10) -- (12.5,14);
\draw[thick] (8.5,12) -- (10.5,14);
\draw[thick] (14.5,6) -- (16.5,8);
\end{tikzpicture}
\caption{Exemplary mesh $\mathcal T_{h_n}$ (left) with nested finer mesh (right).}\label{Fig:Mesh}
\end{figure}
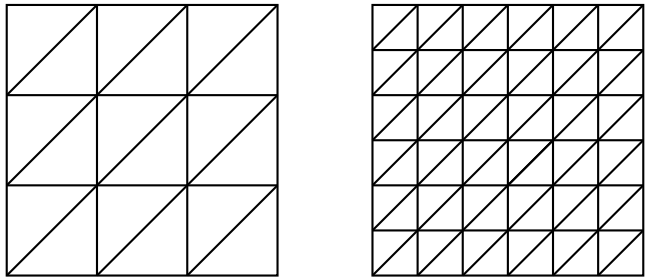
Let $(\mathcal{T}_{h_n})_n$ be a sequence of nested regular meshes of simplices with mesh size $h_n$; an example for $d=2$ (two dimensions) is shown in \cref{Fig:Mesh}. Then for all $v\in C^\infty(\Omega)$, we can construct a continuous piecewise affine function $v_{h_n}$, whereby $v_{h_n}$ coincides with $v$ in every node of $\mathcal{T}_{h_n}$. Note that any continuous piecewise affine function can be represented by a ReLU-NN \cite[Theorem 2.1]{ArBaMiMu:16}. Hence for any $n\in\N$ there exists a $c_n > 0$ and an $M_n\in\N$ such that $v_{h_n}\in \mcHc{c_n}{M_n}$. We choose $c_n$ and $M_n$  such that $c_n \geq c_{n-1}$ and $M_n\geq M_{n-1}$ by possibly adding new neurons to the existing hidden layers (i.e.\ not changing the depth of the neural networks) which yields the nested property $\mcHc{c_{n-1}}{M_{n-1}} \subseteq \mcHc{c_{n}}{M_n}$. We define the operator $\Pi_{h_n}: C^\infty(\Omega) \to \mcHc{c_n}{M_n}$ with $v\mapsto v_{h_n} = \Pi_{h_n} v$. By \cite[Theorem 3.1.5]{Ciarlet:02} or \cite[Corollary 1.109]{ErnGuermond:04} there are $C_1, C_2>0$ such that  
\begin{equation*}
\|v - \Pi_{h_n} v\|_{L^1(\Omega)} \leq C_1 h_n \int_\Omega |\nabla v|_2\dx
\end{equation*}
and
\begin{equation}\label{Eq:2}
\|\nabla v - \nabla \Pi_{h_n} v\|_{L^1(\Omega)} \leq C_2 h_n \int_\Omega |\nabla^2 v|_2\dx.
\end{equation}
Hence for all $v\in C^\infty(\Omega)$ we have
\begin{align*}
\Pi_{h_n} v \to v \quad \text{and}\quad \nabla\Pi_{h_n} v \to \nabla v
\end{align*}
in $L^1(\Omega)$ for $n\to \infty$. In fact, due to \eqref{Eq:1} for all $k\in\N$ there is $n_k> n_{k-1}$ such that for all $n> n_k$ we have
\begin{equation*}
\left |\TV(u) - \|\nabla v_n\|_{L^1(\Omega)}\right| \leq 2^{-k-1}.
\end{equation*}
Then, from \eqref{Eq:2} there exists an $\tilde{n}= \tilde{n}(n)$ such that
\begin{equation*}
\left|\|\nabla v_n\|_{L^1(\Omega)} - \|\nabla \Pi_{h_{\tilde{n}}} v_n\|_{L^1(\Omega)}\right| \leq 2^{-k-1}.
\end{equation*}
The latter two inequalities yield by using the triangle inequality
\begin{equation*}
\left|\TV(u) - \|\nabla \Pi_{h_{\tilde{n}}} v_n\|_{L^1(\Omega)}\right| \leq 2^{-k}.
\end{equation*}
Since $n\geq n_k$ we actually have $\tilde{n}=\tilde{n}(n_k) =: \tilde{n}_k$ and we can construct a sequence 
\begin{equation*}
\Pi_{h_{\tilde{n}_k}}v_{n_k} \in \mcHc{c_{\tilde{n}_k}}{M_{\tilde{n}_k}}
\end{equation*}
such that 
$$
\lim_{k\to\infty}\Pi_{h_{\tilde{n}_k}}v_{n_k} = u \qquad \text{and} \qquad \lim_{n\to\infty} \|\nabla \Pi_{h_{\tilde{n}_k}} v_{n_k}\|_{L^1(\Omega)} = \TV(u).
$$ 
Without loss of generality let $\tilde{n}_{k+1}>\tilde{n}_k$. Then by the nested property of $\mcHc{c_n}{M_n}$ we obtain
\begin{equation*}
\Pi_{h_{\tilde{n}_k}}v_{n_k} \in \mcHc{c_n}{M_n} \quad \text{for all } n\geq \tilde{n}_k.
\end{equation*}
We construct a sequence $u_n \in \mcHc{c_n}{M_n}$ such that
\begin{equation*}
\begin{split}
&u_n = \Pi_{h_{\tilde{n}_1}}v_{n_1} \quad \text{for } n = 1, \ldots,\tilde{n}_{1}-1,\text{ and }\\
&u_n = \Pi_{h_{\tilde{n}_k}}v_{n_k} \quad \text{for } n = \tilde{n}_k, \ldots,\tilde{n}_{k+1}-1, k\geq 1.
\end{split}
\end{equation*}
Then we have that $u_n \to u$ in $L^1(\Omega)$ and $\lim_{n\to\infty}\TV(u_n) \to \TV(u)$, where we used that $\int_\Omega |\nabla v|_2\dx = \TV(v)$ for all $v\in W^{1,1}(\Omega)$ \cite[Section 10.1]{AtBuMi:14} and $\mcHc{c}{M} \subset W^{1,1}(\Omega)$ for all $c\geq0$ and $M\in\N$.
This together with the continuity of $T$ and $u\mapsto \alpha_1 \|T u - g \|_{L^1(\Omega)} + \alpha_2 \|T u - g \|_{L^2(\Omega)}^2$ yields
$$
\lim_{n\to \infty} \mcE_n(u_n) = \mcE(u).
$$
\end{enumerate}
Hence \ref{Proof:GammaConvergence:i} and \ref{Def:GammaConvergence:ii} show the $\Gamma$-convergence of $(\mcE_n)_n$ to $\mcE$.

Since $\mcE$ is coercive on $BV(\Omega)$, there exists a constant $a>0$ such that $S:=\{u\in BV(\Omega) \colon \mcE(u)\leq a\}$ is bounded in $BV(\Omega)$. In particular we choose $a:= \mcE(0) = \mcE_n(0)$, $n\in\N$, which ensures that any minimizer of $(\mcE_n)_n$ and $\mcE$ is in $S$. In fact,  for all $n\in\N$ a minimizer of $\mcE_n$ is ensured by \cref{Prop:ExistenceMinimizer} and since $0\in \mcHc{c}{M}$ for all $c\ge 0$ and $M\in\N$ any minimizer of $\mcE_n$ is in $S$. Further, since $0\in BV(\Omega)$ and the $BV$-coercivity of $\mcE$ implies the existence of a minimizer of $\mcE$ in $BV(\Omega)$ \cite[Theorem 3.3]{Langer:17a}, also the minimizers of $\mcE$ lie in $S$.

The set $S$ is relatively compact in $L^p(\Omega)$ for $1\le p< \frac{d}{d-1}$ \cite[Theorem 2.5]{AcarVogel:94}. Hence $\overline{S}\subset L^1(\Omega)$ is compact such that $\inf_{L^1(\Omega)} \mcE_n = \inf_{\overline{S}} \mcE_n$ for any $n\in\N$. This together with the $\Gamma$-convergence of $(\mcE_n)_n$ to $\mcE$ yields \eqref{Eq:Convergence:Minimizer} due to \cite[Theorem 1.21]{Braides:02}.
Moreover, since any subset of a relatively compact set is relatively compact and the sequence of minimizers $(u_{n})_{n}\subseteq S$, we have that every limit of a subsequence of $(u_{n})_{n}$ is a minimizer of $\mcE$ \cite[Theorem 1.21]{Braides:02}.
\end{proof}

\begin{remark}
\Cref{Thm:GammaConvergence} states that there exist sequences $(c_n)_n$ and $(M_n)_n$, but it does not provide a concrete description how to construct them. From the proof, we deduce that the sequences $(c_n)_n$ and $(M_n)_n$ must be such that $\mcHc{c_n}{M_n}$ is sufficiently large, ensuring that the corresponding continuous piecewise affine approximation is in $\mcHc{c_n}{M_n}$. Additionally we choose $(c_n)_n$ and $(M_n)_n$ monotonically increasing, i.e.\ $c_n \leq c_{n+1}$ and $M_n \leq M_{n+1}$ for all $n\in \N$ which yields by the definition of $\mcHc{c}{M}$ that $\mcHc{c_n}{M_n} \subseteq \mcHc{c_{n+1}}{M_{n+1}}$ (nested property). This allows us to construct a particular sequence $(u_n)_n$. However, alternatively we could have also set $u_k:=\Pi_{h_{\tilde{n}_k}}v_{n_k}$ for $k\in\N$ and choose this as the respective converging sequence. As a consequence the nested property of $\mcHc{c}{M}$ is not needed and hence the requirement that $(c_n)_n$ and $(M_n)_n$ is monotonically increasing can be dropped.
\end{remark} 

\cref{Thm:GammaConvergence} shows a connection between \eqref{Eq:ProblemHw:constrained} and \eqref{Eq:ProblemTV}. In particular, for $c\to \infty$ and $M\to\infty$ we have that minimizers of \eqref{Eq:ProblemHw:constrained} approach a minimizer of \eqref{Eq:ProblemTV}.

\subsection{Error estimate}\label{Sec:ErrorEstimate}
We turn to \eqref{Eq:ProblemTV} to derive an a-posteriori error estimate, which quantifies a possible error between a solution of \eqref{Eq:ProblemTV} and any approximation. Differently to previous work, see \cite{AlHiLa:24,Bartels:15a,BartelsMilicevic:20,HintermullerRincon-Camacho:14}, our error estimate is solely based on the primal formulation of the total variation minimization problem.
 
Let $B:= 2\alpha_2 T^* T$ and define $\|u\|_B^2 := \langle Bu, u \rangle$, then for a minimizer  $u^*$ of \eqref{Eq:ProblemTV} and any arbitrary $v\in L^2(\Omega)$ we have 
 $$
 \frac{1}{2} \|u^* - v\|_{B}^2 \leq E(v) -E(u^*),
 $$
cf. \cite{AlHiLa:24}. We define $F:L^2(\Omega) \to \overline{\R}$ with $F(u) :=  \alpha_2 \| T u - g \|_{L^2(\Omega)}^2 $ and $G:L^2(\Omega) \to \overline{\R}$ with $G(u):= G_1(u) + G_{\TV}(u) := \alpha_1 \|T u - g \|_{L^1(\Omega)} + \lambda\TV(u)$. Then we obtain using the definition of subdifferential
\begin{align*}
\alpha_2 \|T(u^* - v)\|_{L^2(\Omega)}^2 &\leq F(v) -F(u^*) + G(v) -G(u^*)\\
&\leq \langle \xi_F, v - u^* \rangle + \langle \xi_{G}, v - u^* \rangle \\
&\leq \|\xi_F\|_{L^2(\Omega)} \|v - u^*\|_{L^2(\Omega)} + \|\xi_G\|_{L^2(\Omega)} \|v - u^*\|_{L^2(\Omega)}
\end{align*}
where $\xi_F\in \partial F(v)$ and $\xi_G\in \partial G(v)$. Since $F$ is Fr{\'e}chet differentiable we have $\xi_F = 2 \alpha_2 T^*(Tv - g)$ and obtain
\begin{align*}
\alpha_2 \frac{\|T(u^* - v)\|_{L^2(\Omega)}^2}{\|u^* - v\|_{L^2(\Omega)}} \leq 2 \alpha_2\|T^*(Tv - g)\|_{L^2(\Omega)}  + \|\xi_G\|_{L^2(\Omega)}
\end{align*}
where $\xi_G$ can be approximated by $\alpha_1 T^*\frac{Tv-g}{|T v-g| +\epsilon} + \lambda\nabla^* \frac{\nabla v }{|\nabla v|_2+ \epsilon}$ 
with $\epsilon>0$ being a very small constant.

If the bilinear form $\langle Bu, u \rangle$ is coercive, i.e.\ there exists a constant $\tilde{c}>0$ such that $\langle Bu, u \rangle = \alpha_2 \|Tu\|_{L^2(\Omega)}^2 \geq \tilde{c} \|u\|_{L^2(\Omega)}^2$ for any $u \in L^2(\Omega)$, we can further estimate as
\begin{align}\label{Eq:ErrorEstimate}
 \tilde{c} \|u^* - v\|_{L^2(\Omega)} \leq 2 \alpha_2\|T^*(Tv - g)\|_{L^2(\Omega)}  + \|\xi_G\|_{L^2(\Omega)}.
\end{align}
Note that the coercivity condition holds whenever $T$ is invertible, which is, for example, the case for the identity operator $T=I$.

\section{Numerical realization} \label{Sec:NumericalRealization}
In this section we propose a discretization of \eqref{Eq:ProblemHw:constrained} and connect the resulting optimization problem to \eqref{Eq:ProblemTV} again utilizing $\Gamma$-convergence.

\subsection{Numerical integration}
To realize a numerical implementation of \eqref{Eq:ProblemHw:constrained} and \eqref{Eq:ProblemHw:penalized} we need to approximate its integrals. Note that, since the functions in $\mcH$ are continuous, $\TV(u_\theta) = \int_\Omega |\nabla u_\theta(x)|_2 \dx$ for any $u_\theta \in \mcH$. Utilizing an appropriate quadrature \eqref{Eq:ProblemHw:constrained} can be approximated as
\begin{equation}\label{Eq:ProblemHw:constrained:discrete}
\begin{split}
\min_{u_\theta \in \mcHc{c}{M}} \alpha_1 \sum_{x\in \Omega^h} w_1^x |(T u_\theta)(x) - g(x) | &+ \alpha_2 \sum_{x\in\Omega^h} w_2^x|(T u_\theta)(x) - g(x) |^2  \\
&+ \lambda \sum_{x\in\Omega^h} w_{\TV}^x |(\nabla u_\theta)(x)|_2, 
\end{split}
\end{equation}
where $\Omega^h \subset \Omega$ is a finite set of points and $w_q^x$, $q\in \{1,2,\TV\}$, are suitable quadrature weights. The superscript $x$ in $w_q^x$ emphasize that the weights might depend on $x$. Further we define $N:=|\Omega^h|$ the size of $\Omega^h$, or in other words, the number of points in $\Omega^h$.

In image processing, the data points are typically organized in a highly structured manner. These points represent the pixels of the image and are arranged on a regular, equidistant grid. This means that each pixel is uniformly spaced from its neighbors, both horizontally and vertically, forming a systematic grid pattern. Since  we are interested in image processing tasks, in the sequel we assume that the data points can be arranged on a regular and equidistant grid. 
This allows for standard discretizations techniques of the gradient operator, like forward and backward differences. In particular, discretizing the gradient operator as well as a blurring operator on unstructured grids might be tricky. Hence in the sequel we sometimes also write for $d=2$ the number of data points $N=|\Omega^h|$ as $N = N_1 \times N_2$, where $N_1$ is the number of points in the first dimension and $N_2$ is the number of points in the second dimension.

Note that for $d=1$ a none-equidistant discretisation does not usually rise any troubles. However, to keep the presentation simple and concise also for $d=1$ we restrict to equidistant grids.

\subsection{Numerical realization of the total variation}
Considering the step function $g$ given in \eqref{Eq:Example:g}, it is clear that $\nabla g(x)= 0$ for almost every $x \in (-l_\ell, l_u)$, except in 0, where its gradient cannot be computed. Taking a finite number of points $x \in \Omega^h \subset (-l_\ell, l_u)\setminus\{0\}$, means that $\sum_{x\in\Omega^h} w_{\TV}^x |\nabla g(x)|_2 = 0$, although the total variation of $g$ in $(-l_\ell, l_u)$ is 1, i.e.\ $\TV(g)=1$. Hence, it does not seem to be a good idea to take the continuous gradient in \eqref{Eq:ProblemHw:constrained:discrete}. Instead we suggest to solve
\begin{equation}\label{Eq:ProblemHw:constrained:discrete2}
\begin{split}
\min_{u_\theta \in \mcHc{c}{M}} \Big\{E_\theta (u_\theta):=\alpha_1 \sum_{x\in\Omega^h} w_1^x |(T^h u_\theta(x) - g(x) | &+ \alpha_2 \sum_{x\in\Omega^h} w_2^x|(T^h u_\theta(x) - g(x) |^2  \\
&+ \lambda \sum_{x\in\Omega^h} w_{\TV}^x |\nabla^h u_\theta(x)|_{\blacktriangleright} \Big\},
\end{split}
\end{equation}
where $\nabla^h$ denotes a suitable discrete gradient, $|\cdot|_{\blacktriangleright}$ a respective discrete norm and $T^h$ is a suitable discretization of $T$. The choice of 
$\nabla^h$ and $|\cdot|_{\blacktriangleright}$ will become clear later. More precisely, the below $\Gamma$-convergence result from \cref{Sec:Gammaconvergence:discrete} suggests a particular choice of $\nabla^h$ and $|\cdot|_{\blacktriangleright}$, which are associated to each other.

Note that \eqref{Eq:ProblemHw:constrained:discrete2} looks like DIP-TV \cite{LiSuXuKa:19} if $|\cdot|_{\blacktriangleright} = |\cdot|_1$ and DIP \cite{UlVeLe:18} if $\lambda=0$.

\subsection{$\Gamma$-convergence}\label{Sec:Gammaconvergence:discrete}
We investigate now the connection of the discretized problem \eqref{Eq:ProblemHw:constrained:discrete2} with the original continuous one in \eqref{Eq:ProblemTV}. 
Thereby we restrict ourselves for simplicity to $\Omega = (0,1)^d$ for $d=1,2$. That is for $d=2$ our domain is a square easily allowing for a finite difference method. However, note that an extension to more general domains is straightforward.

As in \cite[Section 2.7]{Langer:11} and in the proof of \cref{Thm:GammaConvergence}, we define a sequence $(\mathcal{T}_{h_n})_n$ of nested regular meshes of simplices with mesh size $h_n$ covering $\overline{\Omega}$; see  \cref{Fig:Mesh} for a two-dimensional example. Associated to this mesh we define the space of continuous and piecewise affine functions
\begin{equation}\label{Eq:P1}
\mathcal{A}_{h_n} := \{u \in C(\overline\Omega) \colon u_{\mid_K} \in P_1(K), K\in \mathcal{T}_{h_n}\},
\end{equation}
where $P_1(K)$ denotes the space of affine functions on a cell $K$. 

For $d=1$ we let $x_i < x_{i+1}$ for $i = 1,\ldots, n \in\N$, be an equidistant discretization of $\Omega$ such that $x_{i+1}- x_i = \frac{1}{n} = h_n$ and $x_i = \frac{i-1}{n}$ for $i = 1,\ldots, n+1$. Using the notation $u_i:=u(x_i)$ for $u\in \mathcal{A}_{h_n}$ on the interval $K_i = [x_i, x_{i+1}]$ we have $\nabla u = \frac{u_{i+1}-u_i}{h_n}$. Since $\nabla u$ is constant on each interval and the length of each interval is $h_n$ we obtain 
\begin{equation}\label{Eq:dTV1d}
\int_0^1 |\nabla u(x)|\dx = \sum_{i=1}^{n} \int_{x_i}^{x_{i+1}} \left| \frac{u_{i+1}-u_i}{h_n} \right| \dx =  \sum_{i=1}^{n} \left| u_{i+1}-u_i \right|.
\end{equation}
Hence in \eqref{Eq:ProblemHw:constrained:discrete2} any quadrature which is exact for constants can be used, while the gradient might be approximated by standard forward differences. 

For $d=2$ and $n\in\N$, we set $h_n = 2^{-n}$ and denote by $x_{i,j}=(2^{-n} (i-1), 2^{-n} (j-1))$ for $i,j=1,\ldots, 2^n+1$ the nodes of the respective mesh. Note that $h_n= x_{i+1,j} - x_{i,j} = x_{i,j+1}-x_{i,j}$ for $i,j=1,\ldots, 2^n$ is not the diameter of the triangles, but the length of the two shorter edges, while the diameter of every triangle is $\sqrt{2} h_n$. For $u\in\mathcal{A}_{h_n}$ we write $u_{i,j}:=u(x_{i,j}).$ Each pair of triangles, $K^-_{i,j}$ and $K^+_{i,j}$, that share a long side combine to form a square, see \cref{Figsquare}.
\begin{figure}[htbp]
	\centering
	\begin{tikzpicture}[scale=0.25]
% Draw the outer frame
\draw[thick] (0,0) rectangle (12,12);
% Label K^-_{i,j}
\node at (3,9) {\(K^-_{i,j}\)};
% Label K^+_{i,j}
\node at (9,3) {\(K^+_{i,j}\)};
% Draw the diagonal line
\draw[thick] (0,0) -- (12,12);
% Draw the corner nodes
\filldraw (0,0) circle (6pt);  % Bottom-left corner node
\filldraw (12,0) circle (6pt);  % Bottom-right corner node
\filldraw (0,12) circle (6pt);  % Top-left corner node
\filldraw (12,12) circle (6pt);  % Top-right corner node
% Descriptions near the nodes
\node at (-0.5,-0.8) {\(x_{i+1,j}\)};  % Bottom-left corner description
\node at (12.5,-0.8) {\(x_{i+1,j+1}\)};  % Bottom-right corner description
\node at (-0.5,12.9) {\(x_{i,j}\)};  % Top-left corner description
\node at (12.5,12.9) {\(x_{i,j+1}\)};  % Top-right corner description
\end{tikzpicture}
	\caption{\small \it Square $K^-_{i,j}\cup K^+_{i,j}$ formed by a pair of triangles.}\label{Figsquare}
\end{figure}
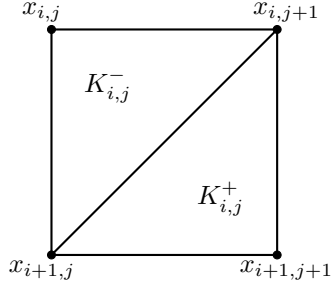
The gradient of $u\in\mathcal{A}_{h_n}$ in the square $K^-_{i,j}\cup K^+_{i,j}$ is given as
\begin{equation}\label{Eq:dGradient2D}
\begin{split}
&\nabla u_{|_{K_{i,j}^-}}=\left(\frac{u_{i+1,j}-u_{i,j}}{h_n}, \frac{u_{i,j+1}-u_{i,j}}{h_n}\right) =: \frac{1}{h_n}\left(v^-_{1,i,j}, v^-_{2,i,j}\right)\\
&\nabla u_{|_{K_{i,j}^+}}=\left(\frac{u_{i+1,j+1}-u_{i,j+1}}{h_n}, \frac{u_{i+1,j+1}-u_{i+1,j}}{h_n}\right)=: \frac{1}{h_n}\left(v^+_{1,i,j}, v^+_{2,i,j}\right).
\end{split}
\end{equation}
Noting that the area of each triangle is $\frac{h_n^2}{2}$ a straightforward calculation yields
\begin{equation}\label{Eq:dTV}
\int_{\Omega} |\nabla u(x)|_2 \dx =
   \frac{h_n}{2} \sum_{i=1}^{2^n} \sum_{j=1}^{2^n} \left(\left|\left(v^-_{1,i,j}, v^-_{2,i,j}\right)\right|_2 + \left|\left(v^+_{1,i,j}, v^+_{2,i,j}\right)\right|_2\right).
\end{equation}
This suggests to set $\Omega^h=\{x_{i,j} \colon i,j=1,\ldots, 2^n\}$, $w_{q}^x = h_n^2$ for all $x\in\Omega^h$ and $q\in\{1,2,\TV\}$, and $|v|_\blacktriangleright = \frac{1}{2}\left(|v^-|_2 + |v^+|_2\right)$ for $ v=(v^-,v^+)$ with $v^-,v^+\in\R^2$ in \eqref{Eq:ProblemHw:constrained:discrete2}.

Note that \eqref{Eq:dTV1d} and \eqref{Eq:dTV} may be interpreted as the discrete analogue of the continuous total variation of $u$ in $\Omega$. 
Hence the ``integral counterpart'' to the discrete functional in \eqref{Eq:ProblemHw:constrained:discrete2} is then given by
\begin{equation}\label{Eq:IntegralCounterpart}
\mcE_n(u) := \begin{cases}
\alpha_1 \|T u - g \|_{L^1(\Omega)} + \alpha_2 \|T u - g \|_{L^2(\Omega)}^2 + \lambda \int_{\Omega} |\nabla u(x)| \dx & \text{if } u\in \mathcal{A}_{h_n} \cap \mcHc{c_n}{M_n}, \\
+\infty &\text{otherwise},
\end{cases}
\end{equation}
where $(c_n)_n\subset \R$ and $(M_n)_n \subset \N$. 
Recall that for every $c_n\geq 0$ and $M_n$ the space $\mcHc{c_n}{M_n}$ is a set of continuous piecewise affine functions. If $c_n$ and $M_n$ are large enough (possibly infinite) then $\mathcal{A}_{h_n} \subseteq \mcHc{c_n}{M_n}$ by \cite[Theorem 2.1]{ArBaMiMu:16}. Moreover, $0\in \mathcal{A}_{h_n} \cap \mcHc{c_n}{M_n}$ and hence $\mathcal{A}_{h_n} \cap \mcHc{c_n}{M_n}\neq \emptyset$ for any $n\in\N$.

We say that the discrete functional in \eqref{Eq:ProblemHw:constrained:discrete2} $\Gamma$-converges to $\mcE$ defined in \eqref{Eq:E}, if $(\mcE_n)_n$, where $\mcE_n$ is defined in \eqref{Eq:IntegralCounterpart}, $\Gamma$-converges to $\mcE$, cf. \cite[Definition 4.2]{Braides:02}.

\begin{theorem}\label{Thm:Gamma:Convergence:discrete}
Let $\mcE_n, \mcE:L^1(\Omega) \to \overline{\R}$ be defined as in \eqref{Eq:IntegralCounterpart} and \eqref{Eq:E}, respectively, and $u_n$ be a minimizer of $\mcE_n$, $n\in\N$. Then there exist sequences $(c_n)_n \subset \R$ and $(M_n)_n \subset \N$ such that $(\mcE_n)_n$ $\Gamma$-converges to $\mcE$ with respect to the $L^1(\Omega)$ topology. Moreover, if $\mcE$ is coercive on $BV(\Omega)$, then we have that 
\begin{equation*}
\lim_{n\to\infty} \inf_{u\in L^1(\Omega)} \mcE_n(u) = \min_{u\in L^1(\Omega)} \mcE(u)
\end{equation*}
and every limit of a subsequence of $(u_{n})_{n}$ is a minimizer of $\mcE$.
\end{theorem}

\begin{proof}
The proof is similar to the one of \cref{Thm:GammaConvergence}. We only recall the important steps and differences.

The $\liminf$-condition (\cref{Def:GammaConvergence}\ref{Def:GammaConvergence:i}) follows directly from the lower semicontinuity of the functional $\mcE$.

The $\limsup$-condition (\cref{Def:GammaConvergence}\ref{Def:GammaConvergence:ii}) clearly holds for any $u \in L^1(\Omega) \setminus (BV(\Omega) \cap L^2(\Omega))$. For $u \in BV(\Omega) \cap L^2(\Omega)$ there is a sequence $(v_n)_n \in C^\infty(\overline{\Omega})$ converging to $u$ in $L^1(\Omega)$ such that $\lim_{n\to\infty} \int_\Omega |\nabla v_n|_2\dx = \TV(u)$ \cite[Theorem 10.1.2]{AtBuMi:14}. Then for all $v\in C^\infty(\overline{\Omega})$, we can construct a continuous piecewise affine function $v_{h_n}\in\mathcal{A}_{h_n}$, whereby $v_{h_n}$ coincides with $v$ in every node of $\mathcal{T}_{h_n}$. 
Further we choose $c_n$ and $M_n$ such that $v_{h_n}\in \mcHc{c_n}{M_n}$ as well, i.e.\ $v_{h_n}\in\mathcal{A}_{h_n} \cap \mcHc{c_n}{M_n}$. Defining the interpolation operator $\Pi_{h_n}: C^\infty(\overline{\Omega}) \to \mathcal{A}_{h_n} \cap \mcHc{c_n}{M_n}$ with $v\mapsto v_{h_n} = \Pi_{h_n} v$ and utilizing interpolation error estimates, see \cite[Theorem 3.1.5]{Ciarlet:02}, we can construct analog to the proof of \cref{Thm:GammaConvergence} a sequence $(u_k)_k$ given as
\begin{equation*}
u_k := \Pi_{h_{k}}v_{k} \in \mathcal{A}_{h_{k}} \cap \mcHc{c_{k}}{M_{k}}
\end{equation*}
such that 
$$
\lim_{k\to\infty} u_k = u \qquad \text{and} \qquad \lim_{n\to\infty} \TV(u_k) = \TV(u).
$$ 
Noting that $T$ and $u\mapsto \alpha_1 \|T u - g \|_{L^1(\Omega)} + \alpha_2 \|T u - g \|_{L^2(\Omega)}^2$ are continuous yields 
$$
\lim_{k\to \infty} \mcE_k(u_k) = \mcE(u)
$$
and hence the $\Gamma$-convergence assertion.

Since $\mcE$ is coercive on $BV(\Omega)$ it possesses a minimizer in $BV(\Omega)$ \cite[Theorem 3.3]{Langer:17a} and the set $S:=\{u\in BV(\Omega) \colon \mcE(u)\leq a\}$ is bounded in $BV(\Omega)$, where $a:= \mcE(0)$. 
Since $0\in BV(\Omega)$ and $0\in \mathcal{A}_{h_n}\cap \mcHc{c_n}{M_n}$ for all $n\in\N$, any minimizer of $(\mcE_n)_n$ and $\mcE$ is in $S$. Thereby \cref{Prop:ExistenceMinimizer} ensures a minimizer of $\mcE_n$.
Since $\overline{S}$ is compact in $L^1(\Omega)$ \cite[Theorem 2.5]{AcarVogel:94} and such that $\inf_{L^1(\Omega)} \mcE_n = \inf_{\overline{S}} \mcE_n$, the $\Gamma$-convergence of $(\mcE_n)_n$ to $\mcE$ and \cite[Theorem 1.21]{Braides:02} yield
\begin{equation*}
\lim_{n\to\infty} \inf_{u\in L^1(\Omega)} \mcE_n(u) = \min_{u\in L^1(\Omega)} \mcE(u).
\end{equation*}
Moreover, since any subset of a compact set is relatively compact and the sequence of minimizers $(u_{n})_{n}\subseteq \overline{S}$, we have that every limit of a subsequence of $(u_{n})_{n}$ is a minimizer of $\mcE$ \cite[Theorem 1.21]{Braides:02}.
\end{proof}

\begin{remark}
Contrary to the proof of \cref{Thm:GammaConvergence} and \cite{Langer:11}, in the proof of \cref{Thm:Gamma:Convergence:discrete} we do not use the nested property of the mesh.  However, this can be easily achieved by requesting that $c_n \geq c_{n-1}$ and $M_n \geq M_{n-1}$, whereby at most neurons to the existing hidden layers are added, for all $n>1$, as in the proof of \cref{Thm:GammaConvergence}. 
\end{remark}

The $\Gamma$-convergence result relies on a continuous and piecewise affine approximation of an arbitrarily smooth function. This approximation suggests particular discretisations for the gradient and total variation:
\begin{enumerate}[(i)]
\item In one-dimension ($d=1$), the gradient of a continuous piecewise affine function $u$ in the interval $[x_i, x_{i+1}]$ is given by $\nabla u = \frac{u_{i+1}-u_{i}}{h_n}$. This suggests to use forward differences to approximate the gradient operator. Nevertheless, analogously we could have also set $\nabla u = \frac{u_{i}-u_{i+1}}{h_n}$ suggesting backward differences. It is clear that this does not change the total variation of $u$, see \eqref{Eq:dTV1d}, and hence the $\Gamma$-convergence proof keeps the same. In our numerical experiments for $d=1$ we will use the forward difference approximation of the gradient.

\item In two-dimensions ($d=2$), the gradient is given by \eqref{Eq:dGradient2D}, which means in a finite difference setting a combination of forward and backward differences is suggested. More precisely, the finite difference gradient $\nabla^h$ is then
\begin{equation}\label{Eq:ForwardBackwardGradient}
\nabla^h u_{i,j} = \nabla_{\text{FB}}^h u_{i,j}:=\left( \frac{u_{i+1,j}-u_{i,j}}{h_n} , \frac{u_{i,j+1}-u_{i,j}}{h_n}, \frac{u_{i,j}-u_{i-1,j}}{h_n}, \frac{u_{i,j}-u_{i,j-1}}{h_n}\right) \in\R^4
\end{equation}
and the discrete total variation as in \eqref{Eq:dTV}. That is, for $v=(v_1,v_2,v_3,v_4)\in\R^4$ and $r\in\N$ we define $|v|_{r,1}$ as 
\begin{equation}\label{Eq:21Norm}
|v|_{r,1}:=\frac{1}{2}\left(\left|\left( v_1,v_2\right)\right|_r +  \left|\left( v_3,v_4\right)\right|_r\right).
\end{equation}
Then the discrete total variation of $u\in\R^{2^n\times 2^n}$ is given as $h_n^2\sum_{i,j=1}^{2^n}|\nabla^h u_{i,j}|_{r,1}$. If $r=1$, then we have that $h_n^2\sum_{i,j=1}^{2^n}|\nabla^h u_{i,j}|_{1,1} = h_n^2\sum_{i,j=1}^{2^n}|\nabla^h u_{i,j}|_{1}$, while for $r\neq 1$ in general we have $h_n^2\sum_{i,j=1}^{2^n}|\nabla^h u_{i,j}|_{r,1} \neq h_n^2\sum_{i,j=1}^{2^n}|\nabla^h u_{i,j}|_{r}$.
\end{enumerate}
Note that if we set $c_n$ and $M_n$ in \eqref{Eq:IntegralCounterpart} such that $\mathcal{A}_{h_n} \cap \mcHc{c_n}{M_n} = \mathcal{A}_{h_n}$, then \cref{Thm:Gamma:Convergence:discrete} shows the $\Gamma$-convergence of a finite difference approximation of $\mcE$ to $\mcE$, as in \cite{Langer:11}. We emphasize that showing the $\Gamma$-convergence of a finite difference approximation of $\mcE$ to $\mcE$ also seems to require a continuous approximation of a $C^\infty$-function. In fact, for piecewise constant approximations a $\Gamma$-convergence result cannot be shown for total variation minimization, as a counterexample presented in \cite{Bartels:12} illustrates.

\subsubsection{Choice of discretisation}

Based on the above result we will make the following choices on the total variation:
\begin{enumerate}[leftmargin=*,align=left]
\item[$d=1$]: We use a standard forward difference discretised gradient operator with incorporated boundary conditions and denote it by $\nabla^h_{\text{F}}$.
This leads to the one-dimensional discrete total variation
$$
\frac{1}{N} \sum_{x\in\Omega^h} |\nabla^h_{\text{F}}u(x)|.
$$

\item[$d=2$]: Here we will consider two different discretised gradient operators. The standard forward difference discretized gradient operator denoted by $\nabla^h_{\text{F}}$ and the above suggested forward-backward discretisation $\nabla^h_{\text{FB}}$. For both choices we incorporate in a standard way respective boundary conditions. The two-dimensional discrete total variations read then as
$$
\TV_2^h(u) := \frac{1}{N} \sum_{x\in\Omega^h} |\nabla^h_{\text{F}}u(x)|_2 \qquad \text{and} \qquad \TV_{2,1}^h(u) :=  \frac{1}{N} \sum_{x\in\Omega^h} |\nabla^h_{\text{FB}}u(x)|_{2,1}
$$
respectively.

We note that $\TV_2^h$ is extensively used in image processing and is referred to as isotropic total variation in the literature, see e.g.\ \cite{Condat:17}. That is the main reason why we additionally consider $\TV_2^h$ in our numerical experiments, despite the missing $\Gamma$-convergence result.
\end{enumerate}

\subsection{Identifying and solving a numerical issue}\label{Sec:NumericalIssue}
The implementation of the neural network model is based on \tf\ and in order to find a minimizer we utilize the optimization algorithm Adam \cite{KingmaBa:15}, which is a stochastic gradient descent method and already pre-implemented in \tf. 
It turns out that Adam in \tf\ has difficulties in optimizing \eqref{Eq:ProblemHw:constrained:discrete2} with $\TV_2^h$ and $\TV_{2,1}^h$ for $d=2$,  which seems due to the presence of the square-root. At least after some iterations the gradient becomes so steep that the algorithm returns an \textsc{NaN} and further computations seem impossible. In the rest of this section we set $d=2$ and present and analyze options to circumvent this issue. In particular we will approximate the total variation considering the following options:
\begin{enumerate}
\item Huber-regularization: This is a commonly used approximation technique for total variation minimization, see for example \cite{BuPaPaSc:16,HiLaAl:23,PoCrBiCh:10}. We define for $\gamma>0$ the Huber-functions $\phi_{2,\gamma}: \R^2 \to [0,\infty)$ and $\phi_{2,1,\gamma}: \R^4 \to [0,\infty)$ by
\begin{equation} \label{eq:phidef}
    \phi_{2,\gamma}(u) := \begin{cases}
      \frac{1}{2\gamma} |u|_2^2 & \text{if } |u|_2 \le \gamma, \\
      |u|_2 - \frac{\gamma}{2} & \text{if } |u|_2 > \gamma,
    \end{cases}
  \end{equation}
  for $u\in\R^2$ and 
 \begin{equation*}
\phi_{2,1,\gamma} (w) := \frac{1}{2}\left(\phi_{2,\gamma} (u) + \phi_{2,\gamma} (v) \right) = \begin{cases}
\frac{1}{2} \left(|u|_2 + |v|_2\right) - \frac{1}{2} \gamma & \text{if } |u|_2 > \gamma \text{ and } |v|_2> \gamma\\
\frac{1}{2}|u|_2 - \frac{1}{4}\gamma + \frac{1}{4\gamma} |v|_2^2 & \text{if } |u|_2 > \gamma \text{ and } |v|_2\le \gamma \\
\frac{1}{2}|v|_2 - \frac{1}{4}\gamma + \frac{1}{4\gamma} |u|_2^2 & \text{if } |u|_2 \le \gamma  \text{ and } |v|_2 > \gamma\\
\frac{1}{4\gamma}\left(|u|_2^2 + |v|_2^2\right) & \text{if } |u|_2 \le \gamma \text{ and } |v|_2\le \gamma
\end{cases}
\end{equation*}
for $w=(u,v)$ with $u,v\in \R^2$.
Then in the case of $\TV_2^h$ and $\TV_{2,1}^h$ we replace $|\cdot|_2$ by $\phi_{2,\gamma}(\cdot)$ and $|\cdot|_{2,1}$ by $\phi_{2,1,\gamma}(\cdot)$. Hence the Huber-regularized discrete total variation is then given by $h^2 \sum_{x\in\Omega^h} \phi_{2,\gamma}(\nabla^h u_\theta(x))$ or $h^2\sum_{x\in\Omega^h} \phi_{2,1,\gamma}(\nabla^h u_\theta(x))$ respectively.

\item $\gamma$-lifting: For $\gamma>0$ we replace $|\cdot|_2$ by $|\cdot|_{2,\gamma}$ and $|\cdot|_{2,1}$ by $|\cdot|_{2,1,\gamma}$, whereby for $v=(v_1,v_2)\in\R^2$ we define $|v|_{2,\gamma} := \sqrt{v_1^2 + v_2^2 + \gamma}$ and for $w=(w_1,w_2,w_3,w_4)\in\R^4$ we define 
$$
|w|_{2,1,\gamma} := \frac{1}{2}\left(\sqrt{w_1^2 + w_2^2 + \gamma} + \sqrt{w_3^2 + w_4^2 + \gamma}\right) = \frac{1}{2}\left(|(w_1,w_2)|_{2,\gamma} + |(w_3,w_4)|_{2,\gamma} \right).
$$
 The choice $|\cdot|_{2,\gamma}$ has been extensively studied in \cite{AcarVogel:94}.
\item $\max$-$\gamma$-lifting: Instead of $|\cdot|_2$ and $|\cdot|_{2,1}$ we use $\max\{|\cdot|_2,\gamma\}$ and $|\cdot|_{2,1,\max_\gamma}$ respectively, where $\gamma>0$ and 
$$
|w|_{2,1,\max_\gamma}:=\frac{1}{2}\left(\max\{|(w_1,w_2)|_{2},\gamma\}+ \max\{|(w_3,w_4)|_{2},\gamma\}\right) \qquad \text{for } w\in\R^4.
$$

\end{enumerate}
In \cite[Proposition 3.3]{HiLaAl:23} it is shown that the Huber-regularization $\phi_{2,\gamma}$ approximates the norm $|\cdot|_2$ from below. Analogously one also shows that $\phi_{2,1,\gamma}$ approximates the norm $|\cdot|_{2,1}$ from below. On the contrary the $\gamma$-lifting and $\max$-$\gamma$-lifting approximate the corresponding norm from above, since $|v|_2 \leq |v|_{2,\gamma}$ for all $v\in\R^2$ and $|w|_{2,1} \leq |w|_{2,1,\gamma}$ for all $w\in\R^4$. Further all three approximations converge for $\gamma\to 0$ to the respective norm.

We recall that $\mcH \subset W^{1,1}(\Omega)$ and hence for any $u\in \mcH$ we have $\TV(u) = \int_\Omega |\nabla u| \dx$.
We analyze now the $\Gamma$-convergence of \eqref{Eq:En} with $\TV(u) = \int_\Omega \varphi_\gamma(\nabla u) \dx$, where $\varphi_\gamma (\cdot) \in\left\{|\cdot|_{2,\gamma}, \max\{|\cdot|_{2},\gamma\}, \phi_{2,\gamma}(\cdot) \right\} $ and how their minimizers behave for $\gamma\to 0$.

\subsubsection{$\Gamma$-convergence of total variation approximations} 

Using the notation and setting of \cref{Sec:Gammaconvergence:discrete} one easily computes on a regular mesh of triangles for $u\in\mathcal{A}_{h_n}$ that 
\[
\int_\Omega \varphi_\gamma(\nabla u(x)) \dx = \frac{h_n^2}{2} \sum_{i,j=1}^{2^n}  \left(\varphi_\gamma\left(\frac{1}{h_n}(v^-_{1,i,j}, v^-_{2,i,j})\right) + \varphi_\gamma\left(\frac{1}{h_n}(v^+_{1,i,j}, v^+_{2,i,j})\right)\right),
\]
cf.\ \eqref{Eq:dTV}. Hence we have
\begin{equation}\label{Eq:varphi:gamma}
\int_\Omega \varphi_\gamma(\nabla u(x)) \dx = \begin{cases}
h_n^2 \sum_{i,j=1}^{2^n}|\nabla^h_{\text{FB}} u_{i,j}|_{2,1,\gamma} & \text{if } \varphi_\gamma(\cdot) = |\cdot|_{2,\gamma} \\
h_n^2 \sum_{i,j=1}^{2^n}  |\nabla^h_{\text{FB}} u_{i,j}|_{2,1,\max_\gamma} & \text{if } \varphi_\gamma(\cdot) = \max\{|\cdot|_{2}, \gamma\}\\
h_n^2 \sum_{i,j=1}^{2^n}  \phi_{2,1,h}(\nabla^h_{\text{FB}} u_{i,j}) & \text{if } \varphi_\gamma(\cdot)= \phi_{2,h}(\cdot).
\end{cases}
\end{equation} 
Note that \eqref{Eq:IntegralCounterpart} with $\TV(u) = \int_\Omega \varphi_\gamma(\nabla u(x)) \dx$ is the continuous counterpart to \eqref{Eq:ProblemHw:constrained:discrete2} with 
\begin{equation*}
|\nabla^h u_{i,j}|_{\blacktriangleright} = \begin{cases}
|\nabla^h_{\text{FB}} u_{i,j}|_{2,1,\gamma} & \text{if } \varphi_\gamma(\cdot) = |\cdot|_{2,\gamma} \\
|\nabla^h_{\text{FB}} u_{i,j}|_{2,1,\max_\gamma} & \text{if } \varphi_\gamma(\cdot) = \max\{|\cdot|_{2}, \gamma\}\\
\phi_{2,1,h}(\nabla^h_{\text{FB}} u_{i,j}) & \text{if } \varphi_\gamma(\cdot)= \phi_{2,h}(\cdot).
\end{cases}
\end{equation*}

We state some elementary properties of $\gamma$-lifting and $\max$-$\gamma$-lifting together with their proofs.

\begin{lemma}\label{Lemma:LiftFunctional:Properties}
Let $\varphi_\gamma (\cdot) \in\left\{|\cdot|_{2,\gamma}, \max\{|\cdot|_{2},\gamma\} \right\}$, then we have the following properties:
\begin{enumerate}[(i)]
\item\label{Lemma:LiftFunctional:Property1} $0\le \gamma^- \le \gamma^+ \Rightarrow \forall v\in\R^2 \colon 0 \le \varphi_{\gamma^-}(v) \leq \varphi_{\gamma^+}(v)$,
\item\label{Lemma:LiftFunctional:Property2} $\forall v\in\R^2 \colon \lim_{\gamma \to 0^+} \varphi_{\gamma}(v) = \varphi_{0}(v) = |v|_2$,
\item\label{Lemma:LiftFunctional:Property3} $v\in L^1(\Omega)^d$ if and only if $\int_\Omega\varphi_{\gamma}(v(x)) \dx < \infty$,
\item\label{Lemma:LiftFunctional:Property4} $\lim_{\gamma \to 0^+} \int_\Omega \varphi_{\gamma}(v(x)) \dx = \int_\Omega |v(x)|_2\dx$ for all $v\in L^1(\Omega)^d$,
\item\label{Lemma:LiftFunctional:Property5} $0\le \gamma^- \le \gamma^+ \Rightarrow \forall v\in L^1(\Omega)^d \colon 0 \le \int_\Omega \varphi_{\gamma^-}(v(x))\dx \leq \int_\Omega \varphi_{\gamma^+}(v(x))\dx$.
\end{enumerate}
\end{lemma}
\begin{proof}
The statements in (i) and (ii) follow directly from the definition of $\varphi_\gamma$. (iii) is implied by the equivalency of vector norms in $\R^d$ and the fact that $\varphi_\gamma (v) \ge |v|_2$, $|v|_{2,\gamma} \le |v|_2 + \gamma^{1/2}$ and $\max\{|v|_2,\gamma\} \le |v|_2 + \gamma$ for any $v\in \R^2$. 

By \ref{Lemma:LiftFunctional:Property1} there is a function $g\ge\varphi_\gamma\geq 0$ for any sufficiently small $\gamma$, e.g.\ $g = \sup_{\gamma\in(0,\tilde{\gamma}]}\varphi_\gamma$ for $\tilde{\gamma}>0$ sufficiently small. Then by the Dominated Convergence Theorem we have  $\lim_{\gamma \to 0^+} \int_\Omega \varphi_{\gamma}(v(x)) \dx =  \int_\Omega \lim_{\gamma \to 0^+} \varphi_{\gamma}(v(x)) \dx$ and due to \ref{Lemma:LiftFunctional:Property2} the assertion (iv) follows.

(v) follows directly from \ref{Lemma:LiftFunctional:Property1}.
\end{proof}

This allows us to show the following $\Gamma$-convergence result.
\begin{theorem}[$\Gamma$-convergence] \label{Thm:GammaConvergence:lifted}
  Let $\varphi_\gamma (\cdot) \in\left\{|\cdot|_{2,\gamma}, \max\{|\cdot|_{2},\gamma\} , \phi_{2,\gamma}(\cdot)\right\}$, $(\gamma_m)_{m\in\N} > 0$ be
  monotonically decreasing sequences with $\lim_{m\to\infty} \gamma_m = 0$ and $A\in\{\mathcal{A}_{h}\cap \mcHc{c}{M},  \mcHc{c}{M} \}$, where $c\ge 0$, $M\in\N$, $h>0$, and $\mathcal{A}_{h}$ defined as in \eqref{Eq:P1}. Further we define the functionals $\mcE_m^c, \mcE^c: L^1(\Omega) \to \overline{\R}$ as 
\begin{equation}\label{Eq:Emc}
\mcE_{m}^c(u):= \begin{cases}
\alpha_1 \|T u - g \|_{L^1(\Omega)} + \alpha_2 \|T u - g \|_{L^2(\Omega)}^2 + \lambda \int_\Omega \varphi_{\gamma_m}(\nabla {u})\dx & \text{if }\ u \in A ,\\
+\infty & \text{otherwise},
\end{cases}
\end{equation}  
and
\begin{equation*}
\mcE^c(u):= \begin{cases}
\alpha_1 \|T u - g \|_{L^1(\Omega)} + \alpha_2 \|T u - g \|_{L^2(\Omega)}^2 + \lambda \int_\Omega |\nabla u|_2\dx & \text{if }\ u \in A ,\\
+\infty & \text{otherwise}.
\end{cases}
\end{equation*}  
Then $(\mcE_m^c)_m$ $\Gamma$-converges to $\mcE^c$ with respect to the $L^1(\Omega)$ topology.

If $S:=\{u\in A \colon \mcE^c(u)\leq \mcE_1^c(0)\}$ is bounded in $BV(\Omega)$, then we have that 
\begin{equation*}
\lim_{m\to\infty} \inf_{u\in L^1(\Omega)} \mcE_{m}^c(u) = \min_{u\in L^1(\Omega)} \mcE^c(u).
\end{equation*}
Additionally, let $u_m$ be a minimizer of $\mcE_m^c$, $m\in\N$. Then every limit of a subsequence of $(u_{m})_{m}$ is a minimizer of $\mcE^c$.
 \end{theorem}
 
 \begin{proof}
 For $\varphi_\gamma = \phi_{2,\gamma}$ the statement can be shown as in \cite[Section 3.3]{HiLaAl:23}. 
 
 The proof for $\varphi_\gamma (\cdot) \in\left\{|\cdot|_{2,\gamma}, \max\{|\cdot|_{2},\gamma\}\right\}$ runs as follows:
 By the monotonicity property of $\varphi_\gamma$ from \cref{Lemma:LiftFunctional:Properties} \ref{Lemma:LiftFunctional:Property5}, we observe that $\mcE_m^c({u}) \ge \mcE_{m+1}^c({u})$ and $\mcE_m^c({u}) \to \mcE^c({u})$ pointwise for every fixed ${u}\in L^1(\Omega)$. Further, $\mcE^c$ is lower semicontinuous with respect to the $L^1(\Omega)$ topology. 
 Then according to \cite[Remark 1.40 (i)]{Braides:02} we thus have that $(\mcE_m^c)_m$ $\Gamma$-converges to $\mcE^c$ with respect to the $L^1(\Omega)$ topology. 
    
   In order to prove the second statement we need to show that there exists a non-empty compact set $K\subset L^1(\Omega)$ such that $\inf_K \mcE_m^c = \inf_{L^1(\Omega)} \mcE_m^c$ for all $m\in \N$:   
  Let $S_m^a:=\{u\in A \colon \mcE_m^c(u)\leq a\}$, $a:=\mcE_1^c(0) < \infty$, be the lower level set of $\mcE_m^c$ for $m\in\N$. The monotonicity of the functional $\mcE_m^c$, i.e.\ $\mcE_m^c \ge \mcE_{m+1}^c$, yields $S_m^a \subseteq S_{m+1}^a \subseteq S$ for $m\in\N$, while  $a:=\mcE_1^c(0)$ ensures $S_m^a\neq \emptyset\neq S$ for all $m\in\N$.
   
   The existence of a minimizer of $\mcE^c$ is ensured by \cref{Prop:ExistenceMinimizer}, while one similarly shows the existence of a minimizer of $\mcE_m^c$, $m\in\N$. By the monotonicity $\mcE_m^c \geq \mcE_{m+1}^c \geq \mcE^c$ all minimizers of $\mcE_m^c$, $m\in\N$, and $\mcE^c$ are contained in $S$. Since $S$ is bounded in $BV(\Omega)$ it is relatively compact in $L^1(\Omega)$ \cite[Theorem 2.5]{AcarVogel:94} and we have $\inf_K \mcE_m^c = \inf_{L^1(\Omega)} \mcE_m^c$ for all $m\in\N$, where $K=\overline{S}$. An application of \cite[Theorem 1.21]{Braides:02} yields 
   \begin{equation*}
\lim_{m\to\infty} \inf_{u\in L^1(\Omega)} \mcE_{m}^c(u) = \min_{u\in L^1(\Omega)} \mcE^c(u).
\end{equation*}
Moreover, for $m\in\N$, let $u_{m}$ be a minimizer of $\mcE_m^c$ such that $$\lim_{m\to\infty} \mcE_{m}^c(u_{m}) = \lim_{m\to\infty} \inf_{u\in L^1(\Omega)} \mcE_{m}^c(u).$$ Then, since $(u_{m})_{m}\subseteq S$ and any subset of a relatively compact set is relatively compact, we have that every limit of a subsequence of $(u_{m})_{m}$ is a minimizer of $\mcE^c$ \cite[Theorem 1.21]{Braides:02}.
\end{proof}

\Cref{Thm:GammaConvergence:lifted} does not link \eqref{Eq:Emc} to \eqref{Eq:ProblemTV}. To obtain this connection we use the following result.
\begin{lemma}\label{Lem:GammaConvergence}
Let $V$ be a Banach space and $\mcE,\mcE_n,\mcE_{n,m}: V \to \overline{\R}$. If $(\mcE_n)_n$ $\Gamma$-converges to $\mcE$ and $(\mcE_{n,m})_m$ $\Gamma$-converges to $\mcE_n$ for all $n\in\N$, then $(\mcE_{n,m})_{n,m}$ $\Gamma$-converges to $\mcE$.
\end{lemma}
\begin{proof}
We start by showing the $\liminf$-condition, i.e. for all $u\in V$ and for all $(u_{n,m})_{n,m}\subset V$ with $\lim_{n,m\to\infty}u_{n,m} = u$ we have 
$$
\mcE(u) \le \liminf_{n,m\to\infty} \mcE_{n,m}(u_{n,m}).
$$
Since $(\mcE_n)_n$ $\Gamma$-converges to $\mcE$ and $(\mcE_{n,m})_m$ $\Gamma$-converges to $\mcE_n$ for all $n\in\N$ we have that 
\begin{equation}\label{Eq:1:Lemma:GammaConvergence:Transitivity}
\mcE(u) \le \liminf_{n\to\infty} \mcE_{n}(u_{n}) \qquad \text{and} \qquad 
\mcE_n(u_n) \le \liminf_{m\to\infty} \mcE_{n,m}(u_{n,m})
\end{equation}
for all $u\in V$, $(u_n)_n\subset V$ and $(u_{n,m})_m\subset V$ with $\lim_{n\to \infty}u_n = u$ and $\lim_{m\to \infty}u_{n,m} = u_n$ for an arbitrary $n\in\N$.
The inequalities in \eqref{Eq:1:Lemma:GammaConvergence:Transitivity} can be equivalently written as: For all $\delta>0$ there is an $N\in\N$ such that for all $n,m\geq N$ we have $\mcE(u)\leq \mcE_n(u_n) + \delta/2$ and $\mcE_n(u_n)\leq \mcE_{n,m}(u_{n,m}) + \delta/2$. Combining the last two inequalities yields $\mcE(u)\leq \mcE_{n,m}(u_{n,m}) + \delta$. For $\delta\to 0$ we obtain
$$
\mcE(u) \le \liminf_{n,m\to\infty} \mcE_{n,m}(u_{n,m}).
$$ 
Let us turn to show the $\limsup$-condition, i.e.\ for all $u\in V$ there is a sequence $(u_{n,m})_{n,m}\subset V$ with $\lim_{n,m\to\infty}u_{n,m} = u$ such that 
$$
\mcE(u) \ge \limsup_{n,m\to\infty} \mcE_{n,m}(u_{n,m}).
$$
Again, since $\mcE_n$ $\Gamma$-converges to $\mcE$ and $\mcE_{n,m}$ $\Gamma$-converges to $\mcE_n$ we have that for all $u\in V$
\begin{equation}\label{Eq:2:Lemma:GammaConvergence:Transitivity}
\mcE(u) \ge \limsup_{n\to\infty} \mcE_n(u_{n}) \qquad \text{and} \qquad 
\mcE_n(u_n) \ge \limsup_{m\to\infty} \mcE_{n,m}(u_{n,m})
\end{equation}
for some $(u_n)_n\subset V$ and $(u_{n,m})_m\subset V$ with $\lim_{n\to \infty}u_n = u$ and $\lim_{m\to \infty}u_{n,m} = u_n$ for an arbitrary $n\in\N$. Instead of \eqref{Eq:2:Lemma:GammaConvergence:Transitivity} we write equivalently: For all $\delta>0$ there is an $N\in\N$ such that for all $n,m\geq N$ we have $\mcE(u)\geq \mcE_n(u_n) - \delta/2$ and $\mcE_n(u_n)\geq \mcE_{n,m}(u_{n,m}) - \delta/2$. Combining the last two inequalities yields $\mcE(u)\geq \mcE_{n,m}(u_{n,m}) - \delta$. For $\delta\to 0$ we obtain
$$
\mcE(u) \ge \limsup_{n,m\to\infty} \mcE_{n,m}(u_{n,m}).
$$ 
\end{proof}

Now we are ready to link \eqref{Eq:varphi:gamma} to the total variation in \eqref{Eq:ProblemTV} for functions in $BV(\Omega)\cap L^2(\Omega)$.

\begin{theorem}[$\Gamma$-convergence]\label{Thm:GammaConvergence:Phigamma}
Let $\varphi_\gamma (\cdot) \in\left\{|\cdot|_{2,\gamma}, \max\{|\cdot|_{2},\gamma\} , \phi_{2,\gamma}(\cdot)\right\}$, $(\gamma_m)_{m\in\N} > 0$ be
 a monotonically decreasing sequences with $\lim_{m\to\infty} \gamma_m = 0$ 
 and $\mathcal{A}_{h}$ defined as in \eqref{Eq:P1} with $h>0$.
Further let $\mcE$ as in \eqref{Eq:E} and define $\mcE_{n,m}: L^1(\Omega) \to \overline{\R}$ as
\begin{equation*}
\mcE_{n,m}(u):= \begin{cases}
\alpha_1 \|T u - g \|_{L^1(\Omega)} + \alpha_2 \|T u - g \|_{L^2(\Omega)}^2 + \lambda \int_\Omega \varphi_{\gamma_m}(\nabla {u})\dx & \text{if }\ u \in A_n ,\\
+\infty & \text{otherwise},
\end{cases}
\end{equation*}
where $A_n\in\{\mathcal{A}_{h_n}\cap \mcHc{c_n}{M_n},   \mcHc{c_n}{M_n} \}$.
Then there exist sequences $(c_n)_n \subset \R$ and $(M_n)_n\subset \N$ such that the sequence $(\mcE_{n,m})_{n,m}$ $\Gamma$-converge to $\mcE$ with respect to the $L^1(\Omega)$ topology.

If $\mcE$ is coercive on $BV(\Omega)$, then we have that 
\begin{equation*}
\lim_{n,m\to\infty} \inf_{u\in L^1(\Omega)} \mcE_{n,m}(u) = \min_{u\in L^1(\Omega)} \mcE(u).
\end{equation*}
Additionally, let $u_{n,m}$ be a minimizer of $\mcE_{n,m}$, $n,m\in\N$. Then every limit of a subsequence of $(u_{n,m})_{n,m}$ is a minimizer of $\mcE$.
\end{theorem}
\begin{proof}
We define
\begin{equation*}
\mcE_n(u):= \begin{cases}
\alpha_1 \|T u - g \|_{L^1(\Omega)} + \alpha_2 \|T u - g \|_{L^2(\Omega)}^2 + \lambda \int_\Omega |\nabla u|_2\dx & \text{if }\ u \in A_n ,\\
+\infty & \text{otherwise}
\end{cases}
\end{equation*}
and apply \cref{Lem:GammaConvergence} with $V=L^1(\Omega)$. More precisely, by \cref{Thm:GammaConvergence:lifted} we have that $(\mcE_{n,m})_m$ $\Gamma$-converges to $\mcE_n$ for any $n\in\N$ and by \cref{Thm:GammaConvergence} and \cref{Thm:Gamma:Convergence:discrete} we have $(\mcE_n)_n$ $\Gamma$-converges to $\mcE$. \Cref{Lem:GammaConvergence} yields then that $(\mcE_{n,m})_{n,m}$ $\Gamma$-converge to $\mcE$ with respect to the $L^1(\Omega)$ topology.

For showing the second part of the theorem we define $S_{n,m}^a:=\{u\in BV(\Omega) \colon \mcE_{n,m}(u)\leq a\}$ for all $n,m\in\N$ and separate the proof between the lifting-regularization and the Huber-regularization. The reason is that the lifting-regularization approaches the limit from above while the Huber-regularization does so from below, which leads to different, although similar, arguments.
\begin{enumerate}[(a)]
\item Let $\varphi_\gamma (\cdot) \in\left\{|\cdot|_{2,\gamma}, \max\{|\cdot|_{2},\gamma\}\right\}$. By the monotonicity of $\varphi_\gamma$, see \cref{Lem:GammaConvergence} \ref{Lemma:LiftFunctional:Property1} and \ref{Lemma:LiftFunctional:Property5}, we observe that $\mcE_{n,m}\ge \mcE_{n,m+1}\ge \mcE_n \ge \mcE$ for all $n,m\in\N$ and $S_{1,1}^a \subseteq S_{n,m}^a \subseteq S_{n,m+1}^a \subseteq S_{n}^a \subseteq S^a$. We choose $a:=\mcE_{1,1}(0)$ which ensures that $S_{1,1}^a \neq \emptyset$. The coercivity of $\mcE$ implies that $S^a$ is relatively compact in $L^1(\Omega)$ \cite[Theorem 2.5]{AcarVogel:94}. By this and $\inf_{L^1(\Omega)}\mcE_{n,m} = \inf_{\overline{S^a}}\mcE_{n,m}$ for all $n,m\in\N$, as the minimizers of $\mcE_{n,m}$ are in $\overline{S^a}$, it follows that 
\begin{equation*}
\lim_{n,m\to\infty} \mcE_{n,m}(u_{n,m}) := \lim_{n,m\to\infty} \inf_{u\in L^1(\Omega)} \mcE_{n,m}(u) = \min_{u\in L^1(\Omega)} \mcE(u)
\end{equation*}
\cite[Theorem 1.21]{Braides:02}. Moreover, since $(u_{n,m})_{n,m}\subseteq S^a$ and any subset of a relatively compact set is relatively compact, we have that every limit of a subsequence of $(u_{n,m})_{n,m}$ is a minimizer of $\mcE$ \cite[Theorem 1.21]{Braides:02}.

\item For $\varphi_\gamma = \phi_{2,\gamma}$ we note that $(\phi_{2,\gamma_m})_m$ is monotonically increasing and hence $\mcE_{n,m}\le \mcE_{n,m+1}\le \mcE_n \le \mcE$ for all $n,m\in\N$ and $S_{1,1}^a \supseteq S_{n,m}^a \supseteq S_{n,m+1}^a \supseteq S_{n}^a \supseteq S^a$. We choose $a:=\mcE(0)$ which ensures that $S^a \neq \emptyset$. The rest of the proof follows the lines of (a).
\end{enumerate}
\end{proof}

\Cref{Thm:GammaConvergence:Phigamma} shows that minimizers of the discrete problem 
\begin{equation}\label{Eq:ProblemHw:constrained:discrete:phi}
\begin{split}
\min_{u_\theta \in \mcHc{c}{M}} \frac{\alpha_1}{N} \sum_{x\in\Omega^h} |(T^h u_\theta(x) - g(x) | &+ \frac{\alpha_2}{N} \sum_{x\in\Omega^h} |(T^h u_\theta(x) - g(x) |^2  \\
&+ \frac{\lambda}{N} \sum_{x\in\Omega^h} \widetilde{\varphi}_\gamma(\nabla_{\text{FB}}^h u(x))
\end{split}
\end{equation}
with $\widetilde{\varphi}_\gamma \in\left\{|\cdot|_{2,1,\gamma}, \max\{|\cdot|_{2,1},\gamma\},\phi_{2,1,\gamma}(\cdot) \right\}$ 
converge to a minimizer of \eqref{Eq:ProblemTV}, while the functional in \eqref{Eq:ProblemHw:constrained:discrete:phi} $\Gamma$-converges to the functional in \eqref{Eq:ProblemTV}. 

\subsection{Connection to finite difference discretisation}

The results in \cref{Sec:Gammaconvergence:discrete} already indicate that there is a strong connection between \eqref{Eq:ProblemHw:constrained:discrete2} and the respective finite difference discretization of \eqref{Eq:ProblemTV}, namely
\begin{equation}\label{Eq:Problem:FDM}
\begin{split}
\min_{u_{\text{FD}} \in \R^{N}} \Big\{E_{\text{FD}}(u_{\text{FD}}) := \alpha_1\sum_{x\in\Omega^h}w_{1}^x |T^h u_{\text{FD}}(x) - g_{\text{FD}}(x) | &+ \alpha_2 \sum_{x\in\Omega^h}w_{2}^x|T^h u_{\text{FD}}(x) - g_{\text{FD}}(x) |^2  \\
&+ \lambda \sum_{x\in\Omega^h} w_{\TV}^x |\nabla^h u_{\text{FD}}(x)|_{\blacktriangleright}\Big\}, 
\end{split}
\end{equation}
where $g_{\text{FD}}\in\R^{N}$ is a finite difference discretisation of $g$ such that $g_{\text{FD}}(x) = g(x)$ for all $x\in\Omega^h$. In particular, if $\mathcal{A}_{h_n} \cap \mcHc{c_n}{M_n} = \mathcal{A}_{h_n}$ in \eqref{Eq:IntegralCounterpart}, then \eqref{Eq:IntegralCounterpart} is also the integral counterpart of \eqref{Eq:Problem:FDM}. Note that here as in \eqref{Eq:ProblemHw:constrained:discrete2} $\nabla^h$ and $|\cdot|_{\blacktriangleright}$ might be any suitable discrete gradient and discrete norm. For example, one may choose $\nabla_{\text{F}}^h$ and $|\cdot|_2$ or $\nabla_{\text{FB}}^h$ and $|\cdot|_{2,1}$.
Note that if $u_{\text{FD}}(x) = u_\theta(x)$ for all $x\in\Omega^h$, then $E_\theta(u_\theta) = E_{\text{FD}}(u_{\text{FD}})$. This observation leads to the following result:

\begin{proposition}\label{Prop:Discrete}
Given the set of points $\{x_i\}_{i=1}^{N}=\Omega^h \subset \Omega$. If $u_{\FD}\in\R^N$ is a solution of \eqref{Eq:Problem:FDM}, then there exists a constant $c\geq 0$ and an $M\in\N$ such that $u_\theta\in \mcHc{c}{M}$ is a solution of \eqref{Eq:ProblemHw:constrained:discrete2}, whereby  ${u}_{\text{FD}} (x_i) = u_\theta(x_i)$ for $i=1,\ldots,N$.
\end{proposition}
\begin{proof}
We construct a mesh $\mathcal{T}$ of simplices with vertices $\{x_i\}_{i=1}^{N}$ and let $u_{\FD}$ be a solution of \eqref{Eq:Problem:FDM}. Then there exists a $c\geq 0$ and an $M\in\N$ such that $u_\theta\in \mcHc{c}{M}$ is its continuous piecewise affine interpolation on $\mathcal{T}$, i.e.\ $u_{\text{FD}}(x_i) = u_\theta(x_i)$ for $i=1,\ldots,N$. Hence we have $E_\theta (u_\theta) = E_{\FD}(u_{\FD}) $.\\
Assume that $u_{\theta}$ is not a minimizer of $E_{\theta}$, i.e.\ there is a $\tilde{u}_{\theta}\in \mcHc{c}{M}$ with $\tilde{u}_{\theta}\neq {u}_{\theta}$ such that $E_{\theta}(\tilde{u}_{\theta}) < E_{\theta}(u_{\theta})$. We define $\tilde{u}_{\FD}\in\R^N$ as $\tilde{u}_{\FD}(x_i) = \tilde{u}_{\theta}(x_i)$ for $i=1,\ldots,N$. Hence we obtain $E_{\FD}(\tilde{u}_{\FD})  = E_\theta (\tilde{u}_\theta)<E_\theta (u_\theta) = E_{\FD}(u_{\FD})$, which is a contradiction showing that $u_{\theta}$  is indeed a minimizer of $E_{\theta}$.
\end{proof}

Note that for any $c\ge 0$ and $M\in\N$ \eqref{Eq:ProblemHw:constrained:discrete2} has a solution, see \cref{Prop:ExistenceMinimizer}, while this is not the case for \eqref{Eq:Problem:FDM}. This is why the other direction in \cref{Prop:Discrete} does not hold in general. Moreover, although \cite[Theorem 2.1]{ArBaMiMu:16} states that any continuous piecewise affine function can be represented by a ReLU-NN, it seems not possible to preset a $c\geq 0$ and $M\in\N$ such that $\mcHc{c}{M}$ contains any continuous piecewise affine functions, not even on a given grid.
However, if \eqref{Eq:Problem:FDM} has a solution then there is a $c\ge 0$ and an $M\in\N$, such that the sets of solutions of \eqref{Eq:Problem:FDM} and \eqref{Eq:ProblemHw:constrained:discrete2}, respectively, are equivalent in the sense of \cref{Prop:Discrete}.

If $c=\infty$, i.e.\ the constraint $|\theta|_\infty < c$ in \eqref{Eq:ProblemHw:constrained:discrete2} is dropped, and $M=\infty$, i.e. the number of neurons in the hidden layers can be arbitrary large and unbounded, then we even have the following statement.
\begin{proposition}\label{Prop:Discrete2}
Given the set of points $\{x_i\}_{i=1}^{N}=\Omega^h \subset \Omega$ and let $c=\infty$ and $M=\infty$ in \eqref{Eq:ProblemHw:constrained:discrete2}, then $u_{\FD}\in\R^N$ is a solution of \eqref{Eq:Problem:FDM}, if and only if $u_\theta\in \mcH$ is a solution of \eqref{Eq:ProblemHw:constrained:discrete2}, whereby  ${u}_{\text{FD}} (x_i) = u_\theta(x_i)$ for $i=1,\ldots,N$.
\end{proposition}

\begin{proof}
Let $\mathcal{T}$ be a mesh of simplices with vertices $\{x_i\}_{i=1}^{N}$.

``$\Rightarrow$'': This direction is shown as \cref{Prop:Discrete} by noting that $c=\infty=M$ here and any continuous piecewise affine interpolation is in $\mcH$.

``$\Leftarrow$'': Let $u_\theta$ be a solution of \eqref{Eq:ProblemHw:constrained:discrete2} with $c=\infty$ and $M=\infty$. We define $u_{\text{FD}}\in\R^N$ such that 
$u_{\text{FD}}(x_i) = u_\theta(x_i)$ for $i=1,\ldots,N$. This implies that $E_\theta (u_\theta) = E_{\FD}(u_{\FD}) $.\\
Assume that $u_{\FD}$ is not a minimizer of $E_{\FD}$, i.e.\ there is a $\tilde{u}_{\FD}\in\R^N$ with $\tilde{u}_{\FD}\neq {u}_{\FD}$ such that $E_{\FD}(\tilde{u}_{\FD}) < E_{\FD}(u_{\FD})$. We construct a continuous piecewise affine interpolation $\tilde{u}_\theta$ of $\tilde{u}_{\FD}$ on $\mathcal{T}$ such that $\tilde{u}_{\text{FD}} (x_i) = \tilde{u}_\theta(x_i)$ for $i=1,\ldots,N$. Since $\mcH$ includes any continuous piecewise affine function on $\mathcal{T}$, we have $\tilde{u}_\theta \in \mcH$. Consequently $E_\theta (\tilde{u}_\theta) = E_{\FD}(\tilde{u}_{\FD}) < E_{\FD}(u_{\FD}) = E_\theta (u_\theta)$, which is a contradiction to $u_\theta$ being a solution of \eqref{Eq:ProblemHw:constrained:discrete2} with $c=\infty=M$ and hence $u_{\FD}$ is indeed a minimizer of $E_{\FD}$.
\end{proof}
A direct consequence of \cref{Prop:Discrete2} is that \eqref{Eq:Problem:FDM} does not have a solution if and only if \eqref{Eq:ProblemHw:constrained:discrete2} with $c,M=\infty$ does not have a solution.

We emphasize that \cref{Prop:Discrete,Prop:Discrete2} also holds if $|\cdot|_{\blacktriangleright}$ is chosen to be any of the approximations presented in \cref{Sec:NumericalIssue}, i.e.\ Huber-regularization, $\gamma$-lifting and $\max$-$\gamma$-lifting, or any arbitrary discrete norm, e.g.\ $|\cdot|_r$, $r\in\N$. Actually, the proof of \cref{Prop:Discrete,Prop:Discrete2} does not rely on the type of discretization but only on the fact that if $u_{\text{FD}}(x) = u_\theta(x)$ for all $x\in\Omega^h$, then $E_\theta(u_\theta) = E_{\text{FD}}(u_{\text{FD}})$.

\cref{Prop:Discrete2} can be generalized in the following way.
\begin{theorem}\label{Thm:Equivalence}
Given the set of points $\{x_i\}_{i=1}^{N}=\Omega^h \subset \Omega$ we define
$$
\wmcH:=\{u: \Omega \to \R \colon \forall x,y\in \Omega^h,  x\neq y, u(x) \text{ and } u(y) \text{ are independent} \},
$$
e.g.\ $\wmcH$ could be the set of ReLU-NNs.
Then we have that 
$$
u_{\FD}^* \in \argmin_{u_{FD}\in \R^N} \mcE_{\FD} (u_{\FD}) \Longleftrightarrow u_{\theta}^*\in \argmin_{u_{\theta}\in \wmcH} \mcE_{\theta}(u_\theta)
$$
with ${u}_{\text{FD}}^* (x) = u^*_{\theta}(x)$ for all $x\in\Omega^h$, 
where $\mcE_{\FD} (u_{\FD}):= \sum_{x\in\Omega^h} \mcG(u_{\FD}(x))$, $\mcE_{\theta} (u_{\theta}):= \sum_{x\in\Omega^h} \mcG(u_{\theta}(x))$ and $\mcG: \R\to\overline{\R}$.
\end{theorem}
\begin{proof}
The proof is similar to the proof of \cref{Prop:Discrete,Prop:Discrete2}. Nevertheless, we state it for completeness. 

``$\Rightarrow$'': Let $u_{\FD}^* \in \R^N$ be any minimizer of $\mcE_{\FD}$. By the definition of $\wmcH$ there exists a $u_{\theta}^*\in \wmcH$ such that $u_{\FD}^*(x) = u_{\theta}^*(x)$ for all $x\in\Omega^h$ and hence $\mcE_{\FD} (u_{\FD}^*) = \mcE_{\theta}(u_{\theta}^*)$. Assume that there exists a $u_{\theta} \neq u_{\theta}^*$ such that $\mcE_{\theta}(u_{\theta})< \mcE_{\theta}(u_{\theta}^*)$. Then we can define $u_{\FD}\in\R^N$ such that $u_{\FD}(x) = u_{\theta}(x)$ for all $x\in\Omega^h$. This yields $  \mcE_{\FD} (u_{\FD})= \mcE_{\theta}(u_{\theta})< \mcE_{\theta}(u_{\theta}^*) = \mcE_{\FD} (u_{\FD}^*)$, which is a contradiction to the optimality of $u_{\FD}^*$.

``$\Leftarrow$'': Let $u_\theta^*\in\wmcH$ be a solution of $\mcE_\theta$. We define $u_{\text{FD}}^*\in\R^N$ such that 
$u_{\text{FD}}^*(x) = u^*_\theta(x)$ for all $x\in\Omega^h$. This implies that $\mcE_\theta (u^*_\theta) = \mcE_{\FD}(u_{\FD}^*) $. Assume that $u_{\FD}^*$ is not a minimizer of $\mcE_{\FD}$, i.e.\ there is a $\tilde{u}_{\FD}\in\R^N$ with $\tilde{u}_{\FD}\neq {u}_{\FD}^*$ such that $\mcE_{\FD}(\tilde{u}_{\FD}) < \mcE_{\FD}(u_{\FD}^*)$. We construct an interpolation $\tilde{u}_\theta\in \wmcH$ of $\tilde{u}_{\FD}$ such that $\tilde{u}_{\text{FD}} (x) = \tilde{u}_\theta(x)$ for all $x\in\Omega^h$. Consequently $\mcE_\theta (\tilde{u}_\theta) = \mcE_{\FD}(\tilde{u}_{\FD}) < \mcE_{\FD}(u_{\FD}^*) = \mcE_\theta (u^*_\theta)$, which is a contradiction to $u_{\theta}^*$ being a solution of $\mcE_\theta$ and hence $u_{\FD}^*$ is indeed a minimizer of $\mcE_{\FD}$.
\end{proof}
Note that $\wmcH$ in \cref{Thm:Equivalence} can be a piecewise polynomial space of any polynomial order and hence we could set $\wmcH = \mcH$ the set of ReLU-NNs as well. 
Let us use \cref{Thm:Equivalence} to connect physics-informed neural networks (PINNs) and the Deep Ritz Method to a finite difference method by a simple example. To this end we only need to specify the function $\mcG$ of \cref{Thm:Equivalence}. The equivalence of the respective solutions in the sense of \cref{Thm:Equivalence} follows directly from \cref{Thm:Equivalence}.
\begin{example}
Given the Poisson problem with Dirichlet boundary conditions 
\begin{alignat*}{2}
\Delta u &= g  &\quad &\text{in } \Omega,\\
u &= g_{D} &\quad &\text{on } \partial \Omega.
\end{alignat*}
\begin{enumerate}[(a)]
\item \emph{PINNs:} $\mcG(u(x)) = w^x |\Delta^h u(x) - g(x)|^2$, where $\Delta^h$ is a finite difference discretization of the Laplace operator with incorporated boundary conditions and $w^x$ some suitable (quadrature) weights.

\item \emph{Deep Ritz Method:} $\mcG(u(x)) = w^x \left(\frac{1}{2}|\nabla^h u(x)|^2 - g(x)u(x)\right)$, where $\nabla^h$ is a finite difference discretization of the gradient operator with incorporated boundary conditions and $w^x$ some suitable (quadrature) weights.
\end{enumerate}

\end{example}

\Cref{Thm:Equivalence} can also be related to supervised learning as the following examples illustrates.
\begin{example}[Supervised learning]
Let $(x_i,g(x_i))_{i=1}^N$ be given training data, then $\mcG$ mimics the empiric loss function. E.g.\ in the case of the square loss we have $\mcG(u(x_i)) = \frac{1}{N} |u(x_i) - g(x_i)|^2$.
\end{example}

\section{Numerical experiments}\label{Sec:NumericalExperiments}

In the following we show numerical experiments supporting our theoretical findings and demonstrating that the model \eqref{Eq:ProblemHw:constrained:discrete2} can indeed be used to solve image processing tasks such as denoising, inpainting and deblurring as expected from total variation minimization. We recall that our implementation is done in Python using \tf\ and can be found at \cite{LangerBehnamianCode:24}.
To obtain a solution, we utilize the optimization algorithm Adam, which is already implemented in \tf. 
Thereby in each iteration step we check whether the loss, i.e.\ the functional in \eqref{Eq:ProblemHw:constrained:discrete2} evaluated at the current iterate, is reduced, i.e.\ is smaller than the loss of all previous iterates, and only update the approximation $u_\theta$ if this is the case. In this way we ensure that the current iterate is always the yet best approximation, i.e.\ the approximation which gives the smallest loss so far, leading to a monotonic decreasing sequence of energies.  

In all our experiments we set $w_q^x = \frac{1}{N}$ for all $x\in\Omega^h$ and $q\in \{1,2,\TV\}$.
Although a comparison of different quadratures might be of interest, as this may have influence on the solution and solution process, it is beyond the scope of the present paper. 

\subsection{Convergence to analytic solution}

In this section we consider two illustrative examples, one in one-dimension ($d=1$) and one in two-dimensions ($d=2$). These examples support our theoretical findings from above, demonstrating the $\Gamma$-convergence and the convergence of minimizers to the solution of the original problem.

 \subsubsection{1D example}\label{Sec:1DExample}

Let $\Omega = [a,b]\subset \R$ and consider the step function
$$
g(x) = \begin{cases}
0 & \text{if } x\in[a,x_1],\\
1 & \text{if } x\in(x_1,b],
\end{cases} 
$$
where $a < x_1 < b$. The exact minimizer of \eqref{Eq:ProblemTV} with $T=I$ and $\lambda=1$ is given in \cref{AppendixA}.
Here we set $a=0$, $x_1=1$, $b=2$, $\alpha_1 = 0.5$, $\alpha_2 = 1.25$ and $\lambda=1$. Then the solution is calculated as
$$
u^*(x) = \begin{cases}
0.2 & \text{if } x\in[a,x_1],\\
0.8 & \text{if } x\in(x_1,b],
\end{cases}
$$
yielding $E(u^*) = 0.9$, where $E$ is the functional in \eqref{Eq:ProblemTV}. To numerically compute a solution of \eqref{Eq:ProblemTV} within this setting we consider \eqref{Eq:ProblemHw:constrained:discrete2} with $N=1000$, $\nabla^h=\nabla_F^h$ being the forward difference gradient with homogeneous Neumann boundary conditions and $|\cdot|_\blacktriangleright = |\cdot|$. 
The neural network architecture, i.e.\ the solution space $\mcH$ is specified as follows: any network consists of 2 hidden layers, whereby the first hidden layer has 64 neurons and the second hidden layer has 128 neurons. We use Adam with learning rate 0.01 and numerically compute for $c\in\{0,0.1,0.5,1,10,100\}$ a solution of \eqref{Eq:ProblemHw:constrained:discrete2}, whereby the respective numerically computed solution is denote by $u_c$. The overall algorithm is terminated after 100001 iterations.

We observe that with increasing $c$ the energy $E_\theta(u_c)$ indeed approaches $E(u^*)$ as expected from our $\Gamma$-convergence result, see \cref{fig:1d:Gammaconvergence}, and the approximation $u_c$ approaches the exact minimizer $u^*$, see \cref{fig:1d:Minimizer}. In fact in this particular example $c=100$ is sufficient to obtain the analytic solution (within rounding-errors), see \cref{fig:1d:Gammaconvergence:signal}.
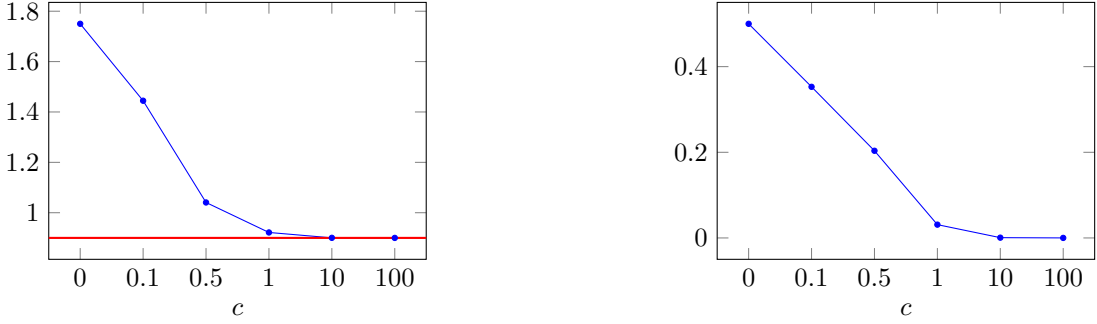
\begin{figure}[h]
\centering
\begin{subfigure}[t]{0.43\columnwidth}
\centering
\begin{tikzpicture}
\begin{axis}[
    xlabel={$c$},
    legend style={at={(1.05,1)},anchor=north west},
    width=6.6cm, height=5cm,
    xtick={0, 1, 2, 3, 4, 5},
    xticklabels={$0$, $0.1$, $0.5$, $1$, $10$, $100$},
    xmax=5.5,
    xmin=-0.5,
]
\addplot[ color=blue,
    mark=*,
    mark size=1pt
] coordinates {(0, 1.75) (1, 1.444625735282898) (2, 1.0406935214996338) (3, 0.9216306209564209) (4, 0.9003081917762756) (5, 0.8999999761581421)};
\addplot[red, thick] coordinates {(-0.5, 0.9) (5.5, 0.9)};
\end{axis}
\end{tikzpicture}
\caption{Energy $E_\theta(u_c)$ for different values $c$. The red line depicts the energy of the analytic solution $E(u^*)=0.9$.}\label{fig:1d:Gammaconvergence}
\end{subfigure}\hfill
\begin{subfigure}[t]{0.43\columnwidth}
\centering
\begin{tikzpicture}
\begin{axis}[
    xlabel={$c$},
    legend style={at={(1.05,1)},anchor=north west},
    width=6.6cm, height=5cm,
    xtick={0, 1, 2, 3, 4, 5},
    xticklabels={$0$, $0.1$, $0.5$, $1$, $10$, $100$},
    xmax=5.5,
    xmin=-0.5,
]
\addplot[ color=blue,
    mark=*,
    mark size=1pt
] coordinates {(0, 500.00006103515625/1000) (1, 352.8641052246094/1000) (2, 203.55564880371094/1000) (3, 31.07950782775879/1000) (4, 0.6743378639221191/1000) (5, 0.0824853777885437/1000)};
\end{axis}
\end{tikzpicture}
\caption{Distance between $u_c$ and $u^*$, i.e.\ $\frac{1}{N}\sum_{x\in\Omega^h} |u_c(x)-u^*(x)|$, for different values $c$.}\label{fig:1d:Minimizer}
\end{subfigure}
\caption{Visualizing $\Gamma$-convergence and the convergence of minimizers to the analytic solution for $d=1$.}
\end{figure}

\graphicspath{{./Scripts/1d/test_boundedweights/denoising/layers_64-128_acts_relu-relu_l2_0/ruleAR/learning_rate0.01/}}

\begin{figure}[h]
  \centering
  \begin{subfigure}[b]{0.3\textwidth}
        \centering
  \includegraphics[width=\linewidth]{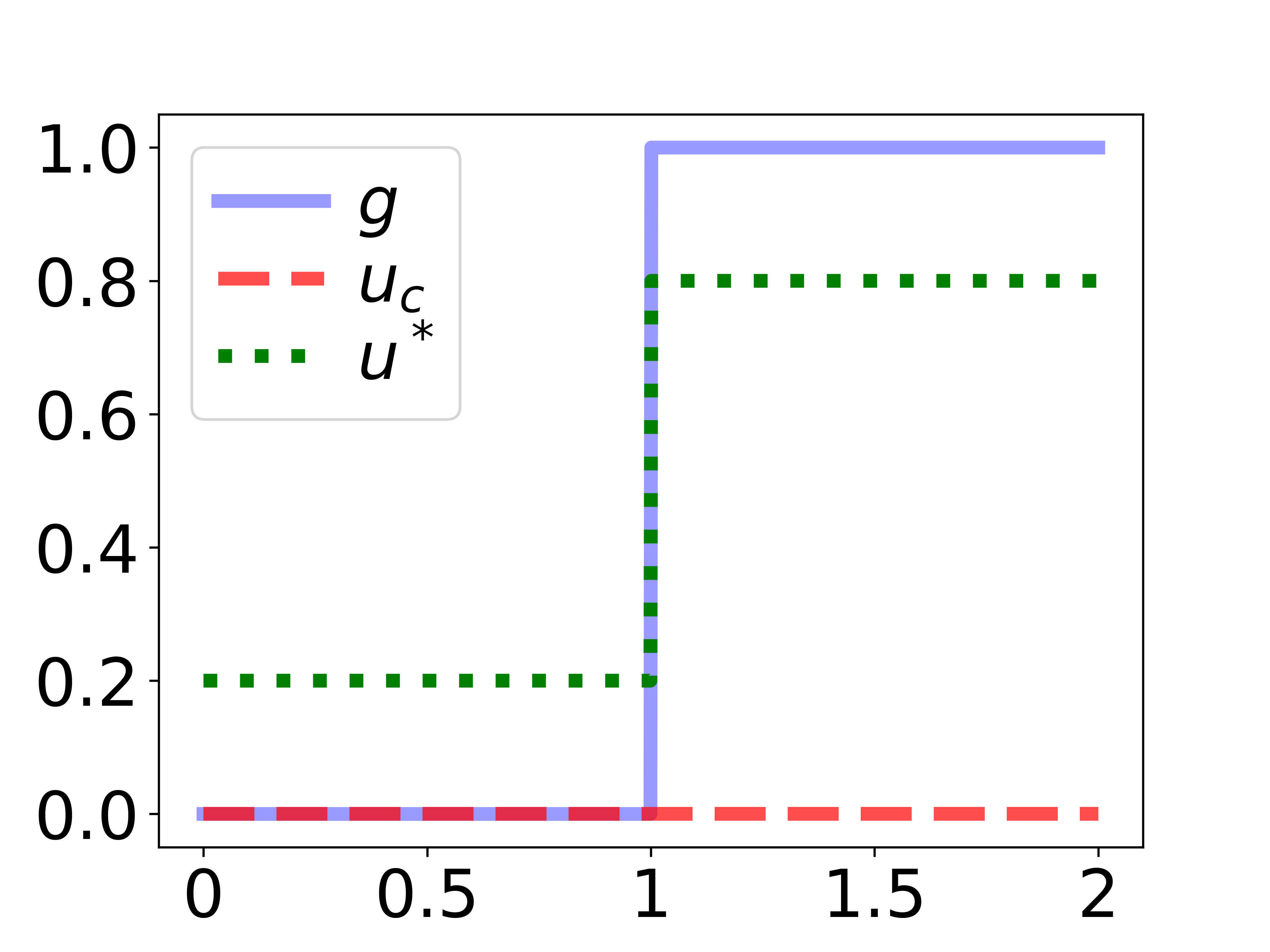}
  \caption{$c=0$}
    \end{subfigure}
    \begin{subfigure}[b]{0.3\textwidth}
        \centering
  \includegraphics[width=\linewidth]{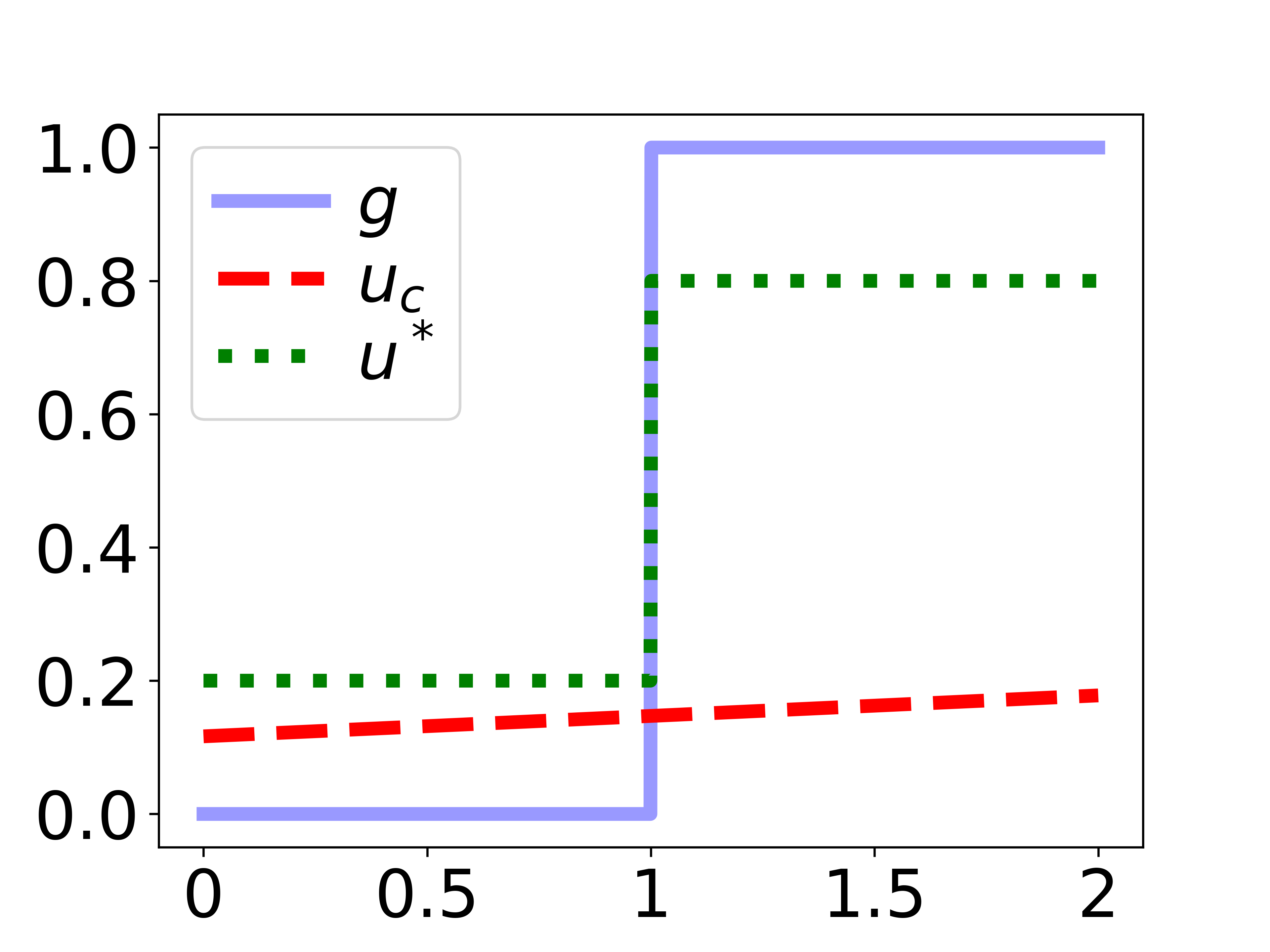}
  \caption{$c=0.1$}
    \end{subfigure}
    \begin{subfigure}[b]{0.3\textwidth}
        \centering
  \includegraphics[width=\linewidth]{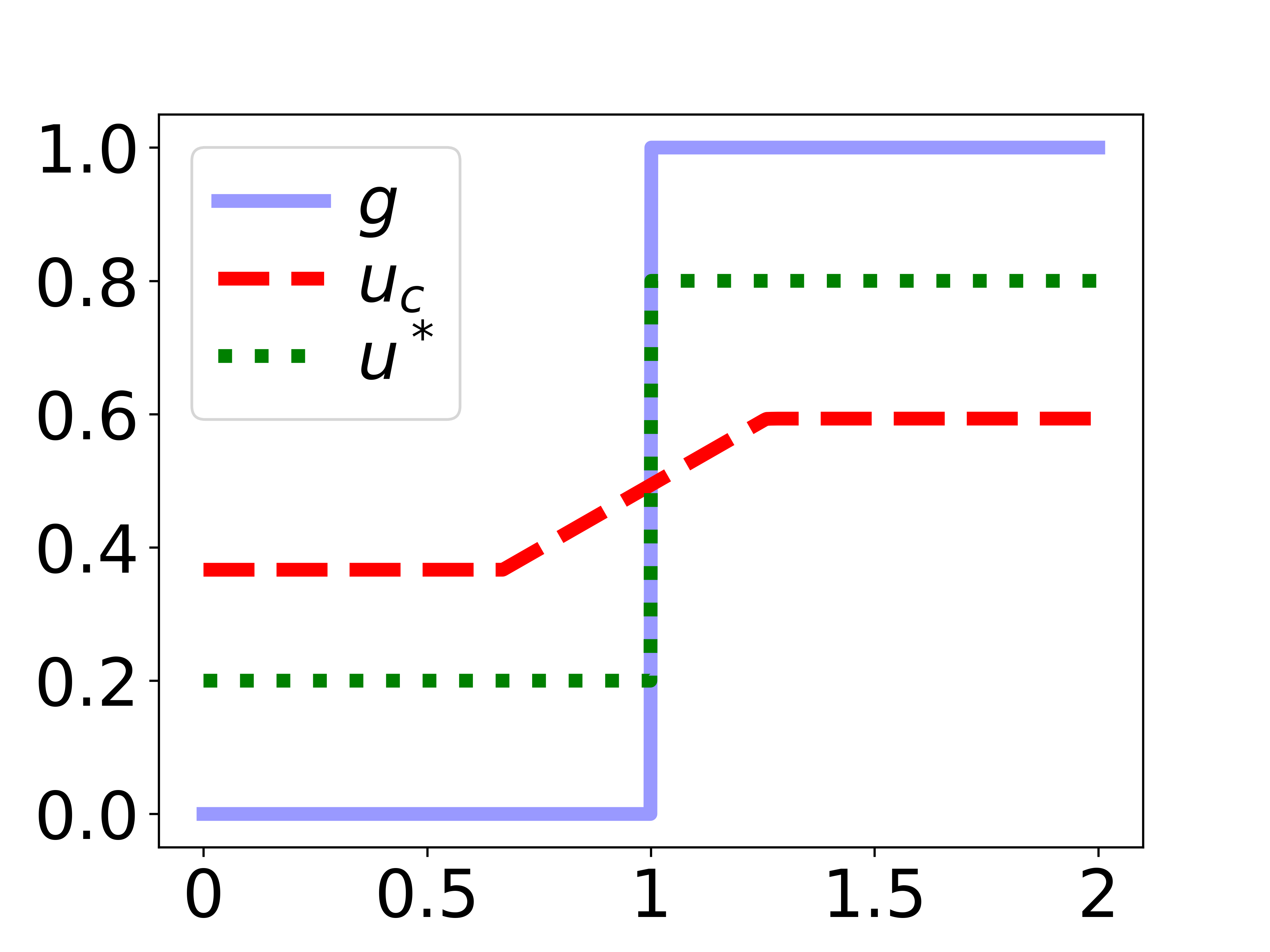}
  \caption{$c=0.5$}
    \end{subfigure}
    \begin{subfigure}[b]{0.3\textwidth}
        \centering
  \includegraphics[width=\linewidth]{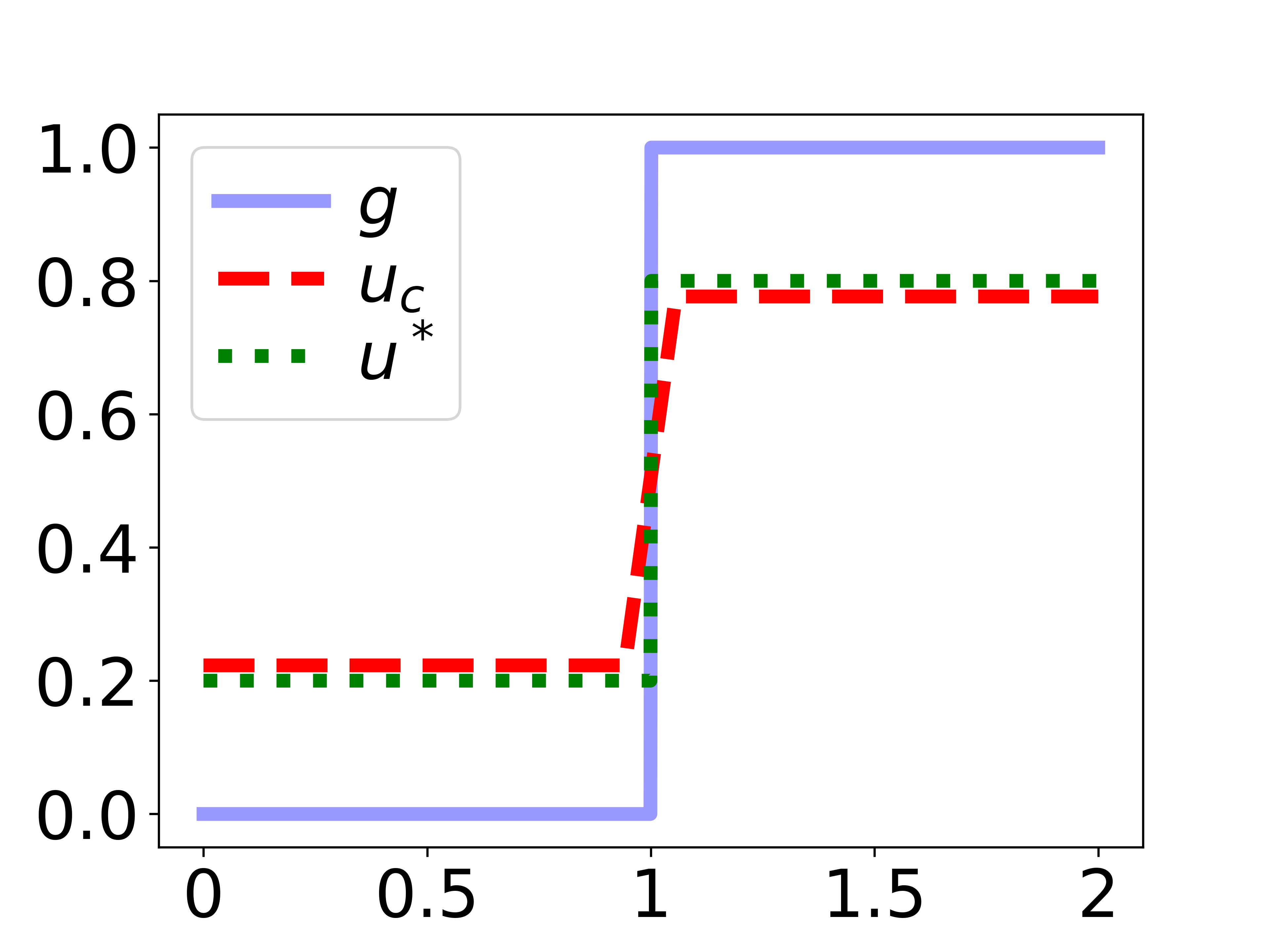}
  \caption{$c=1$}
    \end{subfigure}
    \begin{subfigure}[b]{0.3\textwidth}
        \centering
  \includegraphics[width=\linewidth]{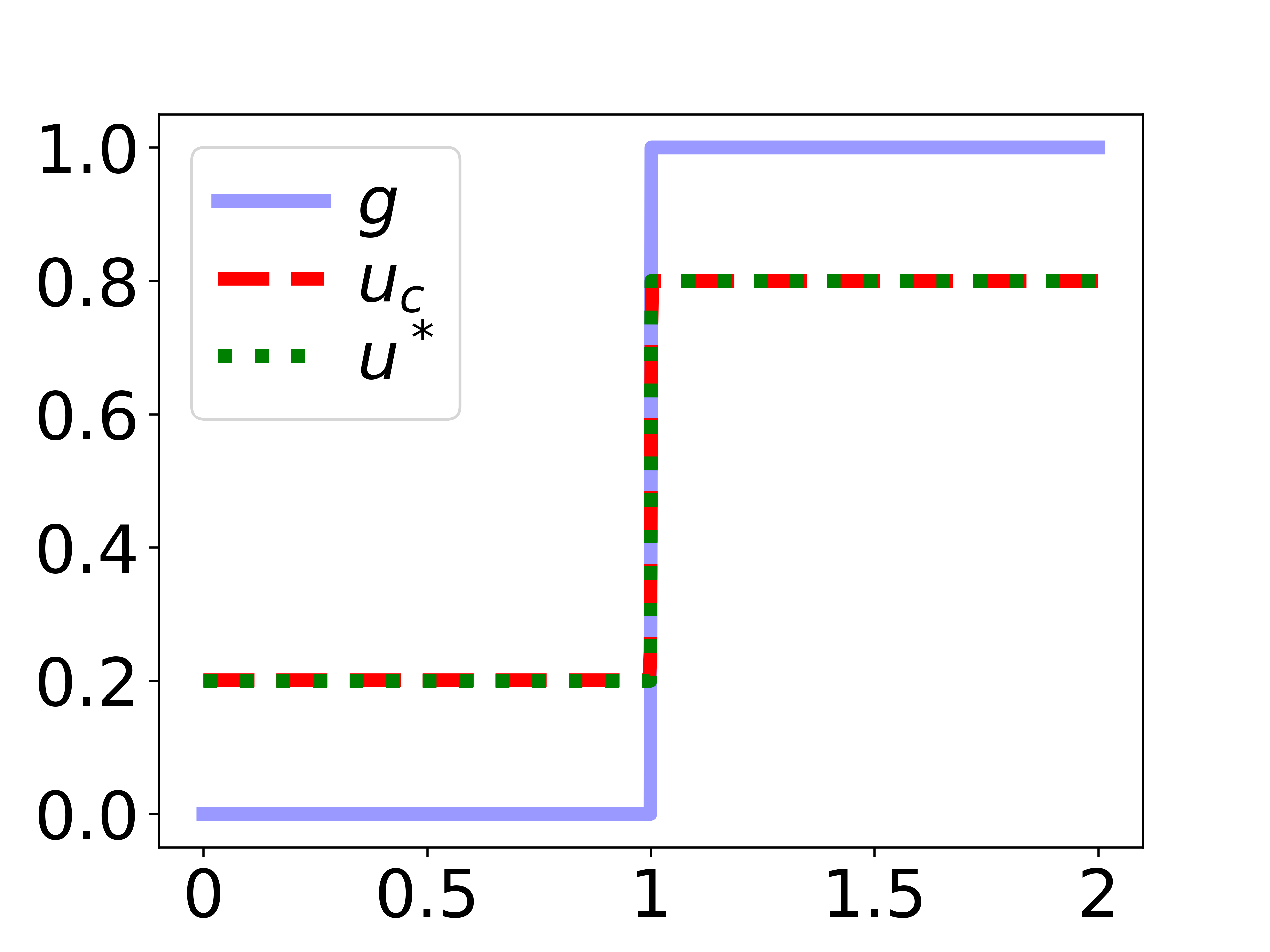}
  \caption{$c=10$}
    \end{subfigure}
    \begin{subfigure}[b]{0.3\textwidth}
        \centering
  \includegraphics[width=\linewidth]{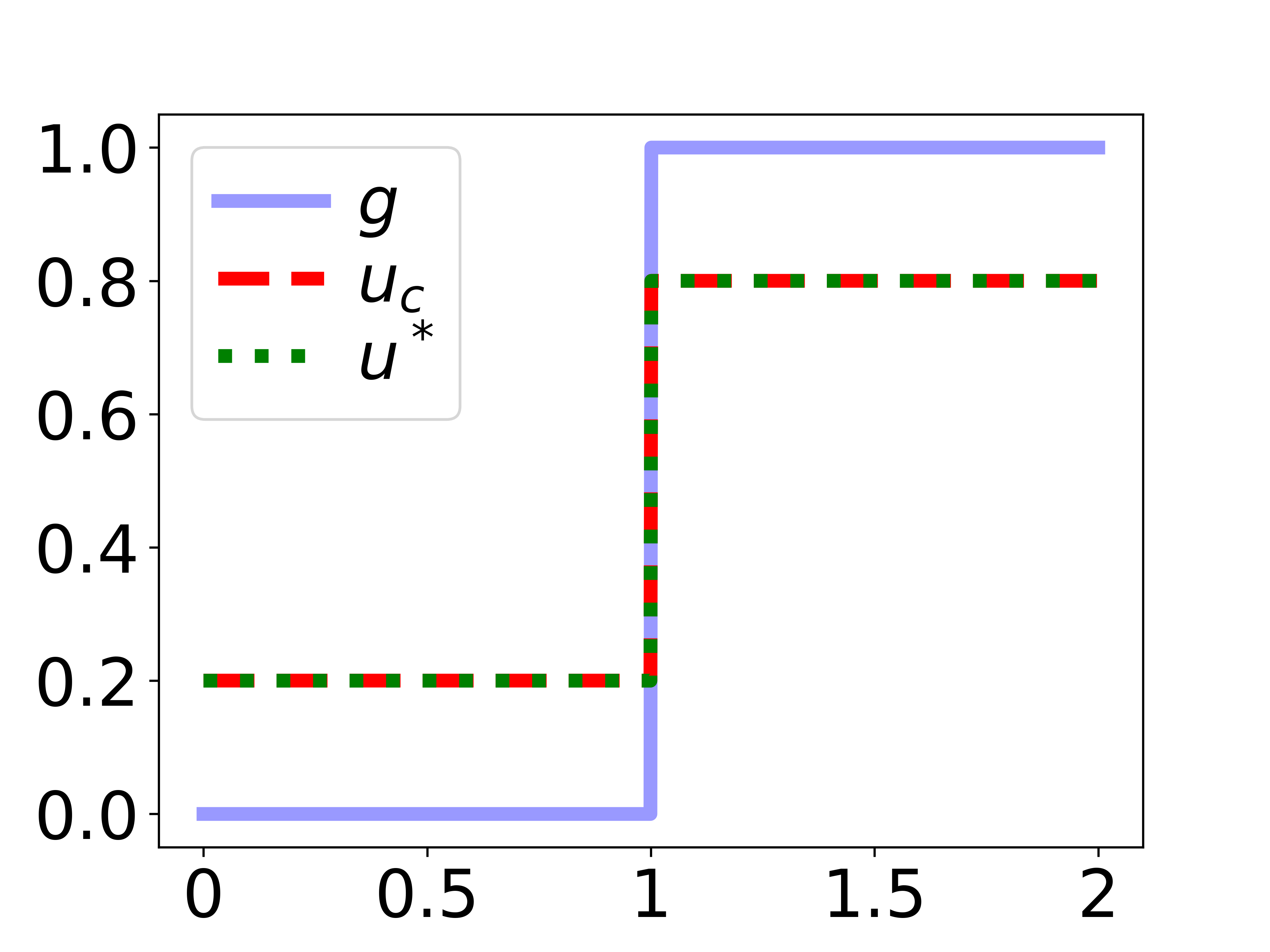}
  \caption{$c=100$}
    \end{subfigure}
  \caption{
    Neural network approximations $u_c$ with different values of $c$, exact solution $u^*$ and given data $g$.
  }
  \label{fig:1d:Gammaconvergence:signal}
\end{figure}

Althought the universal approximation results in \cref{Cor:UniversalApprox,Cor:UniversalApprox:BV} state that for $d=1$ only one hidden layer is needed to approximate any $BV$-function arbitrarily exact, we choose 2 hidden layers in our experiment. The reasons for this are, that it is not clear how many neurons one needs in the hidden layer and the optimization process seems to find easier a solution when 2 hidden layers are used. Note that the optimizer is stochastic and may get stuck at certain local minima, which are far away from the global solution. Indeed we observe this when only one hidden layer is used. For a similar reason we use 3 hidden layers in our experiments for $d=2$, which we will present next.

\subsubsection{2D example}\label{Sec:2DExample}

Let $g=\mathbf{1}_{B_{\radius}(0)}$ be a characteristic function of the disk $B_{\radius}(0)$ centered at the origin with radius $\radius>0$. The explicit solution of \eqref{Eq:ProblemTV} with $\Omega = \R^2$ and $T=I$ is given as
$$
u = \begin{cases}
0 & \text{if } 0 \leq \radius < \frac{2 \lambda}{2\alpha_2 + \alpha_1},\\
\left( \frac{2\alpha_2 + \alpha_1}{2\alpha_2} - \frac{\lambda}{\alpha_2 \radius}\right) \mathbf{1}_{B_{\radius}(0)} & \text{if }  \frac{2 \lambda}{2\alpha_2 + \alpha_1}\leq \radius \leq \frac{2 \lambda}{\alpha_1},\\
\mathbf{1}_{B_{\radius}(0)} & \text{if } \radius > \frac{2\lambda }{\alpha_1},
\end{cases}
$$ 
cf.\ \cite{HintermullerLanger:13}. We choose $\alpha_1=1$, $\alpha_2=7$, $\lambda=1$ and ${\radius}=\tfrac{1}{4}$. This yields the analytic solution $u^*= \tfrac{1}{2} \mathbf{1}_{B_{\radius}(0)}$ and the energy $E(u^*)\approx 1.227184630308513$.

To numerically compute a solution of \eqref{Eq:ProblemTV} within this setting we consider \eqref{Eq:ProblemHw:constrained:discrete2} with 
$$
|\cdot|_{\blacktriangleright}\in \left\{|\cdot|_{2,\gamma},|\cdot|_{2,1,\gamma},\max\{|\cdot|_2,\gamma\},|\cdot|_{2,1,\max_\gamma},\phi_{2,\gamma}(\cdot),\phi_{2,1,\gamma}(\cdot)\right\}
$$
 instead of $|\cdot|_{2}$ and $|\cdot|_{2,1}$, which did not work as mentioned above. We set $\nabla^h = \nabla_{\text{F}}^h$ if we use $|\cdot|_{2,\gamma},\max\{|\cdot|_2,\gamma\}$ or $\phi_{2,\gamma}(\cdot)$, while we set $\nabla^h = \nabla_{\text{FB}}^h$ if we use $|\cdot|_{2,1,\gamma},|\cdot|_{2,1,\max_\gamma}$ or $\phi_{2,1,\gamma}(\cdot)$. In all these cases we choose $\gamma=10^{-10}$.
Moreover, in $\nabla^h$ homogeneous Dirichlet boundary conditions are incorporated to mimic $\Omega=\R^2$  and that the solution is 0 outside the disk, while numerically the domain $[0,1]^2$ is chosen.
The neural network architecture, i.e.\ the solution space $\mcH$ is specified as follows: the network consists of 3 hidden layers while each layer has 128 neurons. Further we use Adam with learning rate 0.001 and the overall algorithm is terminated after 300001 iterations. With this specifications we compute solutions $u_c$ for pairs $(N,c)\in\{(33\times 33,0), (65\times 65,1),(129\times 129, 10),(257\times 257, 100),(513\times 513, 1000),(1025\times 1025, 10000)\}$ to investigate numerically the $\Gamma$-convergence. 
In \cref{fig:2d:Gammaconvergence} we depict the progress of the energy with increasing $(N,c)$ and in \cref{fig:2d:Gammaconvergence:Min} the difference between the computed approximation and the exact solution projected on the respective grid is presented. We clearly see that with increasing $N$ and $c$ in all considered cases the energy approaches $E(u^*)$, while the numerical approximations go towards the exact minimizer. The respective obtained numerical solutions are  presented in \cref{fig:GammaConvergence:2d:Reconstructions}. 
We observe that for all considered types of TV-smoothings, i.e.\ Huber-regularization, $\gamma$-lifting and $\max$-$\gamma$-lifting, $\TV_2^h$ and $\TV_{2,1}^h$ with $(N,c)=(1025\times 1025, 10000)$ obtain a rather close approximation of the analytic exact solution. Actually all these approximated discrete total variations seem to behave very similarly. 
Note that the convergence results from \cref{Thm:Gamma:Convergence:discrete,Thm:GammaConvergence:lifted} only cover $\TV_{2,1}^h$ and its smoothed versions. However, our numerical experiments suggest that also for $\TV_2^h$ and its smoothed versions a similar $\Gamma$-convergence result may hold.

\newcommand{\figgammatwod}{0.3\columnwidth}
\newcommand{\scaletikz}{0.7}
\begin{figure}
\centering
\begin{subfigure}[t]{\figgammatwod}
\centering
\begin{tikzpicture}[scale=\scaletikz]
\begin{axis}[
    xlabel={$c$},
    legend style={at={(1.05,1)},anchor=north west},
    width=6.6cm, height=5cm,
    xtick={0, 1, 2, 3, 4,5},
xticklabels={$0$,$10^0$, $10^1$, $10^2$, $10^3$, $10^4$},
    xmax=5.5,
    xmin=-0.5,
]
\addplot[ color=blue,
    mark=*,
    mark size=1pt
] coordinates {(0,1.6235077381134033) (1, 1.279463529586792) (2, 1.2584868669509888) (3, 1.2452771663665771) (4, 1.2399637699127197) (5, 1.235950231552124)};
\addplot[red, thick] coordinates {(-0.5, 1.227184630308513) (5.5, 1.227184630308513)};
\end{axis}
\end{tikzpicture}
\caption{$\TV_2^h$ with Huber-regularization and $\gamma=10^{-10}$.}
\end{subfigure}\hfill
\begin{subfigure}[t]{\figgammatwod}
\centering
\begin{tikzpicture}[scale=\scaletikz]
\begin{axis}[
    xlabel={$c$},
    legend style={at={(1.05,1)},anchor=north west},
    width=6.6cm, height=5cm,
    xtick={0, 1, 2, 3, 4,5},
xticklabels={$0$,$10^0$, $10^1$, $10^2$, $10^3$, $10^4$},
    xmax=5.5,
    xmin=-0.5,
]
\addplot[ color=blue,
    mark=*,
    mark size=1pt
] coordinates {(0,1.6235177516937256)(1, 1.2776644229888916) (2, 1.2583712339401245) (3, 1.2453627586364746) (4, 1.2396650314331055) (5, 1.2358920574188232)};
\addplot[red, thick] coordinates {(-0.5, 1.227184630308513) (5.5, 1.227184630308513)};
\end{axis}
\end{tikzpicture}
\caption{$\TV_2^h$ with $\gamma$-lifting and $\gamma=10^{-10}$.}
\end{subfigure}\hfill
\begin{subfigure}[t]{\figgammatwod}
\centering
\begin{tikzpicture}[scale=\scaletikz]
\begin{axis}[
    xlabel={$c$},
    legend style={at={(1.05,1)},anchor=north west},
    width=6.6cm, height=5cm,
    xtick={0, 1, 2, 3, 4,5},
xticklabels={$0$,$10^0$, $10^1$, $10^2$, $10^3$, $10^4$},
    xmax=5.5,
    xmin=-0.5,
]
\addplot[ color=blue,
    mark=*,
    mark size=1pt
] coordinates {(0,1.6235077381134033)(1, 1.2779808044433594) (2, 1.2584245204925537) (3, 1.24569833278656) (4, 1.2399111986160278) (5, 1.2359259128570557)};
\addplot[red, thick] coordinates {(-0.5, 1.227184630308513) (5.5, 1.227184630308513)};
\end{axis}
\end{tikzpicture}
\caption{$\TV_2^h$ with $\max$-$\gamma$-lifting and $\gamma=10^{-10}$.}
\end{subfigure}\hfill
\begin{subfigure}[t]{\figgammatwod}
\centering
\begin{tikzpicture}[scale=\scaletikz]
\begin{axis}[
    xlabel={$c$},
    legend style={at={(1.05,1)},anchor=north west},
    width=6.6cm, height=5cm,
    xtick={0, 1, 2, 3, 4,5},
xticklabels={$0$,$10^0$, $10^1$, $10^2$, $10^3$, $10^4$},
    xmax=5.5,
    xmin=-0.5,
]
\addplot[ color=blue,
    mark=*,
    mark size=1pt
] coordinates {(0,1.6235077381134033)(1, 1.3024890422821045) (2, 1.2594045400619507) (3, 1.2456555366516113) (4, 1.2398871183395386) (5, 1.2359402179718018)};
\addplot[red, thick] coordinates {(-0.5, 1.227184630308513) (5.5, 1.227184630308513)};
\end{axis}
\end{tikzpicture}
\caption{$\TV_{2,1}^h$ with Huber-regularization and $\gamma=10^{-10}$.}
\end{subfigure}\hfill
\begin{subfigure}[t]{\figgammatwod}
\centering
\begin{tikzpicture}[scale=\scaletikz]
\begin{axis}[
    xlabel={$c$},
    legend style={at={(1.05,1)},anchor=north west},
    width=6.6cm, height=5cm,
    xtick={0, 1, 2, 3, 4,5},
xticklabels={$0$,$10^0$, $10^1$, $10^2$, $10^3$, $10^4$},
    xmax=5.5,
    xmin=-0.5,
]
\addplot[ color=blue,
    mark=*,
    mark size=1pt
] coordinates {(0,1.6235177516937256)(1, 1.2757240533828735) (2, 1.2593684196472168) (3, 1.2457990646362305) (4, 1.2398912906646729) (5, 1.235342025756836)};
\addplot[red, thick] coordinates {(-0.5, 1.227184630308513) (5.5, 1.227184630308513)};
\end{axis}
\end{tikzpicture}
\caption{$\TV_{2,1}^h$ with $\gamma$-lifting and $\gamma=10^{-10}$.}
\end{subfigure}\hfill
\begin{subfigure}[t]{\figgammatwod}
\centering
\begin{tikzpicture}[scale=\scaletikz]
\begin{axis}[
    xlabel={$c$},
    legend style={at={(1.05,1)},anchor=north west},
    width=6.6cm, height=5cm,
    xtick={0, 1, 2, 3, 4,5},
xticklabels={$0$,$10^0$, $10^1$, $10^2$, $10^3$, $10^4$},
    xmax=5.5,
    xmin=-0.5,
]
\addplot[ color=blue,
    mark=*,
    mark size=1pt
] coordinates {(0,1.6235077381134033)(1, 1.2747306823730469) (2, 1.2593450546264648) (3, 1.24570631980896) (4, 1.2397031784057617) (5, 1.2353105545043945)};
\addplot[red, thick] coordinates {(-0.5, 1.227184630308513) (5.5, 1.227184630308513)};
\end{axis}
\end{tikzpicture}
\caption{$\TV_{2,1}^h$ with $\max$-$\gamma$-lifting and $\gamma=10^{-10}$.}
\end{subfigure}
\caption{Energy $E_\theta(u_c)$ for different values $c$ and $N$, where $N$ depends on $c$ such that $N(c) = 33 \times 33$ if $c=0$ and $N(c) = (2^{6+\log_{10}c}+1)\times (2^{6+\log_{10}c}+1)$ otherwise. The red line depicts the energy of the analytic solution $E(u^*)\approx 1.227184630308513$.}\label{fig:2d:Gammaconvergence}
\end{figure}
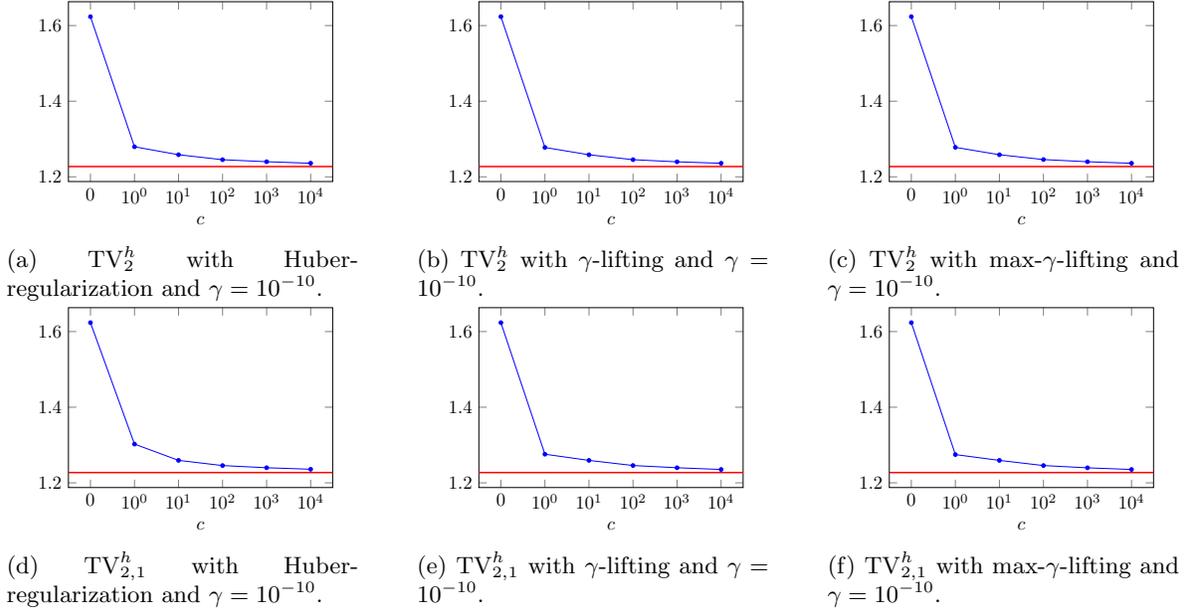

\renewcommand{\scaletikz}{0.63}
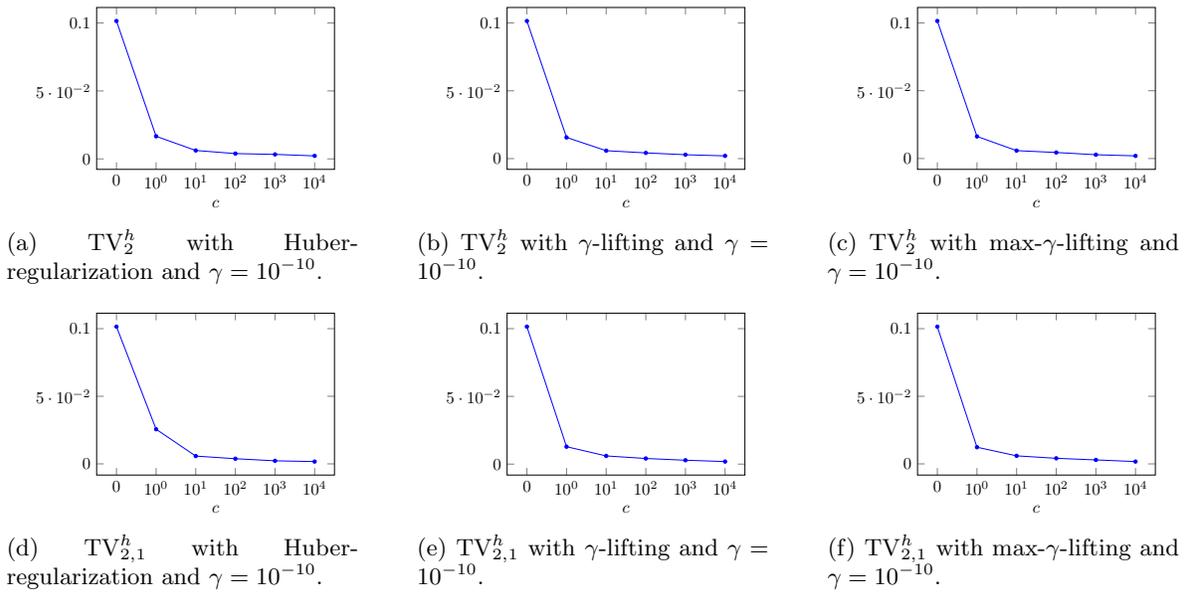
\begin{figure}
\centering
\begin{subfigure}[t]{\figgammatwod}
\centering
\begin{tikzpicture}[scale=\scaletikz]
\begin{axis}[
    xlabel={$c$},
    legend style={at={(1.05,1)},anchor=north west},
    width=6.6cm, height=5cm,
    xtick={0, 1, 2, 3, 4,5},
xticklabels={$0$,$10^0$, $10^1$, $10^2$, $10^3$, $10^4$},
    xmax=5.5,
    xmin=-0.5,
]
\addplot[ color=blue,
    mark=*,
    mark size=1pt
] coordinates {(0,0.1014692378328742) (1, 0.016622495030510356) (2, 0.006187072647541276) (3, 0.003897031749895201) (4, 0.0032986516050077367) (5, 0.0022089818188578227)};
\end{axis}
\end{tikzpicture}
\caption{$\TV_2^h$ with Huber-regularization and $\gamma=10^{-10}$.}
\end{subfigure}\hfill
\begin{subfigure}[t]{\figgammatwod}
\centering
\begin{tikzpicture}[scale=\scaletikz]
\begin{axis}[
    xlabel={$c$},
    legend style={at={(1.05,1)},anchor=north west},
    width=6.6cm, height=5cm,
    xtick={0, 1, 2, 3, 4,5},
xticklabels={$0$,$10^0$, $10^1$, $10^2$, $10^3$, $10^4$},
    xmax=5.5,
    xmin=-0.5,
]
\addplot[ color=blue,
    mark=*,
    mark size=1pt
] coordinates {(0,0.1014692378328742)(1, 0.01562829011996117) (2, 0.005859361245908543) (3, 0.004226841434425105) (4, 0.002879712094313682) (5, 0.002032399102096966)};
\end{axis}
\end{tikzpicture}
\caption{$\TV_2^h$ with $\gamma$-lifting and $\gamma=10^{-10}$.}
\end{subfigure}\hfill
\begin{subfigure}[t]{\figgammatwod}
\centering
\begin{tikzpicture}[scale=\scaletikz]
\begin{axis}[
    xlabel={$c$},
    legend style={at={(1.05,1)},anchor=north west},
    width=6.6cm, height=5cm,
    xtick={0, 1, 2, 3, 4,5},
xticklabels={$0$,$10^0$, $10^1$, $10^2$, $10^3$, $10^4$},
    xmax=5.5,
    xmin=-0.5,
]
\addplot[ color=blue,
    mark=*,
    mark size=1pt
] coordinates {(0,0.1014692378328742)(1, 0.016289532001201926) (2, 0.005832074503866814) (3, 0.004452151435244809) (4, 0.0028017511386605417) (5, 0.0019566922125966685)};
\end{axis}
\end{tikzpicture}
\caption{$\TV_2^h$ with $\max$-$\gamma$-lifting and $\gamma=10^{-10}$.}
\end{subfigure}\\ \vspace{0.3cm}

\begin{subfigure}[t]{\figgammatwod}
\centering
\begin{tikzpicture}[scale=\scaletikz]
\begin{axis}[
    xlabel={$c$},
    legend style={at={(1.05,1)},anchor=north west},
    width=6.6cm, height=5cm,
    xtick={0, 1, 2, 3, 4,5},
xticklabels={$0$,$10^0$, $10^1$, $10^2$, $10^3$, $10^4$},
    xmax=5.5,
    xmin=-0.5,
]
\addplot[ color=blue,
    mark=*,
    mark size=1pt
] coordinates {(0,0.1014692378328742)(1, 0.02562817093888684) (2, 0.005767988690131631) (3, 0.0038068711556417867) (4, 0.002254551037145451) (5, 0.0017059728909503272)};
\end{axis}
\end{tikzpicture}
\caption{$\TV_{2,1}^h$ with Huber-regularization and $\gamma=10^{-10}$.}
\end{subfigure}\hfill
\begin{subfigure}[t]{\figgammatwod}
\centering
\begin{tikzpicture}[scale=\scaletikz]
\begin{axis}[
    xlabel={$c$},
    legend style={at={(1.05,1)},anchor=north west},
    width=6.6cm, height=5cm,
    xtick={0, 1, 2, 3, 4,5},
xticklabels={$0$,$10^0$, $10^1$, $10^2$, $10^3$, $10^4$},
    xmax=5.5,
    xmin=-0.5,
]
\addplot[ color=blue,
    mark=*,
    mark size=1pt
] coordinates {(0,0.1014692378328742)(1, 0.012857675947381195) (2, 0.006043060891409489) (3, 0.004152226327855314) (4, 0.00286297415319386) (5, 0.0018826054524836407)};
\end{axis}
\end{tikzpicture}
\caption{$\TV_{2,1}^h$ with $\gamma$-lifting and $\gamma=10^{-10}$.}
\end{subfigure}\hfill
\begin{subfigure}[t]{\figgammatwod}
\centering
\begin{tikzpicture}[scale=\scaletikz]
\begin{axis}[
    xlabel={$c$},
    legend style={at={(1.05,1)},anchor=north west},
    width=6.6cm, height=5cm,
    xtick={0, 1, 2, 3, 4,5},
xticklabels={$0$,$10^0$, $10^1$, $10^2$, $10^3$, $10^4$},
    xmax=5.5,
    xmin=-0.5,
]
\addplot[ color=blue,
    mark=*,
    mark size=1pt
] coordinates {(0,0.1014692378328742)(1, 0.012331701876849113) (2, 0.005954733033011463) (3, 0.004123452043682768) (4, 0.0029407634852301415) (5, 0.0017255509648274837)};
\end{axis}
\end{tikzpicture}
\caption{$\TV_{2,1}^h$ with $\max$-$\gamma$-lifting and $\gamma=10^{-10}$.}
\end{subfigure}
\caption{Distance between $u_c$ and $u^*$, namely $\frac{1}{N}\sum_{x\in\Omega^h} |u_c(x)-u^*(x)|$, for different values $c$ and $N$, where $N$ depends on $c$ such that $N(c) = 33 \times 33$ if $c=0$ and $N(c) = (2^{6+\log_{10}c}+1)\times (2^{6+\log_{10}c}+1)$ otherwise.}\label{fig:2d:Gammaconvergence:Min}
\end{figure}

\graphicspath{{./Scripts/2024_08_08_gamma_convergance2d/Scripts/GammaConvergence/layers_128-128-128_acts_relu-relu-relu}}
\newcommand{\imgwidth}{0.16\linewidth}
\begin{figure}
\centering
\includegraphics[width=\imgwidth]{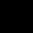}
\includegraphics[width=\imgwidth]{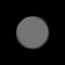}
\includegraphics[width=\imgwidth]{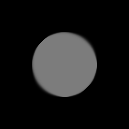}
\includegraphics[width=\imgwidth]{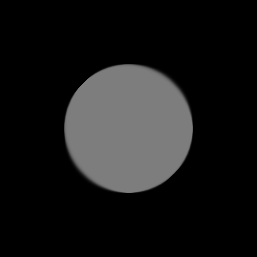}
\includegraphics[width=\imgwidth]{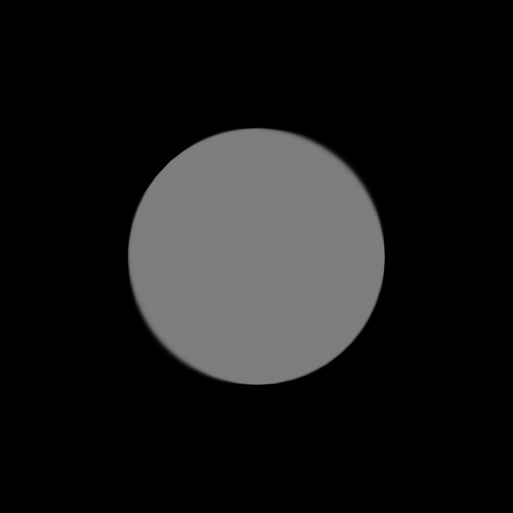}
\includegraphics[width=\imgwidth]{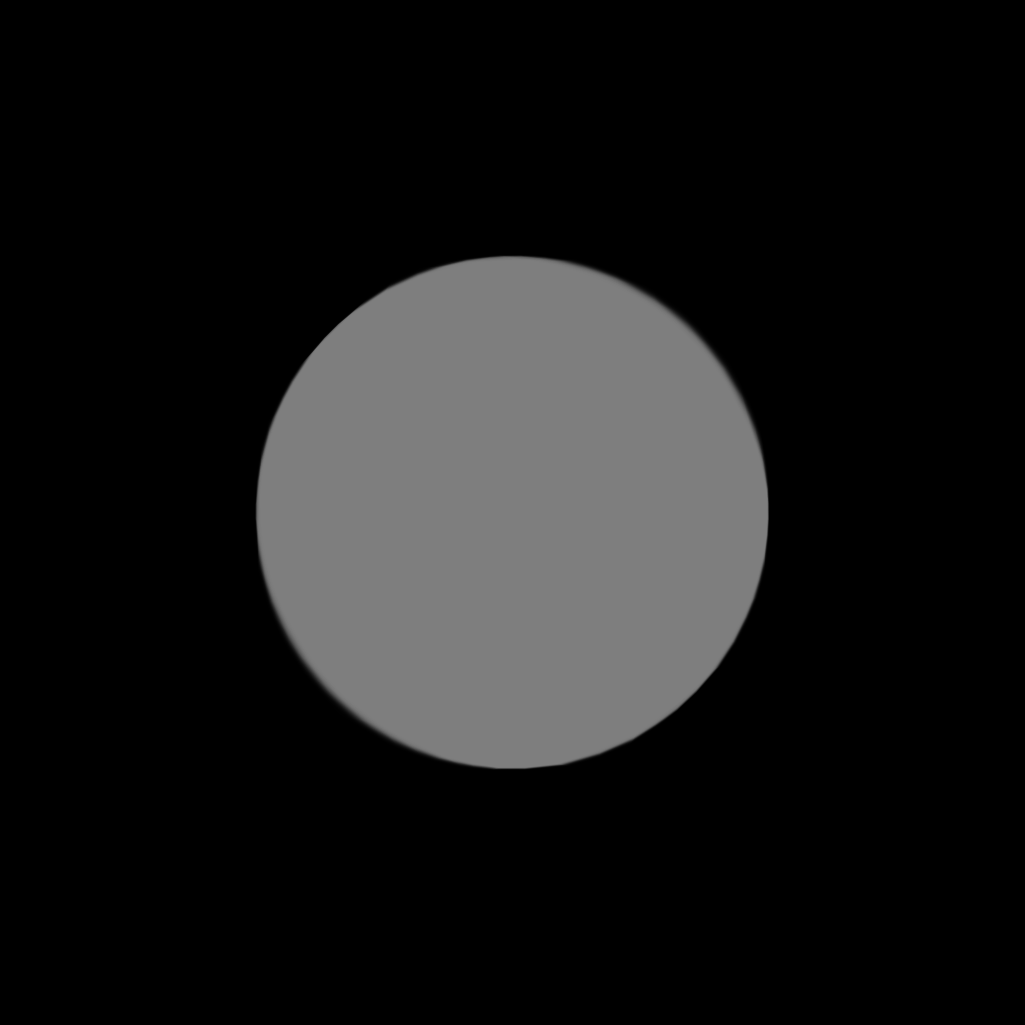}

\includegraphics[width=\imgwidth]{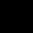}
\includegraphics[width=\imgwidth]{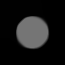}
\includegraphics[width=\imgwidth]{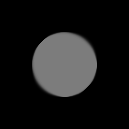}
\includegraphics[width=\imgwidth]{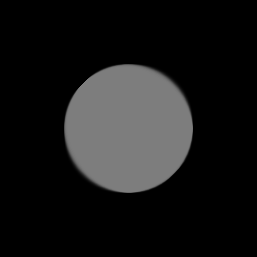}
\includegraphics[width=\imgwidth]{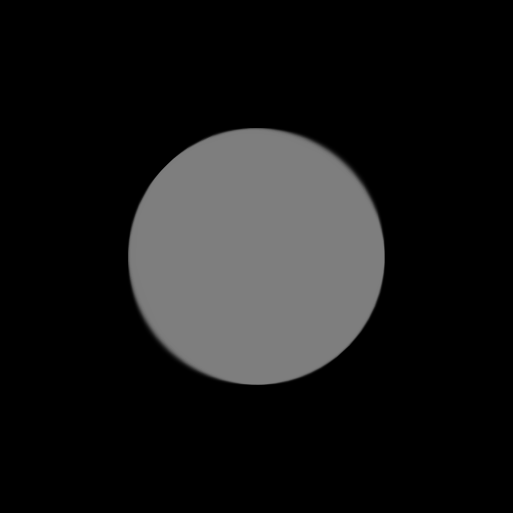}
\includegraphics[width=\imgwidth]{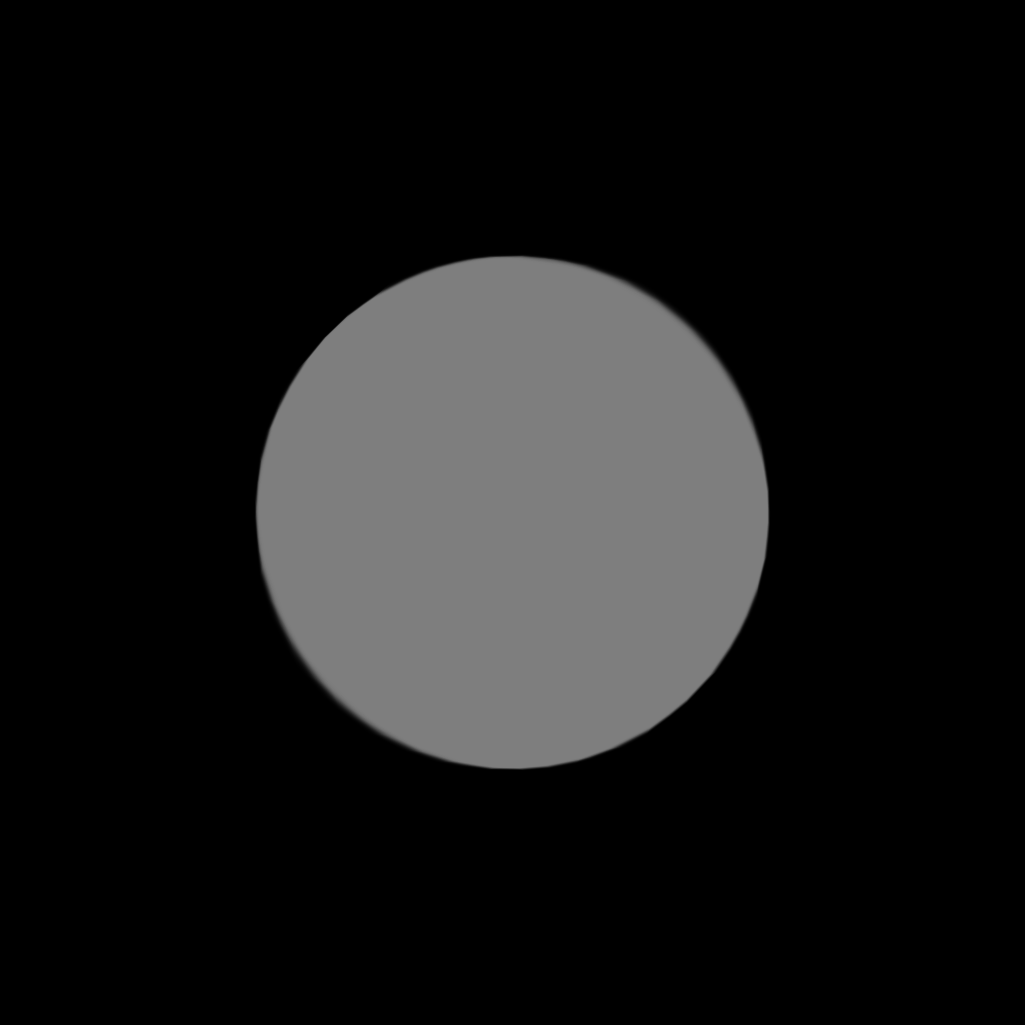}

\includegraphics[width=\imgwidth]{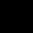}
\includegraphics[width=\imgwidth]{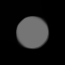}
\includegraphics[width=\imgwidth]{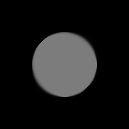}
\includegraphics[width=\imgwidth]{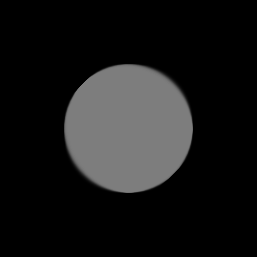}
\includegraphics[width=\imgwidth]{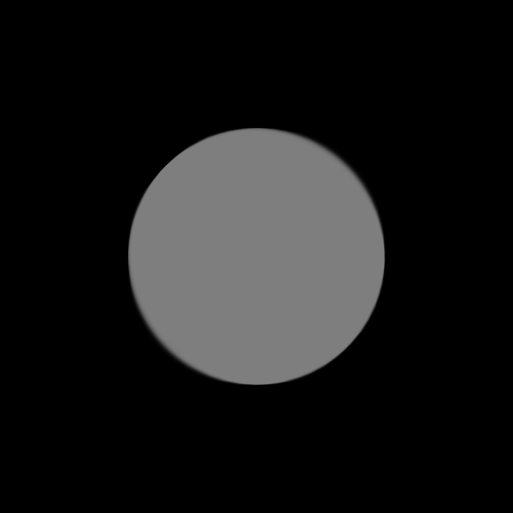}
\includegraphics[width=\imgwidth]{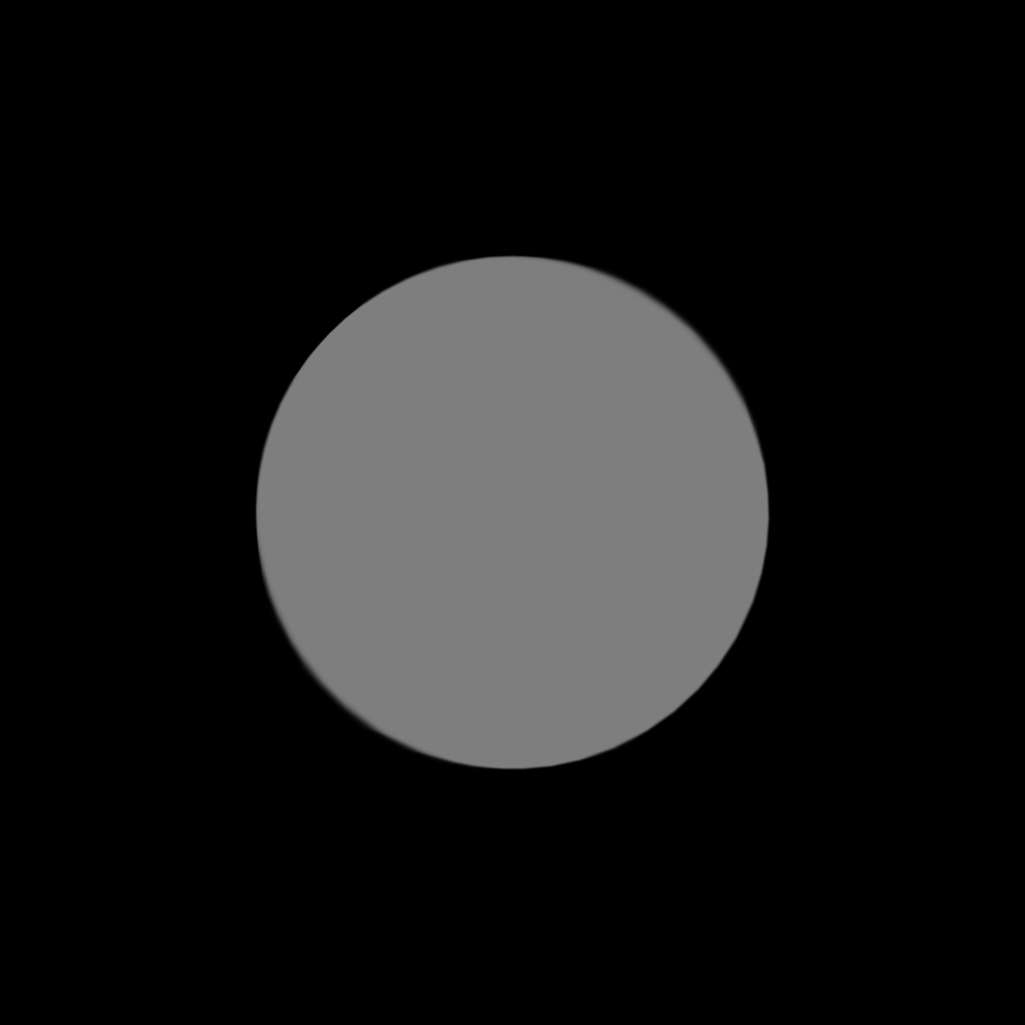}

\includegraphics[width=\imgwidth]{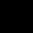}
\includegraphics[width=\imgwidth]{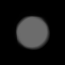}
\includegraphics[width=\imgwidth]{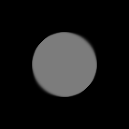}
\includegraphics[width=\imgwidth]{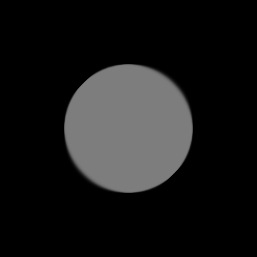}
\includegraphics[width=\imgwidth]{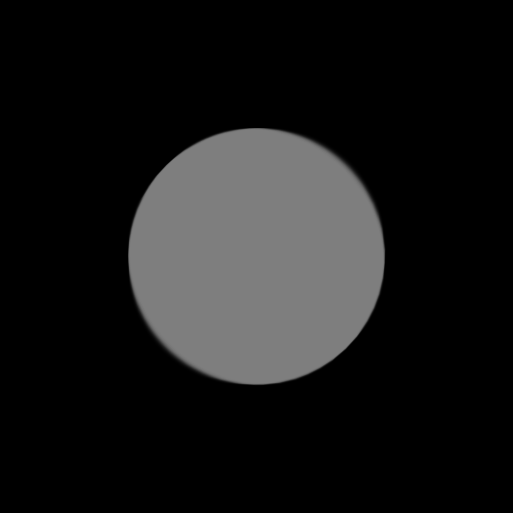}
\includegraphics[width=\imgwidth]{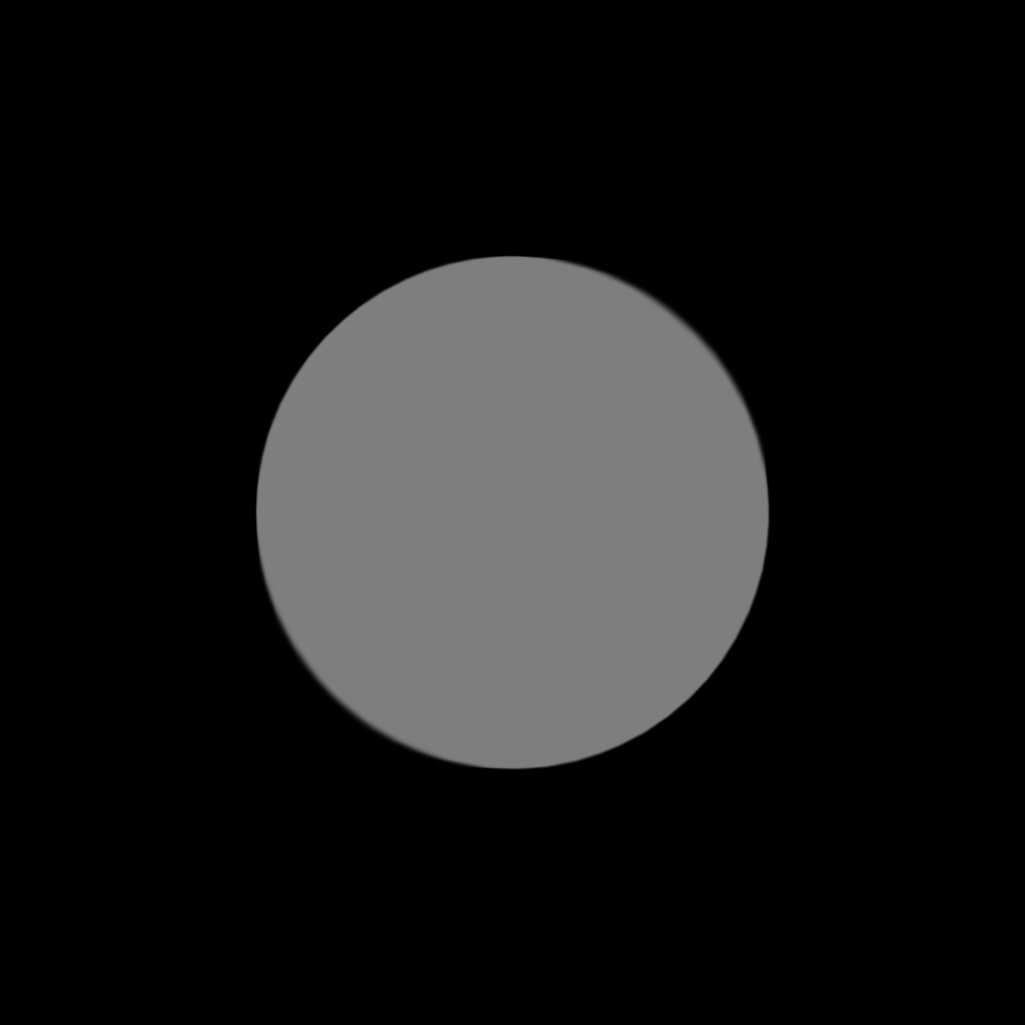}

\includegraphics[width=\imgwidth]{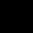}
\includegraphics[width=\imgwidth]{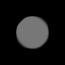}
\includegraphics[width=\imgwidth]{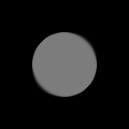}
\includegraphics[width=\imgwidth]{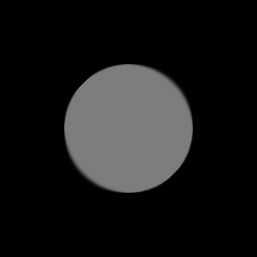}
\includegraphics[width=\imgwidth]{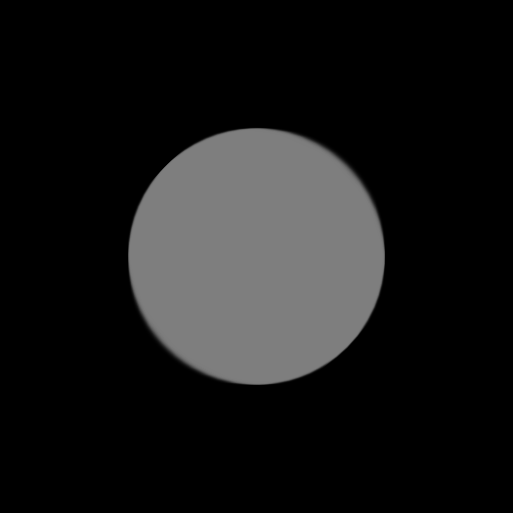}
\includegraphics[width=\imgwidth]{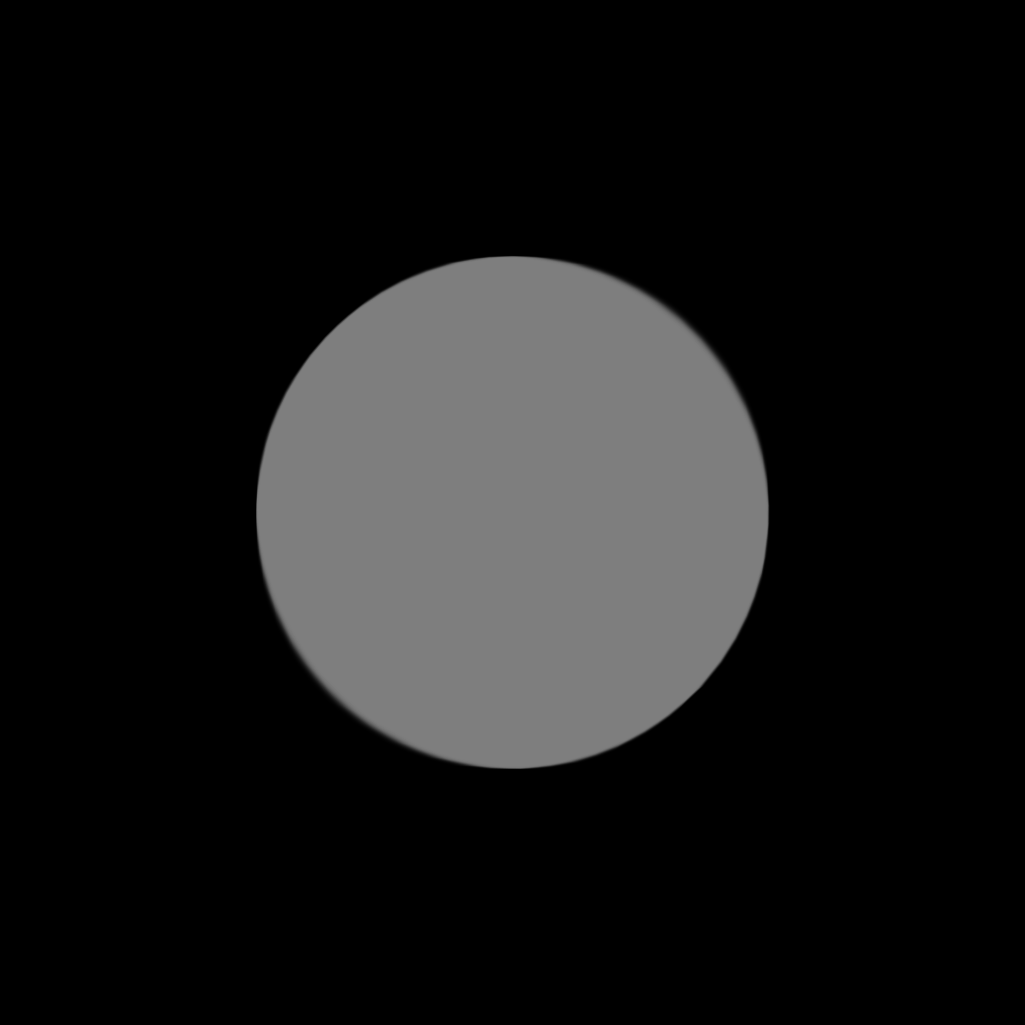}

\includegraphics[width=\imgwidth]{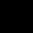}
\includegraphics[width=\imgwidth]{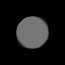}
\includegraphics[width=\imgwidth]{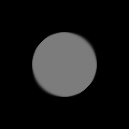}
\includegraphics[width=\imgwidth]{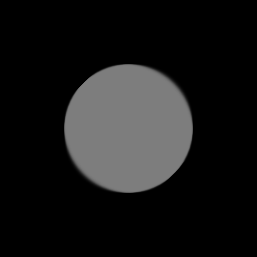}
\includegraphics[width=\imgwidth]{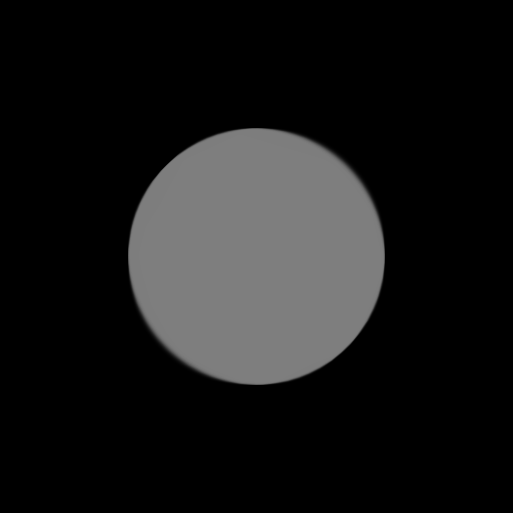}
\includegraphics[width=\imgwidth]{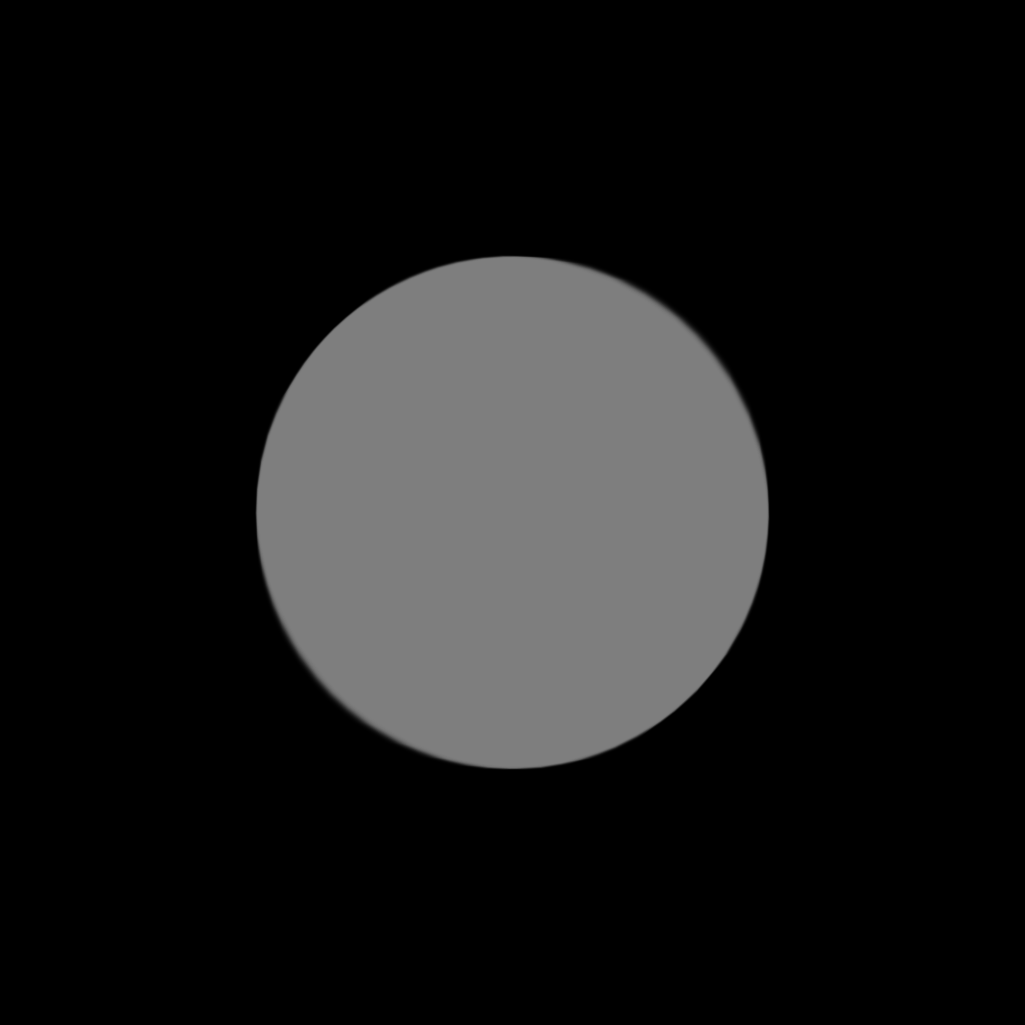}

\caption{TV-solution of \eqref{Eq:ProblemHw:constrained:discrete2}  with  $g =\mathbf{1}_{B_r(0)}$ and different discretisations of the total variation, sizes $N$ and bounds $c$: From top to bottom: $\TV_2^h$ with Huber-regularization, $\TV_2^h$ with $\gamma$-lifting, $\TV_2^h$ with $\max$-$\gamma$-lifting, $\TV_{2,1}^h$ with Huber-regularization, $\TV_{2,1}^h$ with $\gamma$-lifting, $\TV_{2,1}^h$ with $\max$-$\gamma$-lifting. From left to right $(N,c) = (33\times 33,0)$, $(N,c) = (65\times 65,1)$, $(N,c) =(129\times 129, 10)$, $(N,c) =(257\times 257, 100)$, $(N,c) =(513\times 513, 1000)$, $(N,c) =(1025\times 1025, 10000)$.}\label{fig:GammaConvergence:2d:Reconstructions}
\end{figure}

\subsubsection{Evaluation of the error estimate}

Here we evaluate our derived error estimate from \cref{Sec:ErrorEstimate}. We use the same two examples as above, see \cref{Sec:1DExample,Sec:2DExample}, as for these we have the exact solution, which allows us to compute the real error $\|u^* - u\|_{L^2(\Omega)}$, where $u^*$ is the exact solution and $u$ a computed approximation. Moreover, in these examples $T=I$ is the identity operator and hence the error estimate 
\eqref{Eq:ErrorEstimate} writes as
\begin{align*}
 \|u^* - v\|_{L^2(\Omega)} \leq \rho(v):=\rho_1(v) + \rho_2(v)
\end{align*}
for all $v\in L^2(\Omega)$, where $\rho_1(v):= 2 \|v - g\|_{L^2(\Omega)}$ and $\rho_2(v):=\frac{1}{\alpha_2}\|\xi_G\|_{L^2(\Omega)}$.
In order to investigate the nature of the error estimate and how the two terms $\rho_1$ and $\rho_2$ behave, we depict $\rho$, $\rho_1$ and $\rho_2$ together with the real error in \cref{Fig:ErrorEstimate}. Note that in \cref{Fig:ErrorEstimate} on the $x$-axes we have the number of updates of the approximation, i.e.\ the iterations at which the loss was reduced, and not the number of iterations. Exemplary we depict the error estimate for the example in \cref{Sec:1DExample} with $c=100$ and for the example in \cref{Sec:2DExample} with $c=10000$, using $\TV_2^h$ with Huber-regularization and $\TV_{2,1}^h$ with $\gamma$-lifting. A similar behavior can be observed for the other settings of \cref{Sec:1DExample,Sec:2DExample}. 
In \cref{Fig:ErrorEstimate} we see that while the error estimate indeed always holds, as expected, it is governed by $\rho_2$. In fact $\rho_2$ leads to an oscillatory behavior of the error estimate $\rho$, which does not seem to decrease although the real error is. This behavior does not really seem to make the error estimate practical, for example, to be used in a stopping criterion. 

\graphicspath{{./Scripts/1d/test_boundedweights/denoising/layers_64-128_acts_relu-relu_l2_0/ruleAR/learning_rate0.01/}}
\begin{figure}
\centering
\begin{subfigure}[t]{0.31\textwidth}
        \centering
  \includegraphics[width=\linewidth]{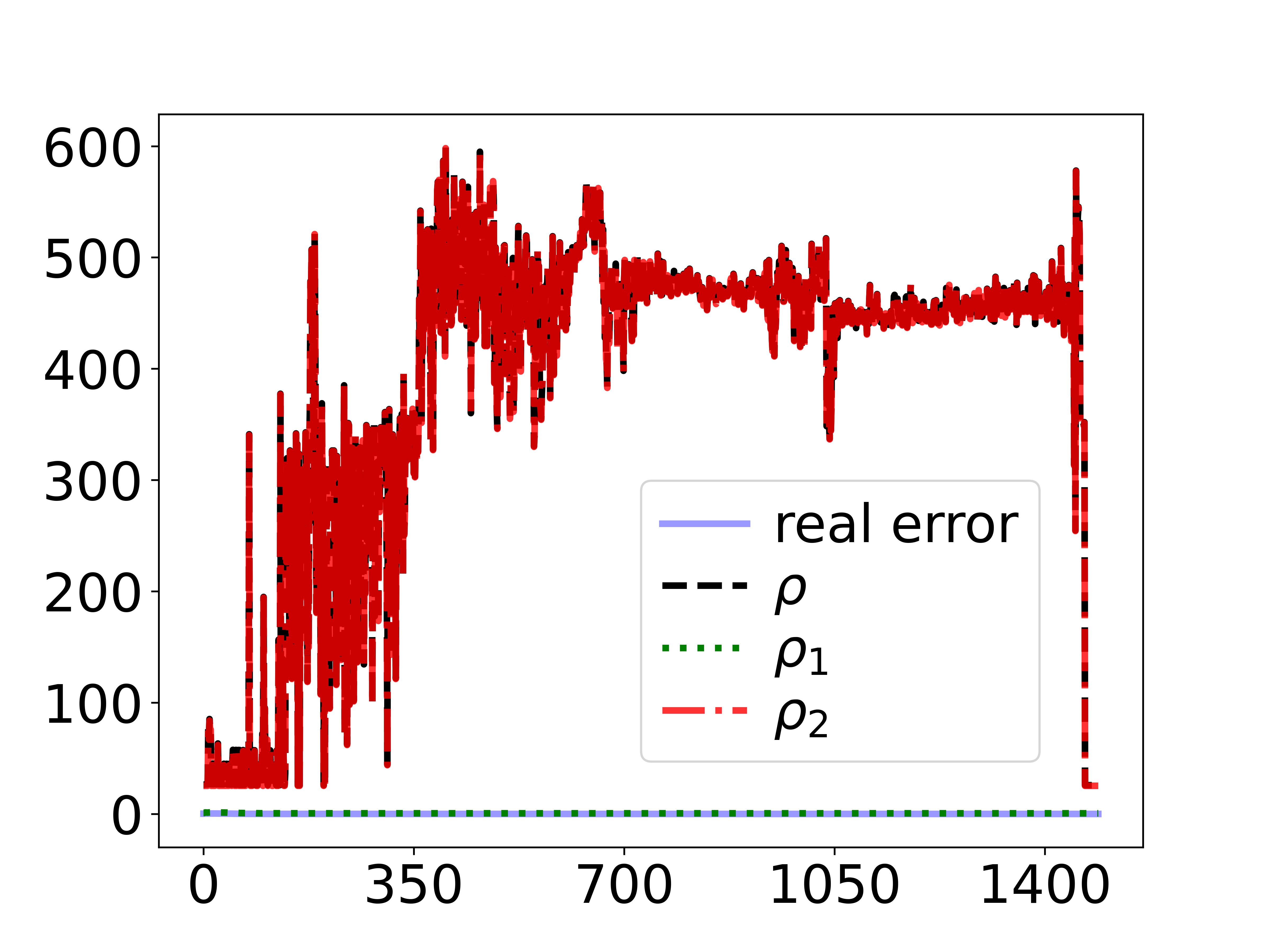}
  \caption{Example of \cref{Sec:1DExample} with $c=100$.}
    \end{subfigure}\hfill
    \graphicspath{{./Scripts/2024_08_08_gamma_convergance2d/Scripts/GammaConvergence/layers_128-128-128_acts_relu-relu-relu}}
    \begin{subfigure}[t]{0.31\textwidth}
        \centering
  \includegraphics[width=\linewidth]{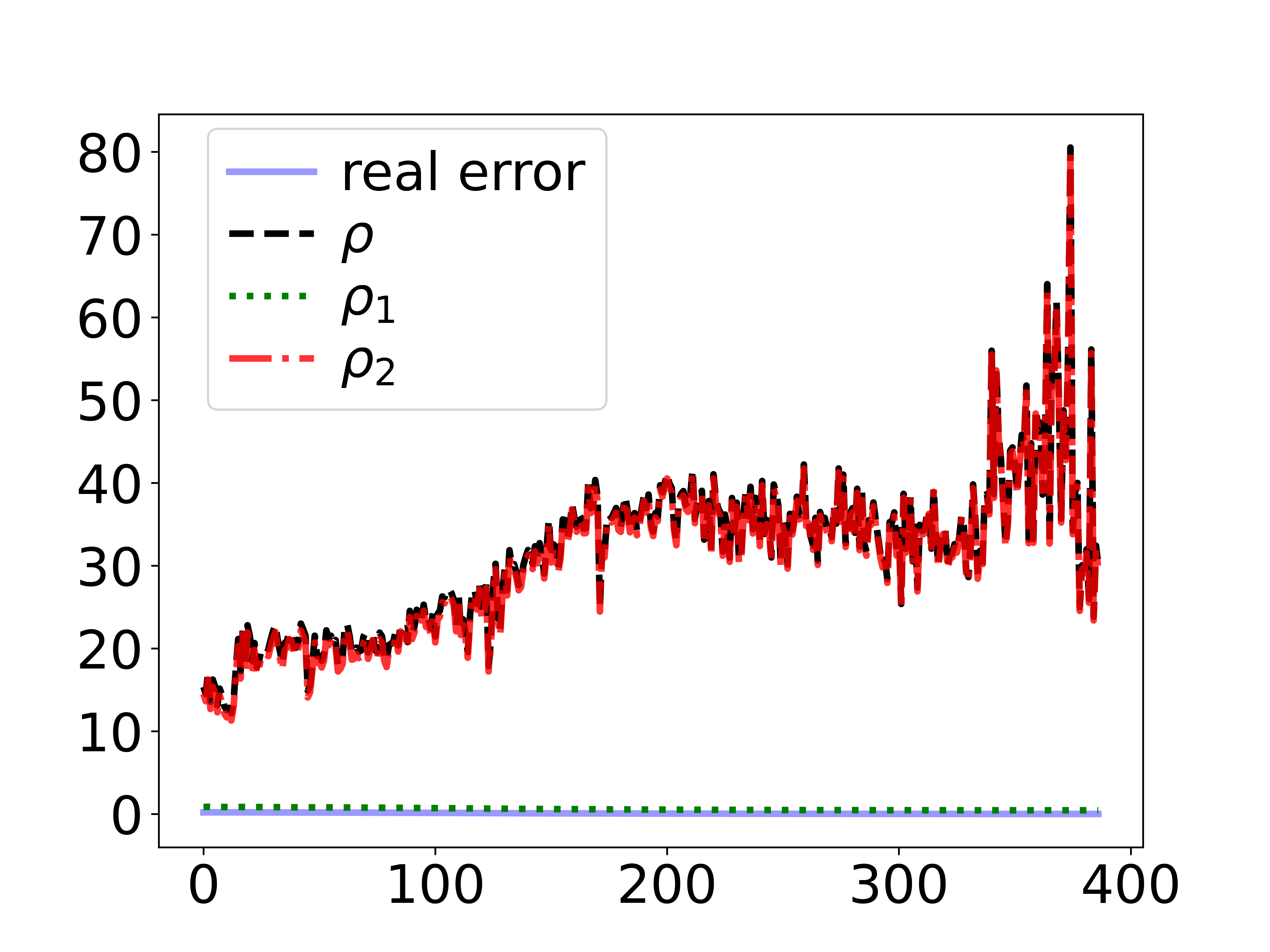}
  \caption{Example of \cref{Sec:2DExample} with $\TV_2^h$ using Huber-regularization and $c=10000$.}
    \end{subfigure}\hfill
    \begin{subfigure}[t]{0.31\textwidth}
        \centering
  \includegraphics[width=\linewidth]{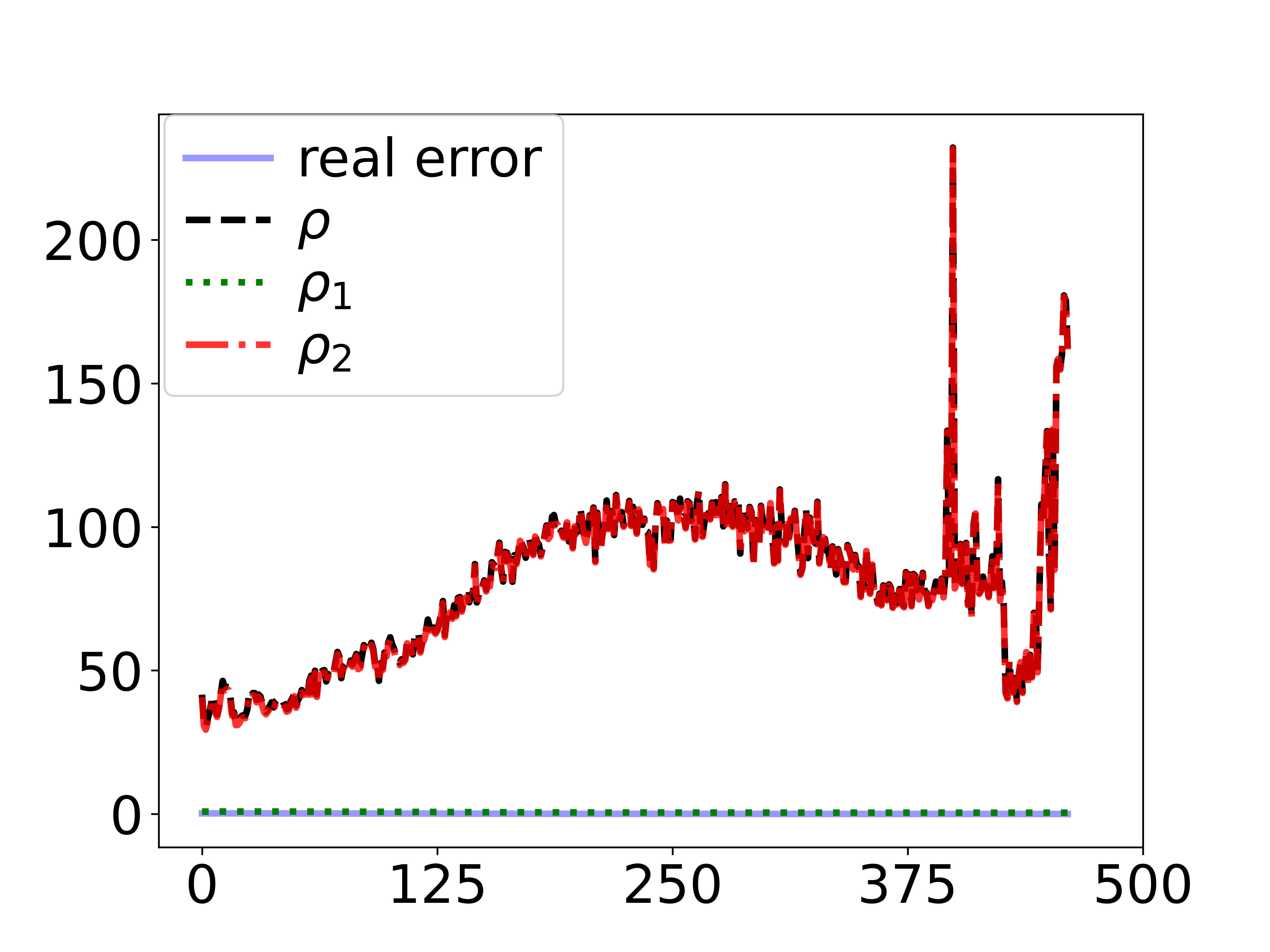}
  \caption{Example of \cref{Sec:2DExample} with $\TV_{2,1}^h$ using $\gamma$-lifting and $c=10000$.}
    \end{subfigure}
    \caption{Error estimate $\rho = \rho_1 + \rho_2$ versus real error. Here the $x$-axes represents the number of updates made to the approximation.} \label{Fig:ErrorEstimate}
\end{figure}

\subsection{Applications}
We consider three applications, namely denoising, inpainting and deblurring. In all the applications the image is scaled to the unit square $[0,1]^2$. In order to recover the original image from the corrupted observation we minimize \eqref{Eq:ProblemHw:constrained:discrete2}. Thereby in all the experiments we choose a forward difference discretisation of the gradient operator with incorporated homogeneous Neumann boundary conditions. Further we choose $\lambda=1$ and $\TV_2^h$ with $\gamma$-lifting. That is the last term in \eqref{Eq:ProblemHw:constrained:discrete2}, i.e.\ $\frac{\lambda}{N} \sum_{x\in\Omega^h} |\nabla^h u_\theta(x)|_\blacktriangleright$, is set to be $\frac{1}{N} \sum_{x\in\Omega^h} |\nabla_{\text{F}}^h u_\theta(x)|_{2,\gamma}$, where $\gamma=10^{-10}$. 
The solution space $\mcHc{c}{M}$ consists of ReLU-NNs with 3 hidden layers, where each layer has 128 neurons, and we set $c=10000$. Adam with learning rate 0.001 is used to find an approximation and the training is always stopped after 300001 iterations. The parameters $\alpha_1$ and $\alpha_2$ are selected differently depending on the application and specified at the respective places. However, they are not optimized in any way and only chosen such that the effect of the total variation can be observed. Next, we describe the image reconstruction tasks before we present and discuss the respective results.

\graphicspath{{./Scripts/application_circle/}}

\begin{figure}
\centering
\includegraphics[width=0.19\linewidth]{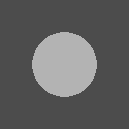}
\includegraphics[width=0.19\linewidth]{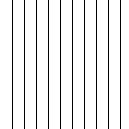}
\caption{Original disk-image of size $129\times 129$ pixels (left) and inpainting mask (right).}\label{Fig:OriginalImages}
\end{figure}

\paragraph{Denoising}
Image denoising is the task of removing noise from an image. We generate from the original image $u_{\text{orig}}$ a noisy input $g = \mcN(u_{\text{orig}} + \eta)$, where $\eta$ represents additive Gaussian noise with zero mean and standard deviation $\sigma$ and $\mcN$ represents salt-and-pepper noise with probability $s\in[0,1)$, i.e.\
$$
g(x) = \begin{cases}
0 &\text{ with probability } \frac{s}{2}, \\
1 &\text{ with probability } \frac{s}{2}, \\
u_{\text{orig}} + \eta &\text{ with probability } 1- s. 
\end{cases}
$$
 Hence for this application we set $T^h$ being the identity operator in \eqref{Eq:ProblemHw:constrained:discrete2}.

For $u_{\text{orig}}$ being the disk-image of size $129\times 129$ pixels, see \cref{Fig:OriginalImages}, we add noise with $\sigma=0.1$ and $s=0.1$. The obtained image $g$ can be seen in the first row and first column of \cref{Fig:Applications:Circle}. Additionally to the above settings we choose $\alpha_1=10$, $\alpha_2=30$, and $N_1 = N_2 = 129$. 

For $u_{\text{orig}}$ being the building-image of size $256\times 256$ pixels, see \cref{Fig:OriginalImages_building}, we add noise with $\sigma=0.1$ and $s=0.05$. The obtained image $g$ can be seen in the first row and first column of \cref{Fig:Applications:Building}. Additionally to the above settings we choose $\alpha_1=100$, $\alpha_2=100$, and $N_1 = N_2 = 256$. 

\graphicspath{{./Scripts/Lund/}}
\begin{figure}[h]
\centering
\includegraphics[width=0.19\linewidth]{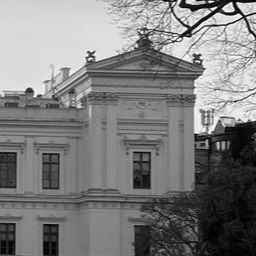}
\includegraphics[width=0.19\linewidth]{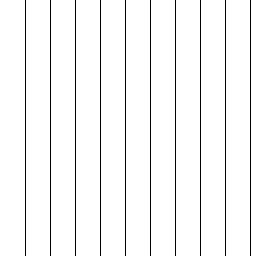}
\caption{Original building-image of size $256\times 256$ pixels (left) and inpainting mask (right).}\label{Fig:OriginalImages_building}
\end{figure}

\paragraph{Inpainting}

Image inpainting is the task of restoring missing parts of an image. That is the observed image $g$ is only given in $\Omega\setminus D$, where $D\subset \Omega$ is the so-called inpainting domain. The task is then to recover $g$ in $D$. Hence for this application we set $T= \mathbf{1}_{\Omega\setminus D}$ (inpainting operator) defined by
$$
\mathbf{1}_{\Omega\setminus D}u(x) = \begin{cases}
u(x) & x\in \Omega\setminus D,\\
0& x\in D.
\end{cases}
$$
For $u_{\text{orig}}$ being the disk-image of size $129\times 129$ pixels, see \cref{Fig:OriginalImages}, i.e.\ $N_1 = N_2= 129$, we apply the inpainting mask in \cref{Fig:OriginalImages}. The obtained image $g$ can be seen in the second row and first column of \cref{Fig:Applications:Circle}. Further we set $\alpha_1=300$, $\alpha_2=300$. 

We also consider the situation where this observed image is distorted by Gaussian and salt-and-pepper noise in addition to the missing parts, i.e.\ we perform simultaneously image inpainting and denoising. In particular we first apply the mask and then add additive Gaussian noise with 0 mean and standard deviation $\sigma=0.1$ and salt-and-pepper noise with $s=0.1$. The respective observation can be seen in the third row and first column of \cref{Fig:Applications:Circle}. Here we set $\alpha_1=10$, $\alpha_2=30$. 

For $u_{\text{orig}}$ being the building-image of size $256\times 256$ pixels, see \cref{Fig:OriginalImages_building}, i.e.\ $N_1 = N_2= 256$, we apply the inpainting mask in \cref{Fig:OriginalImages_building}. The obtained image $g$ can be seen in the second row and first column of \cref{Fig:Applications:Building}. We set $\alpha_1=300$ and $\alpha_2=300$ for our calculations. When this image is additionally distorted by additive Gaussian white noise with $\sigma=0.1$ and salt-and-pepper noise with $s=0.05$, see third row and first column of \cref{Fig:Applications:Building}, we set $\alpha_1=100$ and $\alpha_2=100$.

\paragraph{Deblurring}
Image deblurring is the task of reconstructing a sharp image from its blurry observation. Here we assume that the blurring kernel is known. In particular in our experiments we take a Gaussian blurring kernel of size $11\times 11$ pixels and standard deviation 20. The operator $T^h$ in \eqref{Eq:ProblemHw:constrained:discrete2} is then chosen accordingly. Hence the observed image depicted in the fourth row and first column of \cref{Fig:Applications:Circle} is obtained by applying $T^h$ on the disk-image of size $129\times 129$ pixels, see \cref{Fig:OriginalImages}.
 We set $\alpha_1=300$, $\alpha_2=300$ and $N_1= 129 = N_2$ in \eqref{Eq:ProblemHw:constrained:discrete2}.

We also consider the situation where this observed image is distorted by Gaussian and salt-and-pepper noise in addition to being blurred, i.e.\ we perform simultaneously image deblurring and denoising. In particular, after blurring the image with the above mentioned Gaussian kernel we add additive Gaussian noise with 0 mean and standard deviation $\sigma=0.1$ and salt-and-pepper noise with $s=0.1$. The respective observation can be seen in the last row and first column of \cref{Fig:Applications:Circle}. Here we set $\alpha_1=10$, $\alpha_2=30$. 

For $u_{\text{orig}}$ being the building-image of size $256\times 256$ pixels, see \cref{Fig:OriginalImages_building}, i.e.\ $N_1 = N_2= 256$, we apply the same blurring kernel to obtain $g$, which is depicted in the fourth row and first column of \cref{Fig:Applications:Building}. We set $\alpha_1=300$ and $\alpha_2=300$ for our calculations. When this image is additionally distorted by additive Gaussian white noise with $\sigma=0.1$ and salt-and-pepper noise with $s=0.05$, see last row and first column of \cref{Fig:Applications:Building}, we set $\alpha_1=100$ and $\alpha_2=100$.

\paragraph{Results}

The obtained restoration for all applications on the grid of the observed image can be seen in the second column of \cref{Fig:Applications:Circle,Fig:Applications:Building}. Since a neural network is a continuous function not restricted to a particular grid, the obtained solutions can be plotted on any discretisation of $\Omega$. We plot the obtained solutions on the grid of the observed image, i.e.\ on a grid with $129 \times 129$ pixels for the disk-image and $256 \times 256$ pixels for the building-image, see second column of \cref{Fig:Applications:Circle,Fig:Applications:Building}, and on a 10 times finer grid, i.e.\ on a grid with $1290 \times 1290$ pixels for the disk-image and $2560 \times 2560$ pixels for the building-image, see fourth column of \cref{Fig:Applications:Circle,Fig:Applications:Building}. In order to see the difference in representation of the solutions we zoom-in in one part of the images, see third and fifth column of \cref{Fig:Applications:Circle,Fig:Applications:Building}. 
%While optimizing on a rather coarse grid, we actually obtain a solution in a continuum allowing us to depict the solution on an arbitrary fine grid.
Here we observe that the solutions on a fine grid seems to be well adjusted to the nature of the original signal to be restored, e.g.\ the curvature of the disk seems to be preserved. 
However, we emphasize that according to \cref{Prop:Discrete,Prop:Discrete2} on a fixed grid the solution of \eqref{Eq:ProblemHw:constrained:discrete2} is the same as the respective finite difference solution, if it exists, in the sense that the values of the solutions on the grid points are the same. Hence the effects of the neural network's design are only apparent or relevant at a level of detail finer than what is captured by the current data resolution.

\graphicspath{{./Scripts/application_circle/}}
\renewcommand{\imgwidth}{0.19\linewidth}
\begin{figure}
\includegraphics[width=\imgwidth]{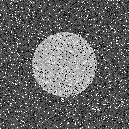}
\includegraphics[width=\imgwidth]{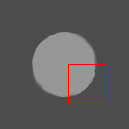}
\includegraphics[width=\imgwidth]{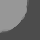}
\includegraphics[width=\imgwidth]{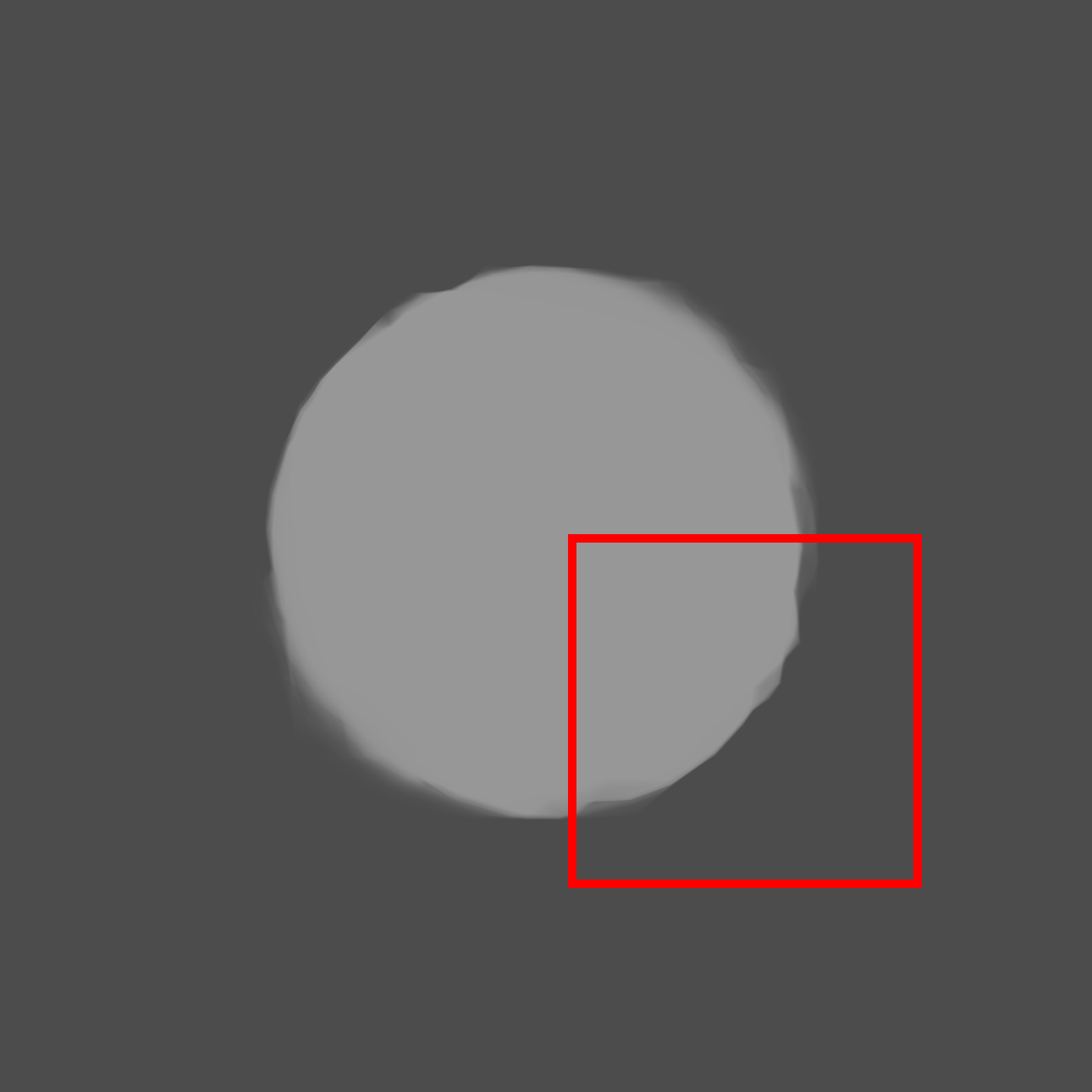}
\includegraphics[width=\imgwidth]{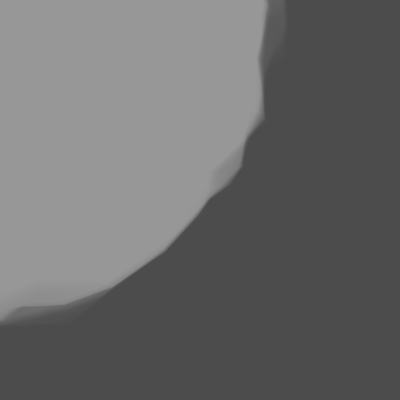}

\includegraphics[width=\imgwidth]{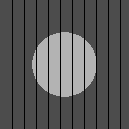}
\includegraphics[width=\imgwidth]{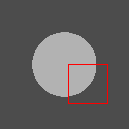}
\includegraphics[width=\imgwidth]{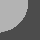}
\includegraphics[width=\imgwidth]{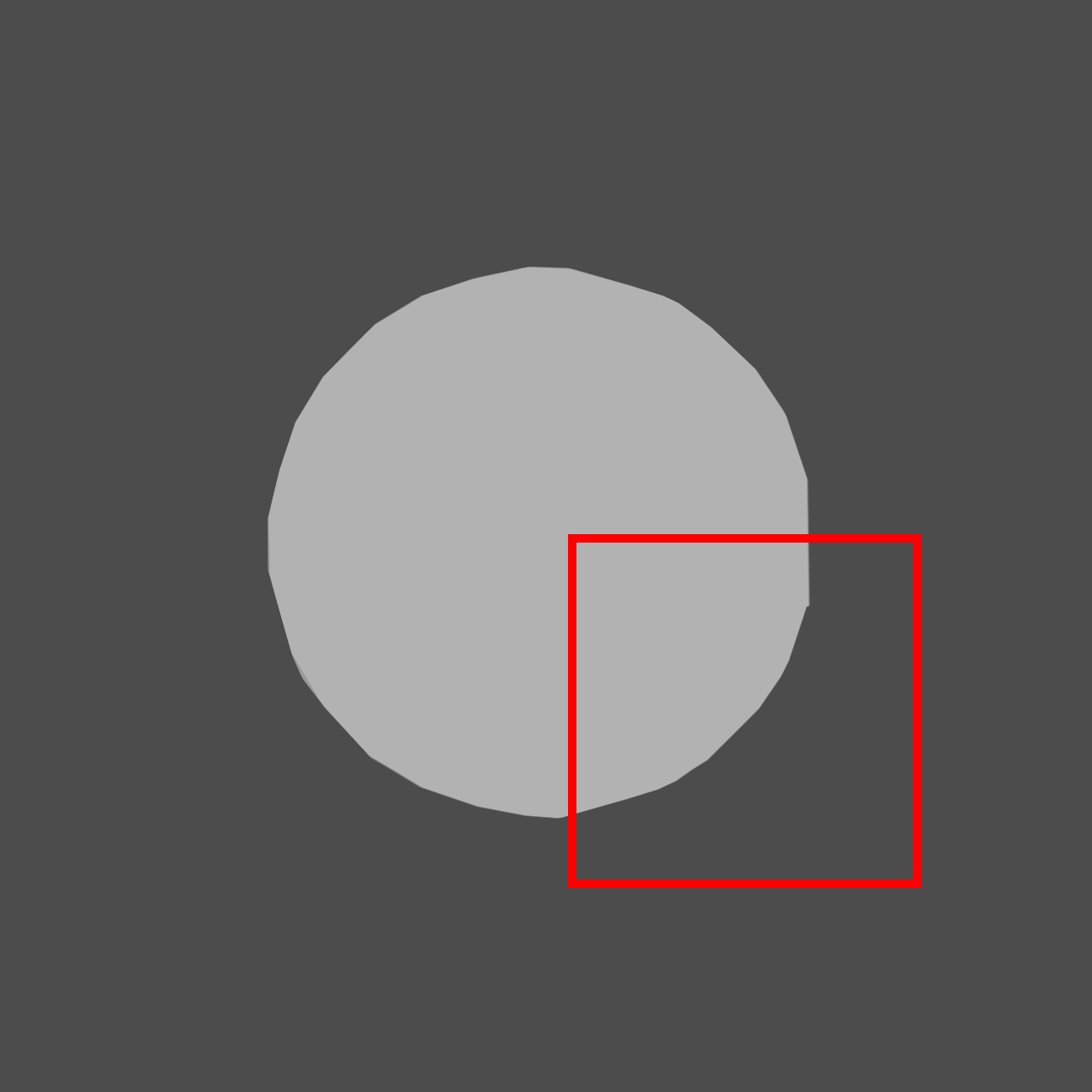}
\includegraphics[width=\imgwidth]{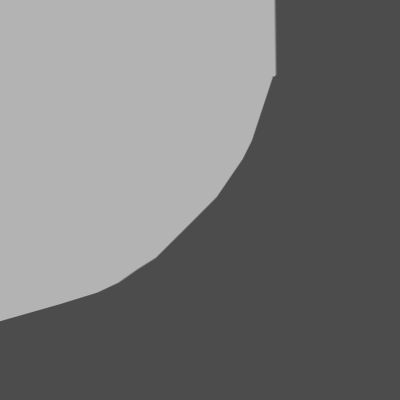}

\includegraphics[width=\imgwidth]{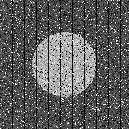}
\includegraphics[width=\imgwidth]{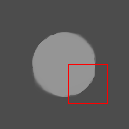}
\includegraphics[width=\imgwidth]{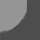}
\includegraphics[width=\imgwidth]{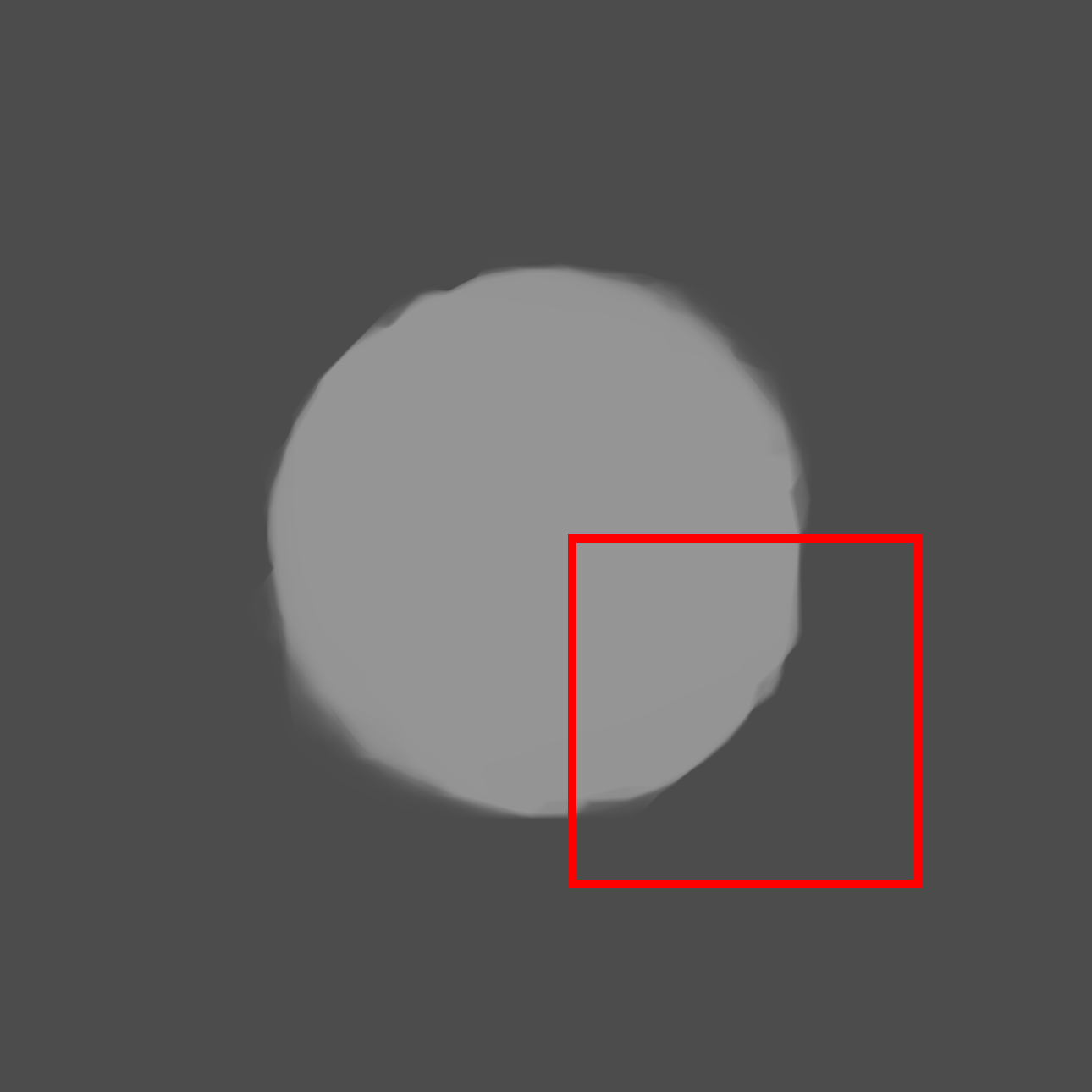}
\includegraphics[width=\imgwidth]{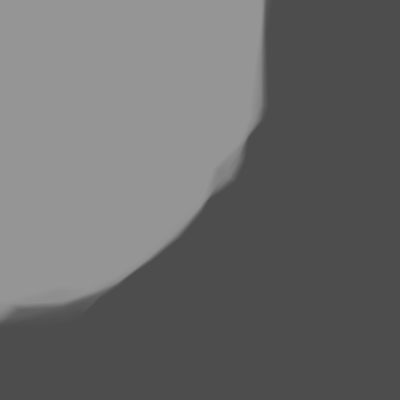}

\includegraphics[width=\imgwidth]{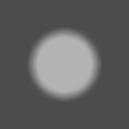}
\includegraphics[width=\imgwidth]{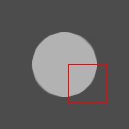}
\includegraphics[width=\imgwidth]{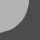}
\includegraphics[width=\imgwidth]{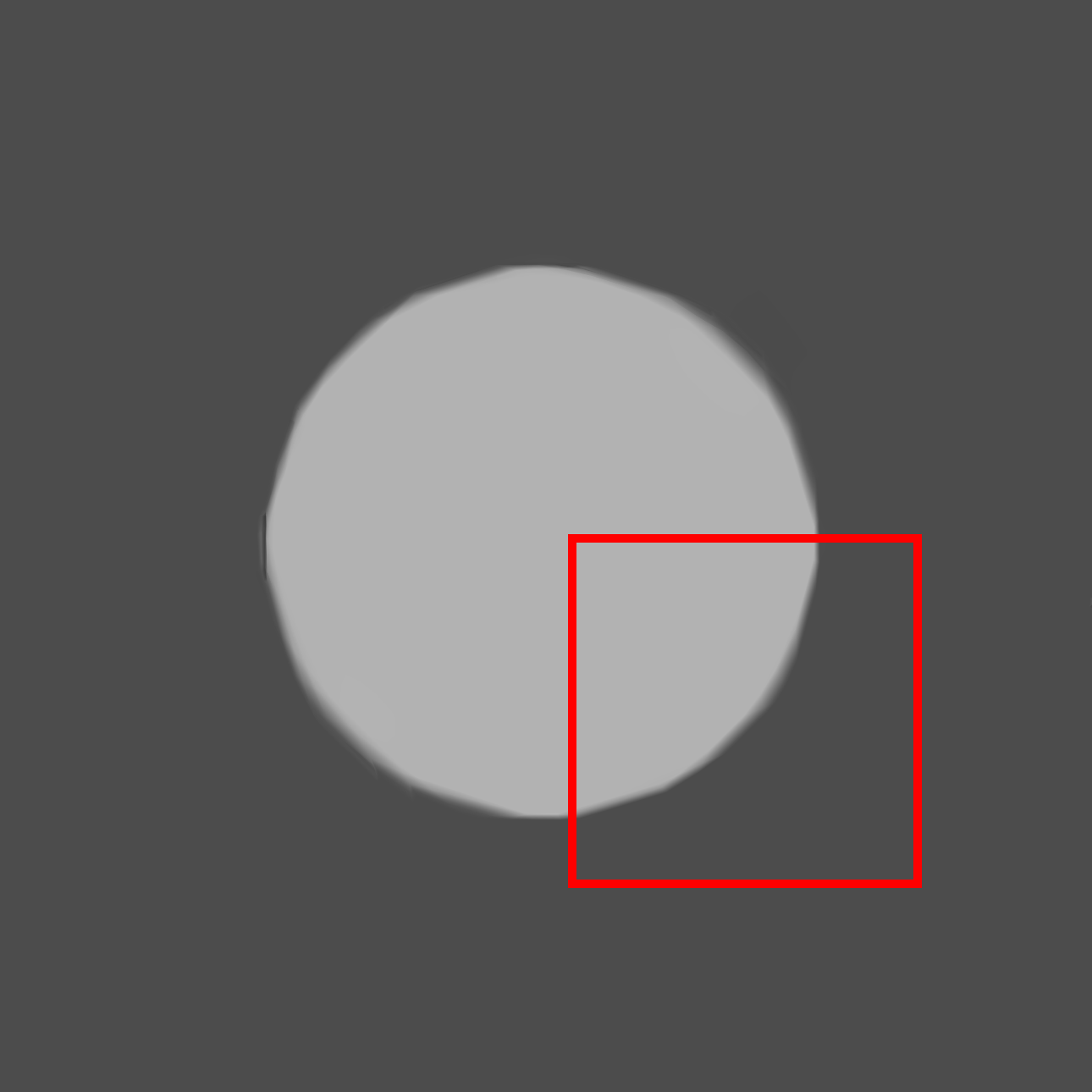}
\includegraphics[width=\imgwidth]{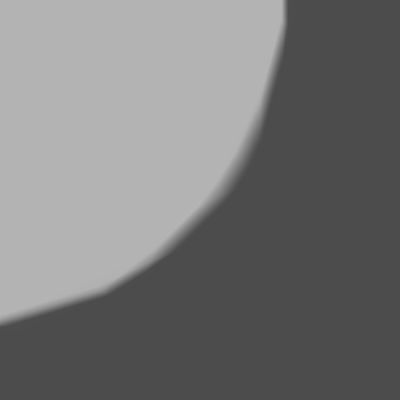}

\includegraphics[width=\imgwidth]{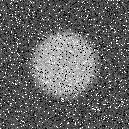}
\includegraphics[width=\imgwidth]{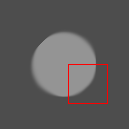}
\includegraphics[width=\imgwidth]{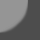}
\includegraphics[width=\imgwidth]{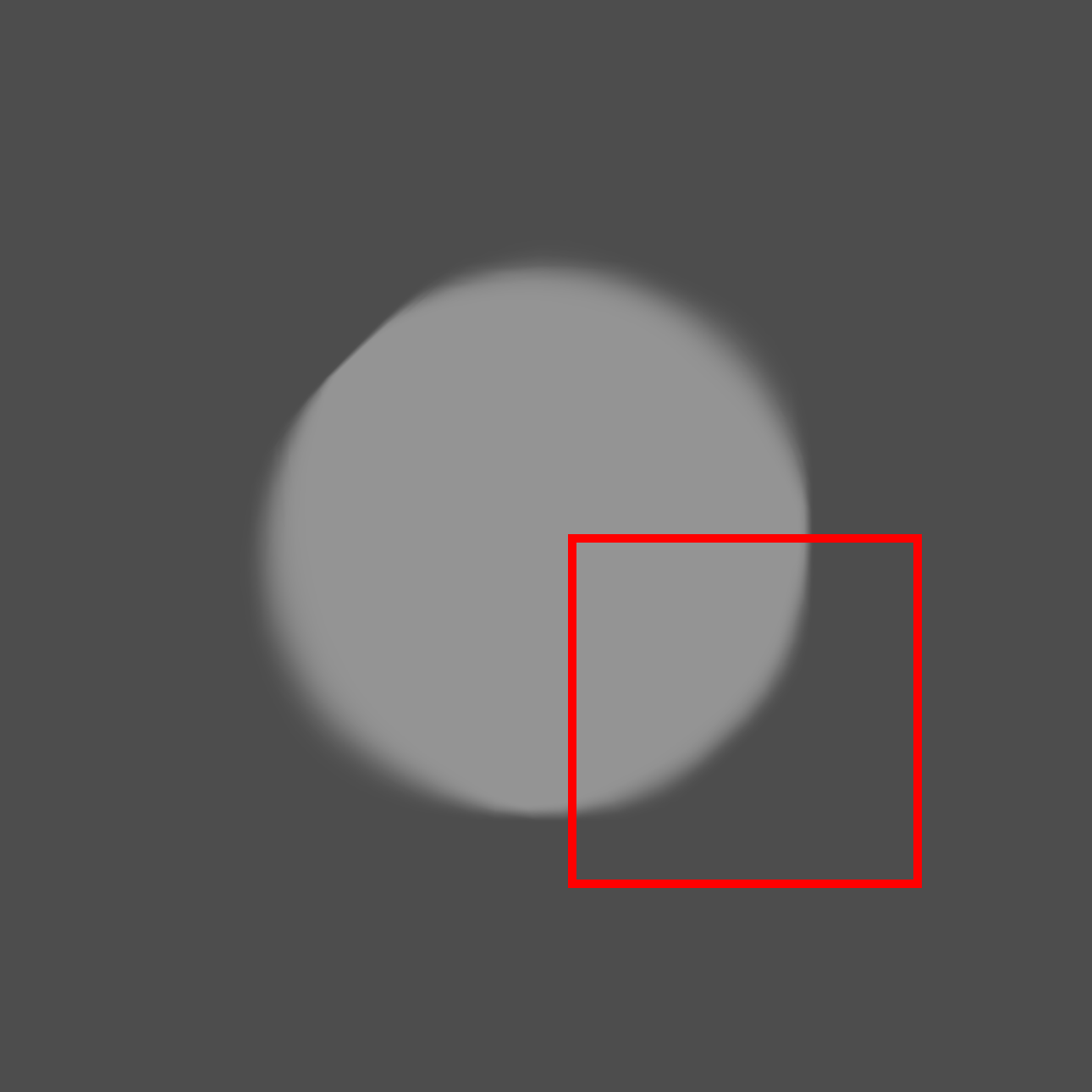}
\includegraphics[width=\imgwidth]{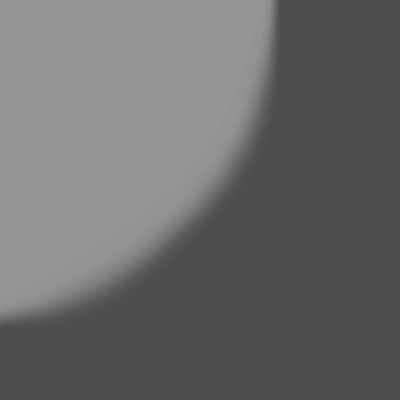}

\caption{Columns from left to right: Corrupted image ($129 \times 129$ pixels), reconstruction plotted with $129 \times 129$ pixels, zoomed reconstruction of highlighted area, reconstruction plotted on $10\times$ finer grid ($1290 \times 1290$ pixels), zoomed reconstruction of highlighted area of finer grid. Applications from top to bottom: denoising, inpainting, inpainting + denoising, deblurring, deblurring + denoising.}\label{Fig:Applications:Circle}
\end{figure}

\graphicspath{{./Scripts/Lund/}}
\renewcommand{\imgwidth}{0.19\linewidth}
\begin{figure}
\includegraphics[width=\imgwidth]{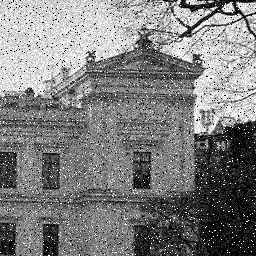}
\includegraphics[width=\imgwidth]{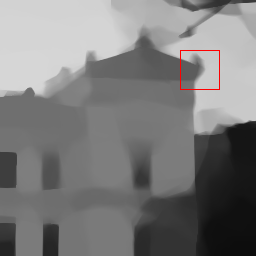}
\includegraphics[width=\imgwidth]{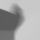}
\includegraphics[width=\imgwidth]{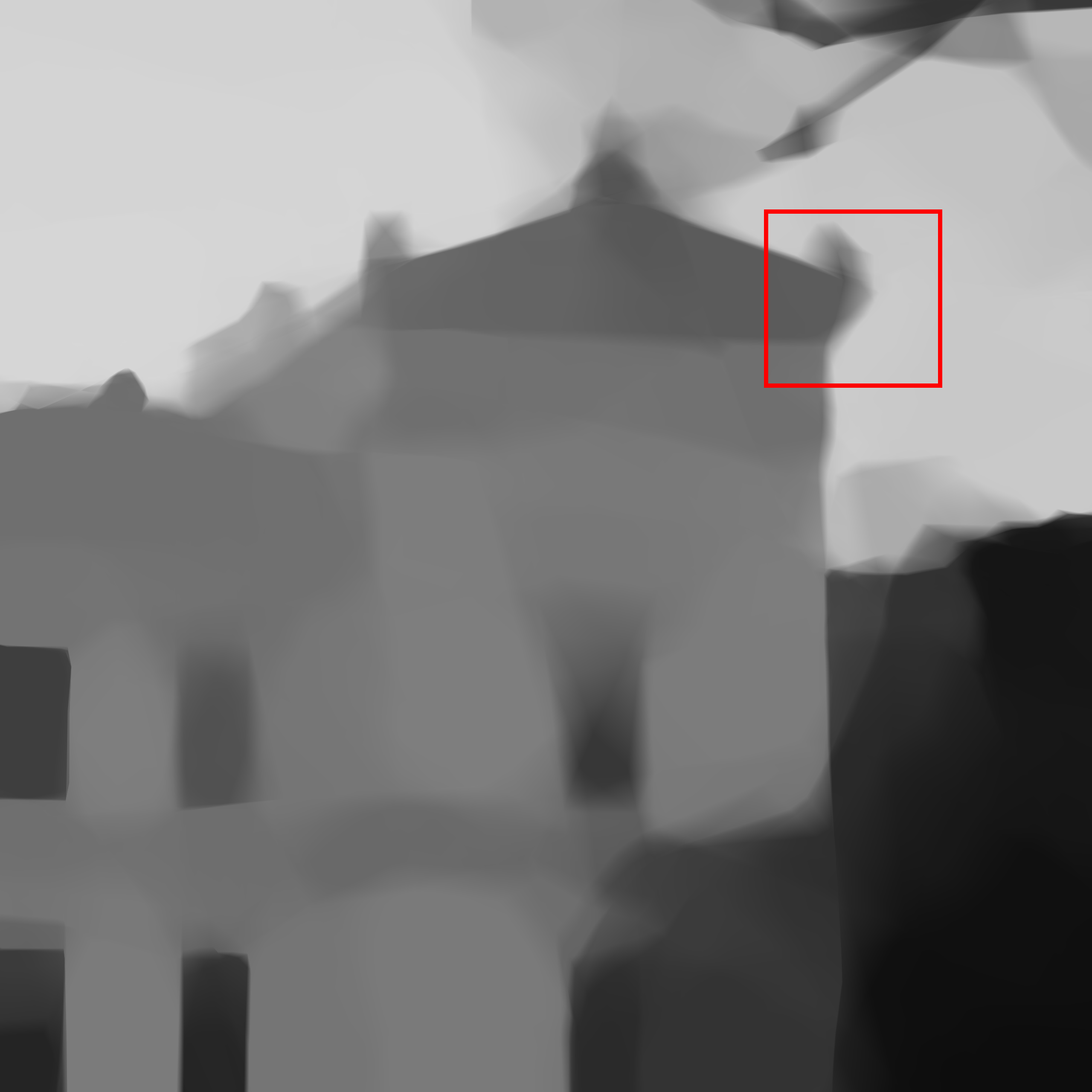}
\includegraphics[width=\imgwidth]{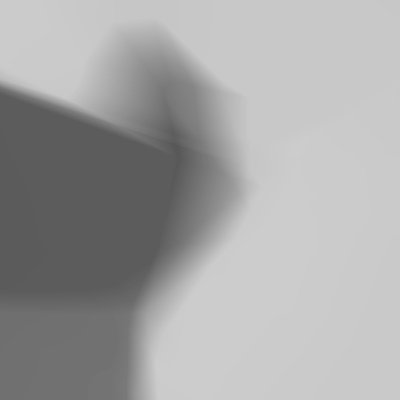}

\includegraphics[width=\imgwidth]{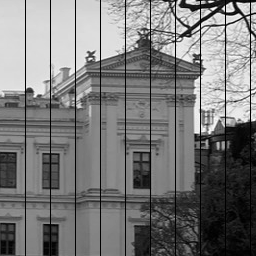}
\includegraphics[width=\imgwidth]{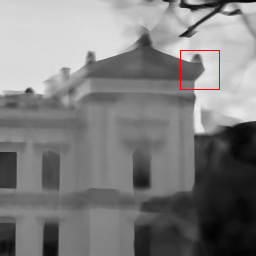}
\includegraphics[width=\imgwidth]{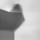}
\includegraphics[width=\imgwidth]{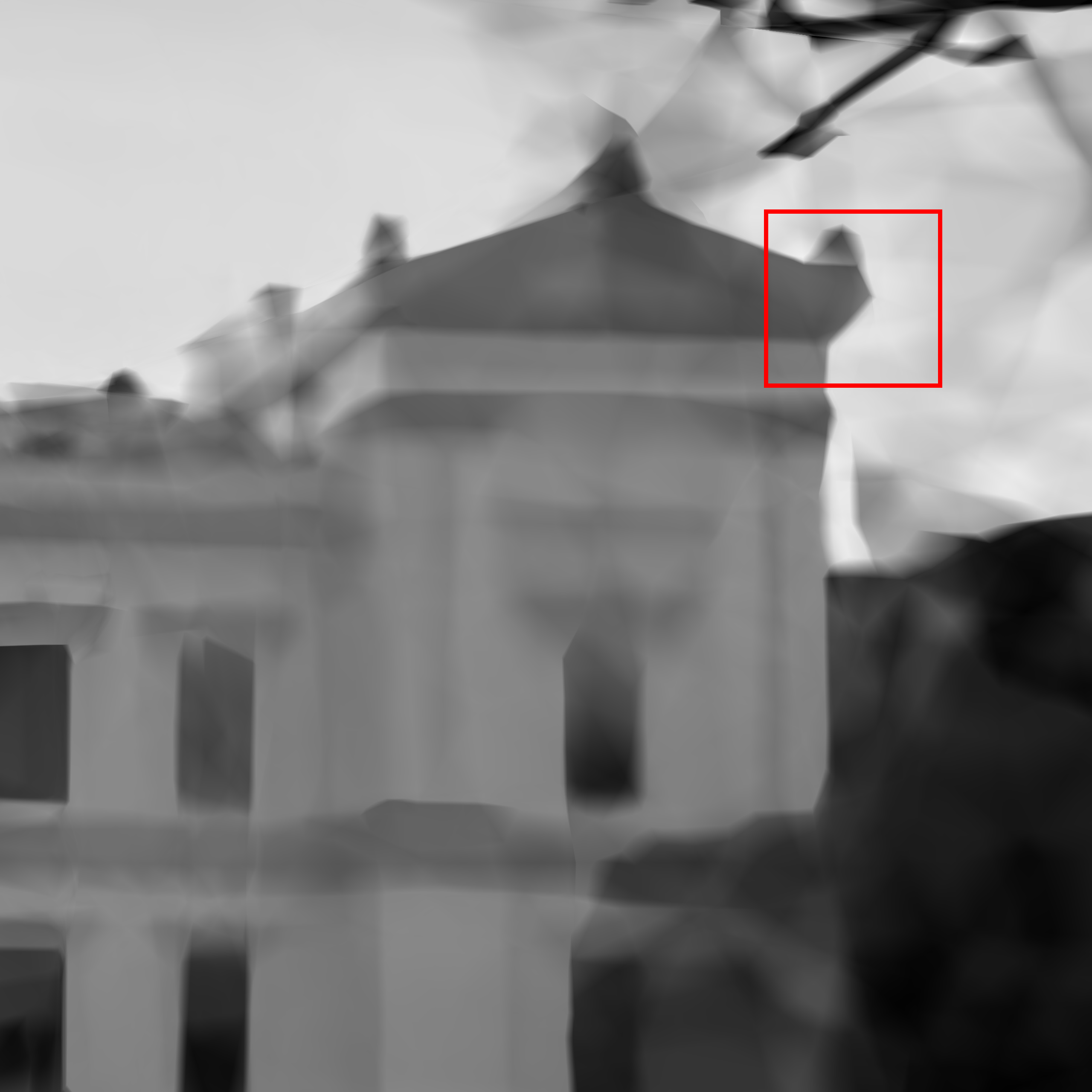}
\includegraphics[width=\imgwidth]{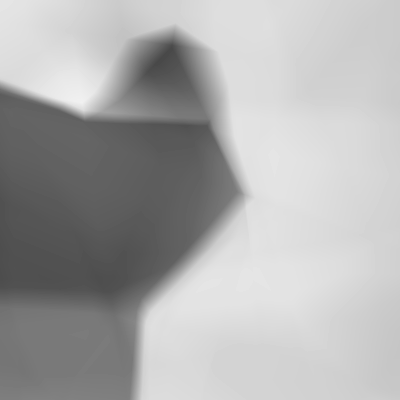}

\includegraphics[width=\imgwidth]{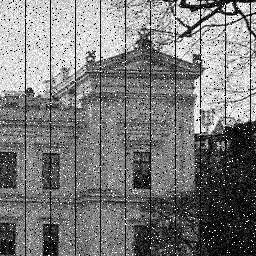}
\includegraphics[width=\imgwidth]{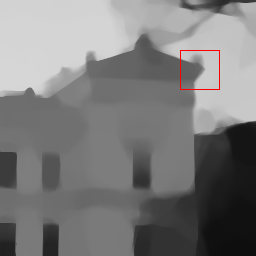}
\includegraphics[width=\imgwidth]{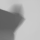}
\includegraphics[width=\imgwidth]{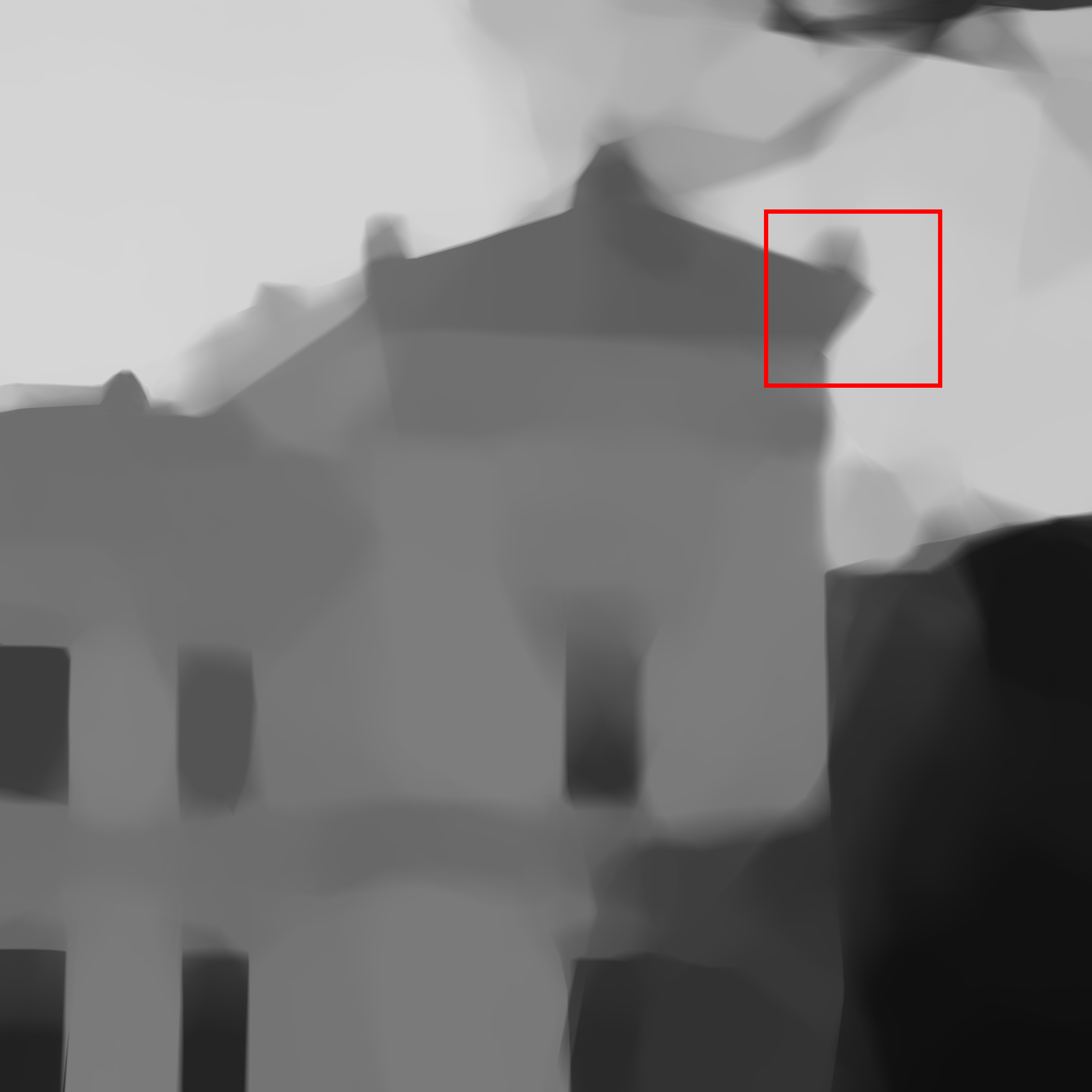}
\includegraphics[width=\imgwidth]{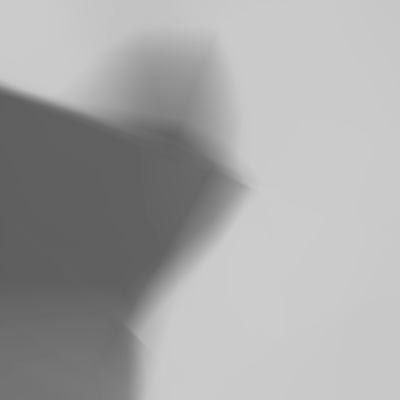}

\includegraphics[width=\imgwidth]{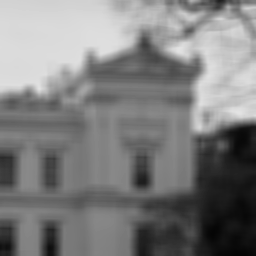}
\includegraphics[width=\imgwidth]{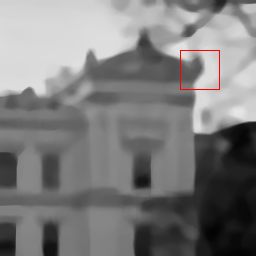}
\includegraphics[width=\imgwidth]{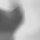}
\includegraphics[width=\imgwidth]{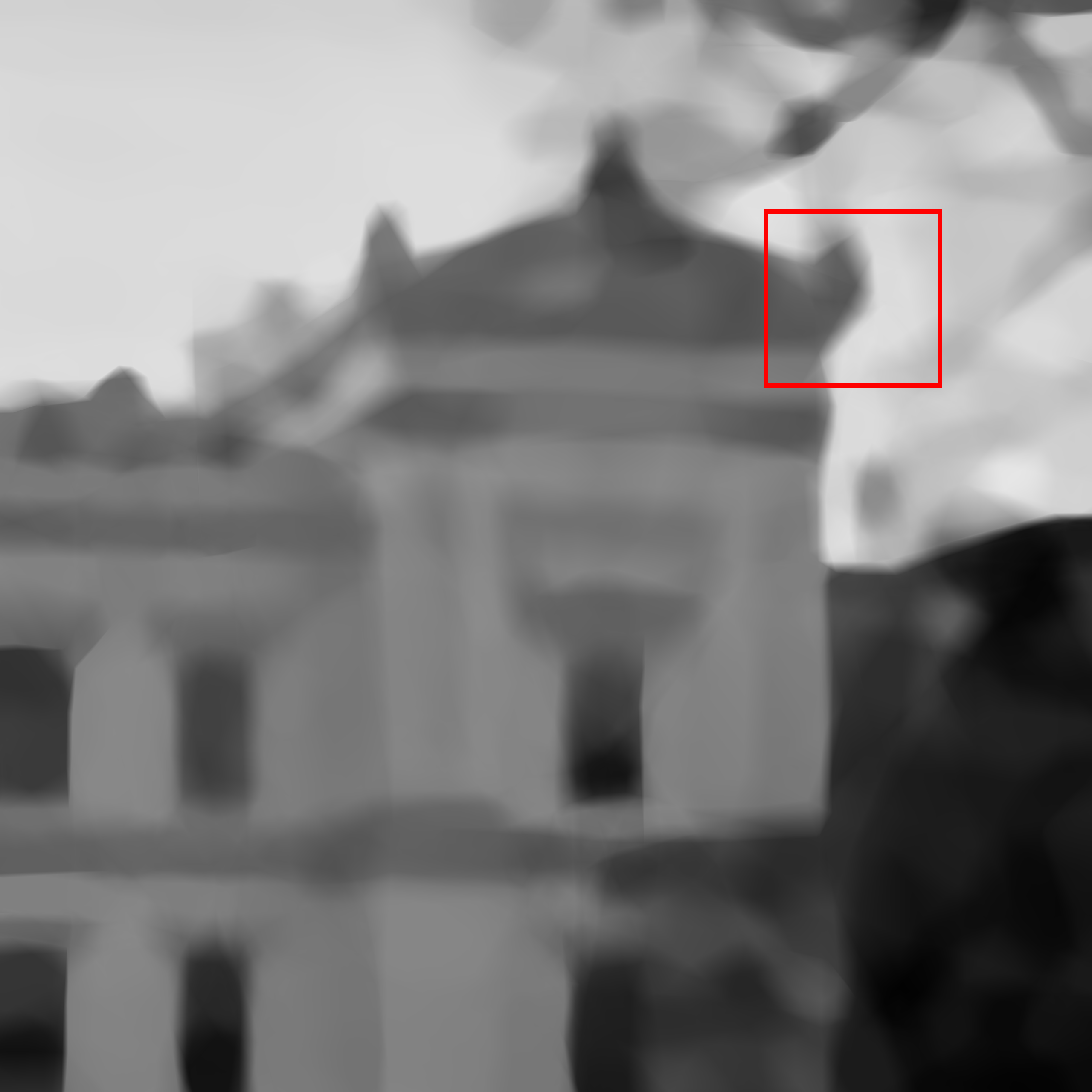}
\includegraphics[width=\imgwidth]{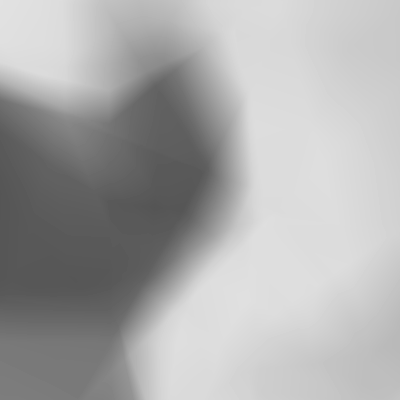}

\includegraphics[width=\imgwidth]{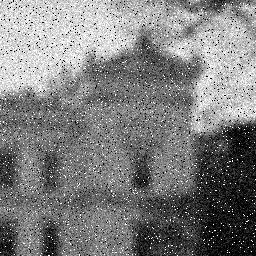}
\includegraphics[width=\imgwidth]{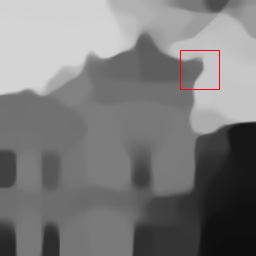}
\includegraphics[width=\imgwidth]{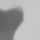}
\includegraphics[width=\imgwidth]{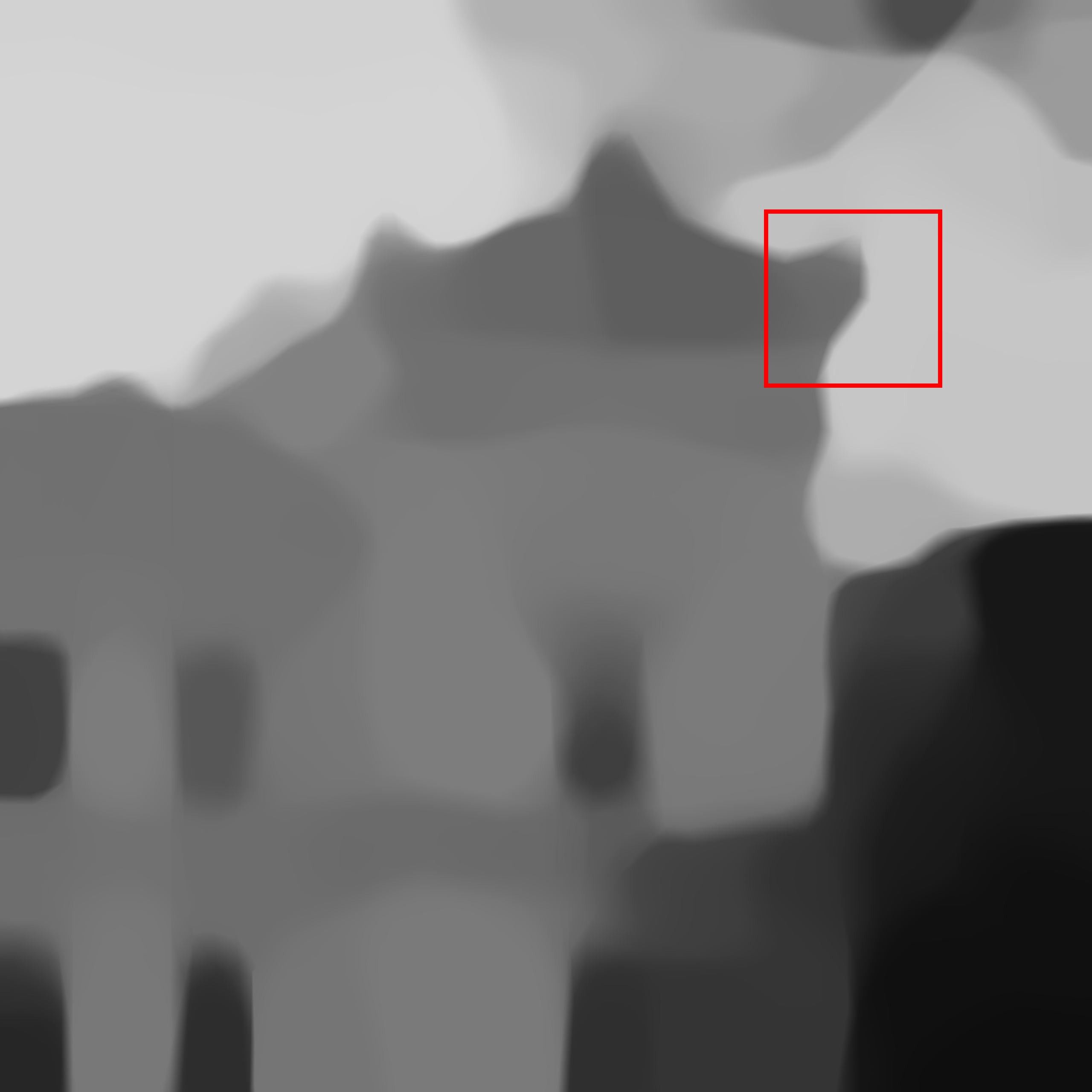}
\includegraphics[width=\imgwidth]{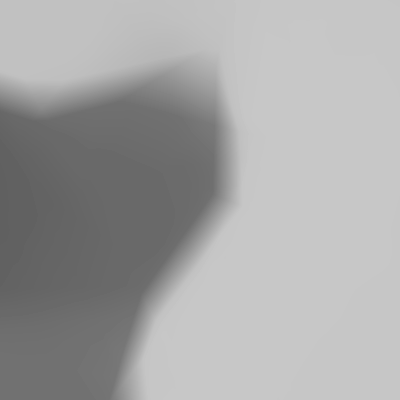}

\caption{Columns from left to right: Corrupted image ($256 \times 256$ pixels), reconstruction plotted with $256 \times 256$ pixels, zoomed reconstruction of highlighted area, reconstruction plotted on $10\times$ finer grid ($2560 \times 2560$ pixels), zoomed reconstruction of highlighted area of finer grid. Applications from top to bottom: denoising, inpainting, inpainting + denoising, deblurring, deblurring + denoising.}\label{Fig:Applications:Building}
\end{figure}

\section{Conclusion and discussion}\label{Sec:Conclusion}
We have seen that the considered neural network approach changes the nature of the original problem. In particular, while the $L^1$-$L^2$-TV model is convex, its neural network version \eqref{Eq:ProblemHw} is not convex anymore. Further we observe that even if the original problem has a solution the resulting neural network optimization problem does in general not, cf. \cref{Example:Counterexample}. Hence, it seems that the neural network approach renders the problem more complicated to deal with. Note that the non-existence of a solution can be easily resolved, as we did by introducing \eqref{Eq:ProblemHw:constrained}. However, the non-convexity seems to be persistent and we are not aware how to circumvent this property. Nevertheless, this additional complexity seems to come with a benefit. 
First, %the solution is mesh-free and continuous allowing to depict the solution on any grid and domain. That is 
while optimizing on a rather coarse grid, we actually obtain a solution in a continuum, i.e.\ the solution is mesh-free and continuous, allowing us to depict the solution on an arbitrary fine grid and any domain. 
Second, it seems that due to the nature of ReLU-NNs on a scale finer than the given data resolution, structures in the image are promoted which are attractive to humans. It is not clear to us why this is the case and further studies in this direction might be needed. 

In order to guarantee the existence of a (neural network) solution we introduced the constrained problem \eqref{Eq:ProblemHw:constrained}. In particular \eqref{Eq:ProblemHw:constrained} has a solution in $\mcH$ as long as $c< \infty$. 
However, the constraint $|\theta|_\infty \le c$ seems mainly theoretical. In particular, since our solution process relies on an iterative (stochastic) method,
numerically we do not expect to reach the analytic solution exactly anyway, but only approximate it. 
In the case that the solution space is $\mcH$, i.e.\ $c=\infty$, we would obtain an approximation in the space of piecewise affine functions which is dense in the space $BV$. 
Hence, solving the optimization problem over $\mcH$ seems numerically sufficient, as the numerical solution in $\mcH$ can arbitrarily closely approximate the exact $BV$ solution, assuming that $M$ is sufficiently large, and a stopping criterion, as the number of iterations, will always allow the numerical process to terminate with an approximation. 

Although we used in our experiments the number of iterations as a stopping criterion, it does not seem ideal, as we do not know how much progress has been already made in the solution process. Also our derived error estimate seems not to be applicable in a stopping criterion, due to its oscillatory behavior. Finding a suitable stopping criterion for non-smooth optimization problems is in general a difficult task in numerical analysis. As such problems do not allow for usual derivatives, it is difficult to derive numerically usable optimality conditions.

We conclude by emphasizing that the presented neural network approach can be applied to any optimization problem of the form
$$
\inf_{u\in V} F(u),
$$
where $V$ is a Banach space and $F:V \to \overline{\R}$, and any set of neural networks.

\paragraph*{Acknowledgement}
The work of A.L. was supported by the Crafoord Foundation
through the project ``Locally Adaptive Methods for Free Discontinuity Problems''.

\appendix

\section{Explicit solution of the $L^1$-$L^2$-TV problem for a simple 1D example}\label{AppendixA}
 Let $\Omega=(-l_\ell,l_u)$ with $l_\ell,l_u>0$ and $g$ be the function in \eqref{Eq:Example:g}. We compute the explicit solution of \eqref{Eq:ProblemTV} when $T=I$ is the identity operator and $\alpha_1,\alpha_2 > 0$, $\lambda=1$. One can reason that every minimizer $u$ has to be a step function of the form 
 \begin{equation*}
 u(x)=\begin{cases}
 c_1 &\text{if } x\in [0,l_u),\\
c_2 &\text{if } x\in (-l_\ell,0),
 \end{cases}
 \end{equation*}
 with $0\leq c_2 \leq c_1 \leq 1$. Therefore we just need to minimize
 \begin{equation*}
 \min_{c_1,c_2} \alpha_1  l_\ell |c_2 - 0| + \alpha_1 l_u |c_1 - 1| + \alpha_2 l_\ell |c_2-0|^2 + \alpha_2 l_u |c_1-1|^2 + |c_1 - c_2|.
 \end{equation*}
\begin{itemize}
\item Assume $0 < c_2 < c_1 < 1$: Then the optimality conditions for $c_1$ and $c_2$ are given by
\begin{align*}
-\alpha_1 l_u + 2 \alpha_2(c_1-1)l_u + 1 = 0\\
\alpha_1 l_\ell + 2 \alpha_2 c_2 l_\ell - 1 = 0
\end{align*}
and hence
\begin{equation}\label{Eq:c1c2}
c_2 = \frac{1 - \alpha_1 l_\ell}{2\alpha_2 l_\ell} \qquad \text{ and } \qquad 
c_1 = \frac{\alpha_1 l_u - \lambda + 2 \alpha_2 l_u}{2 \alpha_2 l_u} = 1 - \frac{ 1 - \alpha_1 l_u}{2 \alpha_2 l_u}.
\end{equation}
Since $c_2 < c_1$ this means
\begin{equation*}
\begin{split}
1 - \frac{ 1 - \alpha_1 l_u}{2 \alpha_2 l_u} > \frac{1 - \alpha_1 l_\ell}{2\alpha_2 l_\ell} \ \Longleftrightarrow \ \alpha_1 + \alpha_2 > \frac{l_u + l_\ell}{2l_ul_\ell}.
\end{split}
\end{equation*}
\item Assume $c_2 = c_1 =: C$: Then one can argue that $0<C<1$ and the optimality condition for $C$ is given by
\begin{align*}
\alpha_1 l_\ell-\alpha_1 l_u  + 2 \alpha_2 l_u C   + 2 \alpha_2 l_u (C-1)  = 0
\end{align*}
and hence
\begin{equation*}
C=\frac{2\alpha_2 l_u + \alpha_1(l_u-l_\ell)}{2\alpha_2 (l_\ell + l_u)}.
\end{equation*}
\end{itemize} 
From \eqref{Eq:c1c2} we observe that $c_2=0$ if $\alpha_1 \ge \frac{1}{l_\ell}$ and that $c_1 = 1$ if $\alpha_1 \ge \frac{1}{l_u}$.
Putting this and the above considerations together yields
\begin{equation*}
c_1 = \begin{cases}
1 & \text{if } \alpha_1 \ge \frac{1}{l_u},\\
1- \frac{1 - \alpha_1 l_u}{2\alpha_2 l_u} & \text{if } \frac{1}{l_u} >  \alpha_1  > \frac{l_u + l_\ell}{2l_ul_\ell} - \alpha_2,\\
\frac{2\alpha_2 l_u + \alpha_1(l_u-l_\ell)}{2\alpha_2 (l_\ell + l_u)} & \text{if } \frac{l_u + l_\ell}{2l_ul_\ell} - \alpha_2 \geq \alpha_1
\end{cases} 
\end{equation*}
and 
\begin{equation*}
 c_2 =\begin{cases}
0 & \text{if } \alpha_1 \ge \frac{1}{l_\ell},\\
\frac{1 - \alpha_1 l_\ell}{2\alpha_2 l_\ell} & \text{if } \frac{1}{l_\ell} >  \alpha_1  > \frac{l_u + l_\ell}{2l_ul_\ell} - \alpha_2,\\
\frac{2\alpha_2 l_u + \alpha_1(l_u-l_\ell)}{2\alpha_2 (l_\ell + l_u)} & \text{if } \frac{l_u + l_\ell}{2l_ul_\ell} - \alpha_2 \geq \alpha_1.
\end{cases}
\end{equation*}

\bibliographystyle{abbrv}
\bibliography{nmh}
\end{document}